\documentclass[11pt]{article}

\usepackage{natbib}
 \bibpunct[, ]{[}{]}{,}{n}{}{,}%

\usepackage{amsmath}
\usepackage{amssymb}
\usepackage{amsfonts}
\usepackage{mathrsfs}
\usepackage{graphicx}
\usepackage{color,xcolor}
\usepackage{float}
\usepackage{subfig}
\usepackage[normalem]{ulem}
\usepackage[linktocpage,colorlinks,linkcolor=blue,anchorcolor=blue, citecolor=blue,urlcolor=blue]{hyperref}
\usepackage{pdfpages}
\usepackage{enumitem}
\usepackage{multirow}
\usepackage{makecell}
\usepackage{breakcites}
\usepackage{bbm}
\usepackage{wrapfig}

\newtheorem{theorem}{Theorem}[section]
\newtheorem{lemma}[theorem]{Lemma}
\newtheorem{proposition}[theorem]{Proposition}
\newtheorem{corollary}[theorem]{Corollary}

\newtheorem{algorithm}[theorem]{Algorithm}

\newcommand{\red}[1]{\textcolor{black}{#1}}
\newcommand{\redmath}[1]{\mathcolor{black}{#1}}

\newcommand{\li}[1]{\red{#1}}
\newcommand{\limath}[1]{\redmath{#1}}

\def\bd{\boldsymbol}
\def\pn{\par\smallskip\noindent}

\newcommand{\BD}{\mathbb{D}}
\newcommand{\BE}{\mathbb{E}}
\newcommand{\BH}{\mathbb{H}}
\newcommand{\BX}{\mathbb{X}}
\newcommand{\BY}{\mathbb{Y}}
\newcommand{\BZ}{\mathbb{Z}}
\newcommand{\BS}{\mathbb{S}}
\newcommand{\BT}{\mathbb{T}}
\newcommand{\BN}{\mathbb{N}}
\newcommand{\BI}{\mathbb{I}}
\newcommand{\R}{\mathbb{R}}
\newcommand{\BK}{\mathbb{K}}
\newcommand{\BQ}{\mathbb{Q}}

\newcommand{\A}{\mathcal{A}}

\newcommand{\X}{\mathcal{X}}
\newcommand{\Y}{\mathcal{Y}}
\newcommand{\Z}{\mathcal{Z}}
\newcommand{\TT}{\mathcal{T}}
\newcommand{\U}{\mathcal{U}}
\newcommand{\V}{\mathcal{V}}

\newcommand{\be}{\boldsymbol{e}}
\newcommand{\bx}{\boldsymbol{x}}
\newcommand{\by}{\boldsymbol{y}}
\newcommand{\bz}{\boldsymbol{z}}
\newcommand{\bu}{\boldsymbol{u}}
\newcommand{\bv}{\boldsymbol{v}}
\newcommand{\bw}{\boldsymbol{w}}

\newcommand{\bt}{\boldsymbol{t}}

\newcommand{\bp}{\boldsymbol{p}}

\newcommand{\T}{\textnormal{T}}
\newcommand{\st}{\textnormal{s.t.}}

\newcommand{\Mat}{M}

\newcommand{\Prob}{\textnormal{Prob}\,}

\newcommand{\sign}{\textnormal{sign}\,}

\newcommand{\Diag}{D}

\usepackage{booktabs}
\usepackage{graphicx}
\usepackage{stmaryrd}
\usepackage{mathtools}
\usepackage{multicol}

\usepackage{algorithm}
\usepackage{algorithmic}

\makeatletter
\newcommand{\pushright}[1]{\ifmeasuring@#1\else\omit\hfill$\displaystyle#1$\fi\ignorespaces}
\newcommand{\pushleft}[1]{\ifmeasuring@#1\else\omit$\displaystyle#1$\hfill\fi\ignorespaces}
\makeatother

\usepackage{tikz}
\usetikzlibrary{calc}
\usepackage{pgfplots}
\usepackage{xxcolor}
\pgfplotsset{compat=1.16}
\usepgfplotslibrary{fillbetween}
\pgfdeclareradialshading[tikz@ball]{ball}{\pgfqpoint{0bp}{0bp}}{%
 color(0bp)=(tikz@ball!0!white);
 color(7bp)=(tikz@ball!0!white);
 color(15bp)=(tikz@ball!70!black);
 color(20bp)=(black!70);
 color(30bp)=(black!70)}
\makeatother

\tikzset{viewport/.style 2 args={
    x={({cos(-#1)*1cm},{sin(-#1)*sin(#2)*1cm})},
    y={({-sin(-#1)*1cm},{cos(-#1)*sin(#2)*1cm})},
    z={(0,{cos(#2)*1cm})}
}}

\pgfplotsset{only foreground/.style={
    restrict expr to domain={rawx*\CameraX + rawy*\CameraY + rawz*\CameraZ}{-0.05:100},
}}
\pgfplotsset{only background/.style={
    restrict expr to domain={rawx*\CameraX + rawy*\CameraY + rawz*\CameraZ}{-100:0.05}
}}

\def\addFGBGplot[#1]#2;{
    \addplot3[#1,only background, opacity=0.25] #2;
    \addplot3[#1,only foreground] #2;
}

\usetikzlibrary{3d}
\usetikzlibrary{calc}
\usetikzlibrary{babel} 

\pgfmathsetmacro\xx{1/sqrt(2)}
\pgfmathsetmacro\xy{1/sqrt(6)}
\pgfmathsetmacro\zy{sqrt(2/3)}

\begin{document}

\title{$\ell_p$-sphere covering and approximating nuclear $p$-norm}

\author{
Jiewen GUAN
\thanks{Research Institute for Interdisciplinary Sciences, School of Information Management and Engineering, Shanghai University of Finance and Economics, Shanghai 200433, China. Email: seemjwguan@gmail.com}
    \and
Simai HE
\thanks{Research Institute for Interdisciplinary Sciences, School of Information Management and Engineering, Shanghai University of Finance and Economics, Shanghai 200433, China. Email: simaihe@mail.shufe.edu.cn}
    \and
Bo JIANG
\thanks{Research Institute for Interdisciplinary Sciences, School of Information Management and Engineering, Shanghai University of Finance and Economics, Shanghai 200433, China. Email: isyebojiang@gmail.com}
    \and
Zhening LI
\thanks{School of Mathematics and Physics, University of Portsmouth, Portsmouth PO1 3HF, United Kingdom. Email: zheningli@gmail.com}
}

\date{\today}

\maketitle

\begin{abstract}
The spectral $p$-norm and nuclear $p$-norm of matrices and tensors appear in various applications albeit both are NP-hard to compute. The former sets a foundation of $\ell_p$-sphere constrained polynomial optimization problems and the latter has been found in many rank minimization problems in machine learning. We study approximation algorithms of the tensor nuclear $p$-norm with an aim to establish the approximation bound matching the best one of its dual norm, the tensor spectral $p$-norm. Driven by the application of sphere covering to approximate both tensor spectral and nuclear norms ($p=2$), we propose several types of hitting sets that approximately represent $\ell_p$-sphere with adjustable parameters for different levels of approximations and cardinalities, providing an independent toolbox for decision making on $\ell_p$-spheres. Using the idea in robust optimization and second-order cone programming, we obtain the first polynomial-time algorithm with an $\Omega(1)$-approximation bound for the computation of the matrix nuclear $p$-norm when $p\in(2,\infty)$ is a rational, paving a way for applications in modeling with the matrix nuclear $p$-norm. These two new results enable us to propose various polynomial-time approximation algorithms for the computation of the tensor nuclear $p$-norm using tensor partitions, convex optimization and duality theory, attaining the same approximation bound to the best one of the tensor spectral $p$-norm. \li{Effective performance of the proposed algorithms for the tensor nuclear $p$-norm are justified by numerical implementations.} We believe the ideas of $\ell_p$-sphere covering with its applications in approximating nuclear $p$-norm would be useful to tackle optimization problems on other sets such as the binary hypercube with its applications in graph theory and neural networks, the nonnegative sphere with its applications in copositive programming and nonnegative matrix factorization.

\vspace{0.25cm}

\noindent {\bf Keywords:} optimization on $\ell_p$-sphere, nuclear $p$-norm, semidefinite programming, $\ell_p$-sphere covering, approximation algorithm, spectral $p$-norm, polynomial optimization

\vspace{0.25cm}

\noindent {\bf Mathematics Subject Classification:} 15A60, 
90C59, 
52C17, 
68Q17 

\end{abstract}


\section{Introduction}\label{sec:introduction}

Decision making on spheres has been one of the important topics in mathematical optimization. It includes an extensive list of both theoretical and practical applications in areas of study such as biomedical engineering~\cite{BJVS07}, computational anatomy~\cite{FVJ09}, numerical multilinear algebra~\cite{Q05}, quantum mechanics~\cite{DLMO07}, solid mechanics~\cite{HDQ09}, signal processing~\cite{WV10}, tensor decompositions~\cite{KB09}, to name a few. Taking one of its simple models, $\min\{f(\bx):\|\bx\|_2=1,\,\bx\in\R^n\}$ where $f(\bx)$ is a multivariate polynomial function, this includes particular cases such as Cauchy-Schwarz inequality for a linear objective, extreme matrix eigenvalues for a quadratic form, trust region subproblem for a generic quadratic function~\cite{Y00}, as well as deciding the nonnegativity of a homogeneous polynomial~\cite{R00}. However, even in this simple form the problem can be very difficult since minimizing a cubic form on the unit sphere is already NP-hard~\cite{N03}. This fact has resulted a lot of research works~\cite{he2010approximation,so2011deterministic,zhangqy12,libook12,zhou12, he2014probability} on polynomial-time approximation algorithms for sphere constrained polynomial optimization in the recent decade. All these works essentially rely on a fundamental problem which is to maximize a multilinear form on Cartesian products of unit spheres. Taking a trilinear form
$\TT(\bx,\by,\bz) = \sum_{i=1}^{n}\sum_{j=1}^{n}\sum_{k=1}^{n} t_{ijk} x_iy_jz_k$ where $\TT=(t_{ijk})\in\R^{n\times n\times n}$ is a tensor of order three as a typical example, the optimization model is
\begin{equation}\label{eq:3tnorm}
\|\TT\|_{\sigma}= \max \left\{\TT(\bx,\by,\bz): \|\bx\|_2=\|\by\|_2=\|\bz\|_2=1,\,\bx,\by,\bz\in\R^n\right\},
\end{equation}
and the optimal value $\|\TT\|_{\sigma}$ is known as the spectral norm of $\TT$. \li{The spectral norm has found various applications, in particular, in analyzing the complexity of higher-order optimization algorithms; see, e.g.,~\cite{jiang2021optimal,lin2023monotone}.} The model~\eqref{eq:3tnorm} is NP-hard~\cite{he2010approximation}. The best-known approximation bound obtained by a polynomial-time algorithm is $\Omega(\sqrt{\ln n/n})$~\cite{so2011deterministic,he2014probability,he2023approx}.

The dual norm problem to~\eqref{eq:3tnorm} is
\begin{equation}\label{eq:3nnorm}
\|\TT\|_{*}= \max \left\{\left\langle\TT,\X\right\rangle: \|\X\|_\sigma\le 1,\,\X\in\R^{n\times n\times n}\right\},
\end{equation}
and the optimal value $\|\TT\|_*$ is known as the nuclear norm of $\TT$. The nuclear norm is much more widely used in practical applications as the convex surrogate of the rank function, no matter for matrices~\cite{RFP10} or tensors~\cite{yuan2016tensor}. The model~\eqref{eq:3nnorm} is also NP-hard~\cite{friedland2018nuclear} \li{and no computational methods have been proposed except~\cite{nie2017symmetric} for symmetric tensors}. It seems irrelevant to decision making on spheres but it is actually highly related. By carefully choosing vectors on the unit sphere and using \li{semidefinite programming (SDP)} with duality theory~\cite{hu2022complexity}, He et al.~\cite{he2023approx} obtained the best-known approximation bound $\Omega(\sqrt{\ln n/n})$ for~\eqref{eq:3nnorm}. The bound improved the previous best one $\Omega(\sqrt{1/n})$ and matched the best one $\Omega(\sqrt{\ln n/n})$ for~\eqref{eq:3tnorm}, hence bridging the gap between the primal and dual. The key idea of the algorithm in~\cite{he2023approx} is to explicitly construct a number of spherical caps to cover the unit sphere. It provides a useful tool and opens a new door to deal with decision making on spheres.

The ideas of constructions for sphere covering and the applications to~\eqref{eq:3tnorm} and~\eqref{eq:3nnorm} in~\cite{he2023approx} motivate us to study decision making on other important sets. An immediate one that generalizes the unit sphere is the $\ell_p$-sphere defined by $\|\bx\|_p=1$ for $p\in[1,\infty]$. It includes an important case of the unit box constraint when $p=\infty$. The example models~\eqref{eq:3tnorm} and~\eqref{eq:3nnorm} are then generalized to the spectral $p$-norm problem
\begin{equation}\label{eq:3tpnorm}
\|\TT\|_{p_\sigma}= \max \left\{\TT(\bx,\by,\bz): \|\bx\|_p=\|\by\|_p=\|\bz\|_p=1,\,\bx,\by,\bz\in\R^n\right\}
\end{equation}
and the nuclear $p$-norm problem
\begin{equation}\label{eq:3npnorm}
\|\TT\|_{p_*}=\max \left\{\left\langle\TT,\X\right\rangle: \|\X\|_{p_\sigma}\le 1,\,\X\in\R^{n\times n\times n}\right\},
\end{equation}
respectively. 
They are often found in many practical applications and in modeling various tensor optimization \li{problems}. In fact, many labeling problems in pattern recognition and image processing can be modeled by maximizing a certain polynomial function on the $\ell_p$-sphere~\cite{BBP98}. \li{Besides}, local optimal solutions of~\eqref{eq:3tpnorm} correspond to $\ell_p$-singular vectors of a tensor~\cite{lim2005singular} that have been extensively studied in the spectral theory of tensors and play an important role in signal processing, automatic control, data analysis, \li{as well as hypergraph theory~\cite{nikiforov2014analytic,keevash2014spectral}. Moreover, these norms naturally connect to the $\ell_p$-Grothendieck problem~\cite{kindler2010ugc}.} The nuclear $p$-norm~\eqref{eq:3npnorm} has found successful applications in temporal knowledge base completion~\cite{lacroix2018canonical, LacroixOU20}. In particular, under the nontemporal setting~\cite{lacroix2018canonical}, the nuclear $3$-norm is adopted as a regularizer for the binary tensor completion problem to simultaneously reduce the intrinsic complexity of the learned model and induce coefficient separability. This leads to great performance improvement and optimization convenience under stochastic gradient descent. Similarly, under the temporal setting~\cite{LacroixOU20}, the nuclear $3$-norm and $4$-norm are respectively adopted as two different regularizers for binary tensor completion and they both experimentally exhibit salient performance promotion over other comparative regularizers. \li{Apart from the important applications of the two norms across many disciplines when $p=2$ and $p=\infty$ (see, e.g.,~\cite[Section~1]{so2011deterministic},~\cite[Section~1]{he2023approx},~\cite[Section~1]{he2013approximation} and the references therein), it is also natural and necessary to study and understand their various behaviors during the smooth interpolation; i.e., $p\in(2,\infty)$.}

In fact, Hou and So~\cite{hou2014hardness} studied various $\ell_p$-sphere constrained polynomial optimization problems including~\eqref{eq:3tpnorm} as a typical example. In particular, they showed that~\eqref{eq:3tpnorm} is NP-hard when $p\in(2,\infty)$ and proposed a deterministic polynomial-time algorithm with an approximation bound $\Omega(\sqrt[p]{\ln n}/\sqrt{n})$, which can be improved to $\Omega(\sqrt{\ln n/n})$ with the help of randomization. These bounds remain currently the best when $p\in(2,\infty)$. The dual problem~\eqref{eq:3npnorm} is also NP-hard, evidenced by the complexity of duality~\cite{friedland2016computational}. However, to the best of our knowledge, the only known polynomial-time approximation bound of~\eqref{eq:3npnorm} in the literature is $\Omega(1/\sqrt[q]{n^2})$~\cite[Proposition 4.3]{chen2020tensor} where $\frac{1}{p}+\frac{1}{q}=1$. It was obtained by a simple partition of vectors and is much worse than that of~\eqref{eq:3tpnorm}. An improved approximation bound $\Omega(1/\sqrt[q]{n})$ of~\eqref{eq:3npnorm} can be obtained by \li{partitioning to matrices} or matrix unfolding in~\cite{chen2020tensor} but the methods rely on efficient computation of the matrix nuclear $p$-norm who itself is already NP-hard~\cite[Section~7]{friedland2018nuclear}. In fact, the bound $\Omega(1/\sqrt[q]{n})$ is made possible when $p\in(2,\infty)$ with the help of an $\Omega(1)$-approximation bound of the matrix nuclear $p$-norm developed in Section~\ref{sec:matrix-p-norm}. Well, this is still worse than that of~\eqref{eq:3tpnorm} when $p\in(2,\infty)$, no matter the deterministic one $\Omega (\sqrt[p]{\ln n}/\sqrt{n})$ or the randomized one $\Omega(\sqrt{\ln n/n})$. This naturally motivates the following question: Can the idea of sphere covering be applied to develop polynomial-time algorithms for~\eqref{eq:3npnorm} with approximation bounds matching currently the best \li{ones} for its dual problem~\eqref{eq:3tpnorm}?

In this paper, we answer the above question affirmatively by proposing easily implementable polynomial-time approximation algorithms for both the tensor spectral $p$-norm (the general version of~\eqref{eq:3tpnorm}) and the tensor nuclear $p$-norm (the general version of~\eqref{eq:3npnorm}) with the same best approximation bounds, in a short word, the $\ell_p$-sphere generalization of the $\ell_2$-sphere in~\cite{he2023approx}. These extensions are highly nontrivial and many challenges need to be tackled in order to achieve the goal. \li{As a reward, our study also leads to a wide array of technical tools and products that would be beneficial to the community. In particular, the explicit constructions of $\ell_p$-sphere covering already have some applications but will attract new applications to decision making problems over $\ell_p$-spheres.} In the following, we elaborate the difficulties conquered in this paper.

The first difficulty lies in the NP-hardness to compute the matrix nuclear $p$-norm that sets a foundation to the approximation of the tensor nuclear $p$-norm, in contrast to the matrix nuclear norm ($p=2$) that can be computed easily as the sum of its singular values. The research of approximating the matrix nuclear $p$-norm was almost blank although there are quite a few works on approximating the matrix spectral $p$-norm~\cite{N00,BN01,steinberg2005computation,hou2014hardness} \li{using conic optimization in the optimization community} 
albeit it is also NP-hard~\cite[Section~2.3]{steinberg2005computation}. To tackle this difficulty, we first apply a convex optimization model that approximates the matrix spectral $p$-norm within a constant factor~\cite[Proposition~3]{hou2014hardness} to the model~\eqref{eq:3npnorm} but it results the optimal value of the optimization model appearing in the constraints. We then apply the techniques developed in robust optimization~\cite{ben2009book,goldfarb2003robust} to compute the dual formulation of the convex optimization model that makes~\eqref{eq:3npnorm} explicit via a strong duality. In a final step, we reformulate some seemingly hard constrains to second-order cone (SOC) constraints. These enable us to (approximately) transfer~\eqref{eq:3npnorm} to a linear conic optimization model with semidefinite constraints and SOC constraints. As a result, the optimal value of the conic optimization model approximates the matrix nuclear $p$-norm within a constant factor when $p\in(2,\infty)$.

The second difficulty is that $\ell_p$-sphere covering cannot be easily extended from $\ell_2$-sphere covering as the former is not self-dual. Specifically, the goal is to find a hitting set $\BH^n$ consisting of a polynomial number of vectors on $\ell_p$-sphere in $\R^n$ such that $\max\{\bz^{\T}\bx:\bz\in\BH^n\}\ge\tau$ holds for any $\|\bx\|_q=1$. The $\tau$ is called the covering ratio that is aimed to be the largest possible. Therefore, covering $\ell_q$-sphere needs to analyze vectors on $\ell_p$-sphere. \li{The standard generalization from the $\ell_2$-sphere covering that enjoys the best covering ratio $\Omega(\sqrt{\ln n/n})$~\cite[Theorem 2.14]{he2023approx} can only achieve a covering ratio $\Omega(\sqrt[q]{\ln{n}/n})$ for $\ell_p$-sphere. Although this covering ratio is optimal for $p\in(1,2]$, it is much worse for $p\in(2,\infty)$, a required domain to apply the results to approximate the tensor spectral $p$-norm.
To further improve the covering ratio for $p\in(2,\infty)$, we explore an insightful idea from} computational geometry~\cite{brieden1998approximation} that an extended Hadamard transform (a linear transformation obtained by replacing $1$ with identity matrix in a Hadamard matrix)~\cite{syl1867,SY92} can enlarge $\ell_2$-norm while keeping $\ell_p$-norm unchanged for some vectors of our interest. This refinement sets a foundation to improve the covering ratio to $\Omega(\sqrt[p]{\ln{n}}/\sqrt{n})$ which matches the best deterministic approximation bound of~\eqref{eq:3tpnorm}. 
\li{Some interesting auxiliary results are generated and can be of independent interest.}
Finally, we construct a randomized hitting set of $\ell_p$-sphere whose covering ratio further improves to $\Omega(\sqrt{\ln{n}/n})$ which matches the best randomized approximation bound of~\eqref{eq:3tpnorm}. 
\li{Unlike the $\ell_2$-sphere where the tightness of Cauchy-Schwarz inequality holds for a pair of identical vectors, the tightness of H\"{o}lder's inequality holds only for a pair of vectors that are in different directions and can even be orthogonal, making the direct generalization from the result of $\ell_2$-sphere impossible. New techniques to refining the $\ell_p$-sphere need to be developed to fine tune the construction.}
As a result, we propose various constructions of hitting sets of $\ell_p$-sphere covering that are of independent interest and have direct applications to $\ell_p$-sphere decision making such as polynomial optimization on $\ell_p$-spheres.

With the aforementioned new results, we are able to put all the pieces together to derive various polynomial-time algorithms to approximate the tensor nuclear $p$-norm when $p\in(2,\infty)$. In particular, we first adopt the tensor partition technique~\cite{chen2020tensor} that helps to obtain an $\Omega(1/\sqrt[q]{n})$-approximation bound of~\eqref{eq:3npnorm}, an improvement from the only known one $\Omega(1/\sqrt[q]{n^2})$ in the literature. We then apply the hitting set relaxation approach that was originally proposed in~\cite{hu2022complexity} to improve the approximation bound of~\eqref{eq:3npnorm} to $\Omega(\sqrt[q]{\ln{n}/n})$ and $\Omega (\sqrt[p]{\ln n}/\sqrt{n})$ based on two different hitting sets of $\ell_p$-sphere. The latter bound matches the best approximation bound of~\eqref{eq:3tpnorm} by a deterministic algorithm. Finally, with the help of a randomized hitting set of $\ell_p$-sphere and a careful treatment of probability argument, we obtain the best approximation bound $\Omega(\sqrt{\ln n/n})$ of~\eqref{eq:3npnorm} that matches the best randomized one \li{of}~\eqref{eq:3tpnorm}, hence bridging the gap between the primal and dual perfectly.

We summarize our main contributions in the paper below.
\begin{itemize}
\item We obtain the first polynomial-time algorithm with $\Omega(1)$-approximation bound for the computation of the matrix nuclear $p$-norm when $p\in(2,\infty)$. This paves a way for applications in modeling with the matrix nuclear $p$-norm.
\item We propose several types of hitting sets of $\ell_p$-sphere with varying cardinalities and covering ratios. This provides an independent tool for decision making on $\ell_p$-spheres. 
\item We propose various polynomial-time algorithms for the computation of the tensor nuclear $p$-norm when $p\in(2,\infty)$ with the best approximation bound matching the best one for the tensor spectral $p$-norm.
\end{itemize}

The rest of the paper is organized as follows. We first introduce many uniform notations, formal definitions and their optimization models, as well as some basic properties \li{and connections} in Section~\ref{sec:notation}. We then study the computation of the matrix nuclear $p$-norm in Section~\ref{sec:matrix-p-norm} by proposing a convex optimization model whose optimal value approximates the matrix nuclear $p$-norm within a constant factor. The construction of various hitting sets of $\ell_p$-sphere covering is discussed in Section~\ref{sec:hitting-sets}. In Section~\ref{sec:algorithms}, we study the design and analysis of several deterministic and randomized polynomial-time algorithms to approximate the tensor nuclear $p$-norm \li{followed by some numerical results to justify the performance of the proposed methods.} Finally, we conclude this paper and propose some future research in Section~\ref{sec:conclusion}.

\section{Notations and preliminaries}\label{sec:notation}

\subsection{Uniform notations}

Throughout this paper we uniformly adopt lowercase letters (e.g., $x$), boldface lowercase letters (e.g., $\bx=(x_i)$), capital letters (e.g., $X=(x_{ij})$), and calligraphic letters (e.g., $\X=(x_{i_1i_2\dots i_d})$) to denote scalars, vectors, matrices, and higher-order (order three or more) tensors, respectively. Lowercase Greek letters are usually denoted for constants and $\delta$'s, specifically, for universal constants. Denote $\R^{n_1\times n_2\times\dots\times n_d}$ to be the space of real tensors of order $d$ with dimension $n_1\times n_2\times\dots\times n_d$. The same notation applies to a vector space and a matrix space when $d=1$ and $d=2$, respectively. Denote $\BN$ to be the set of positive integers and $\BQ$ to be the set of rationals. In particular, all blackboard bold capital letters denote sets, such as $\R^n$, the standard basis $\BE^n:=\{\be_1,\be_2,\dots,\be_n\}$ of $\R^n$, the $\ell_p$-sphere $\BS^n_p:=\{\bx\in\R^n:\|\bx\|_p=1\}$. The superscript $n$ of a set, specifically, indicates that the concerned set is a subset of $\R^n$. 

For a vector $\bx\in\R^n$, denote $\|\bx\|_p:=(\sum_{i=1}^n |x_i|^p)^{1/p}$ to be the $\ell_p$-norm for $p\in[1,\infty]$ and $|\bx|\in\R^n$ to be the vector taking its absolute values element-wisely. We denote $D(\bx)\in\R^{n\times n}$ to be the matrix whose diagonal vector is $\bx$ and off-diagonal entries are zeros. For a square matrix $X\in\R^{n\times n}$, we denote $D(X)\in\R^{n\times n}$ to be the matrix by keeping only the diagonal vector of $X$, i.e., replacing off-diagonal entries with zeros. The $n$-dimensional all-zero vector is denoted by $\bd{0}_n$ and all-one vector by $\bd{1}_n$. The $n\times n$ identity matrix is denoted by $I_n$ and a zero matrix is denoted by $O$. The subscript of these special vectors and matrices is often omitted as long as there is no ambiguity. For symmetric matrices $A,B\in\R^{n\times n}$, $A\succeq O$ indicates that $A$ is positive semidefinite and $A\succeq B$ indicates that $A-B\succeq O$. The Frobenius inner product between two tensors $\U,\V\in\R^{n_1\times n_2\times\dots\times n_d}$ is defined as
\[
\langle\U, \V\rangle := \sum_{i_1=1}^{n_1}\sum_{i_2=1}^{n_2} \dots\sum_{i_d=1}^{n_d} u_{i_1i_2\dots i_d} v_{i_1i_2\dots i_d}.
\]
Its induced Frobenius norm is naturally defined as $\|\TT\|:=\sqrt{\langle\TT,\TT\rangle}$. The two terms automatically apply to tensors of order two (matrices) and tensors of order one (vectors) as well. This is the conventional norm (a norm without a subscript) used throughout the paper.

Three vector operations are used frequently, namely the outer product $\otimes$, the Kronecker product $\boxtimes$, and appending vectors $\vee$. In particular, if $\bx\in\R^{n_1}$ and $\by\in\R^{n_2}$, then
\begin{align*}
\bx\otimes\by&=\bx\by^{\T}\in\R^{n_1\times n_2} \\
\bx\boxtimes\by&=(x_1\by^{\T},x_2\by^{\T},\dots,x_{n_1}\by^{\T})^{\T}\in\R^{n_1n_2} \\
\bx\vee\by&=(x_1,x_2,\dots,x_{n_1},y_1,y_2,\dots,y_{n_2})^{\T}\in\R^{n_1+n_2}.
\end{align*}
These three operators also apply to vector sets via element-wise operations. \li{We also denote $\bx^{\otimes d}:=\underbrace{\bx\otimes \bx\otimes\dots\otimes\bx}_d$.}

The notion $\Omega(f(n))$ means the same order of magnitude to $f(n)$, i.e., there exist positive constants $\alpha,\beta$ and $n_0$ such that $\alpha f(n)\le\Omega(f(n))\le\beta f(n)$ for all $n\ge n_0$. As a convention, the notation $O(f(n))$ means at most the same order of magnitude to $f(n)$.

\subsection{Spectral $p$-norm and nuclear $p$-norm}\label{sec:norm}

Given any tensor $\TT\in\R^{n_1 \times n_2 \times \dots \times n_d}$ and constant $p\in[1,\infty]$, the spectral $p$-norm~\cite{lim2005singular} of $\TT$ is defined as
\begin{equation}\label{def:spectral}
    \|\TT\|_{p_\sigma}:=\max \left\{\left\langle\TT, \bx_1 \otimes \bx_2 \otimes \dots \otimes \bx_d\right\rangle:\|\bx_k\|_p=1,\, k=1,2,\dots, d\right\}.
\end{equation}
This includes the matrix spectral $p$-norm for $d=2$ as its special case. In particular, the spectral $p$-norm reduces to the spectral norm when $p=2$.

A tensor $\TT\in\R^{n_1\times n_2\times\dots\times n_d}$ has $d$ modes, namely $1,2,\dots,d$. Fixing every mode index to a fixed value except the mode-$k$ index will result a vector in $\R^{n_k}$, called a mode-$k$ fiber. Fixing every mode index except two will result a matrix. In particular, fixing only one, say the mode-$k$ index to $i$ where $1\le i\le n_k$ will result a tensor of order $d-1$ in $\R^{n_1\times\dots \times n_{k-1}\times n_{k+1}\times\dots\times n_d}$. We call it the $i$th mode-$k$ slice, denoted by $\TT^{(k)}_i$. The mode-$k$ product of $\TT$ with a vector $\bx\in\R^{n_k}$ is denoted by
$$
\TT\times_k\bx:=\sum_{i=1}^{n_k}x_i\TT^{(k)}_i\in \R^{n_1\times\dots \times n_{k-1}\times n_{k+1}\times\dots\times n_d}.
$$
As a consequence, mode products with more vectors are obtained one by one, e.g.,
$$
\TT \times_1 \bx \times_2 \by =  (\TT \times_2 \by) \times_1 \bx = (\TT \times_1 \bx) \times_1 \by,
$$
where $\times_1 \by$ in the last equality is used instead of $\times_2 \by$ as mode $2$ of $\TT$ becomes mode $1$ of $\TT \times_1 \bx$. Mode products with $d-2$ vectors result a matrix, and with one more product result a vector. In \li{particular}, one has
$$\label{eq:multilinear}
\langle \TT, \bx_1 \otimes\bx_2\otimes \dots \otimes \bx_d \rangle
= \langle \TT \times_1 \bx_1, \bx_2 \otimes \dots \otimes \bx_d \rangle=
\dots 
=    \TT \times_1 \bx_1 \times_2 \bx_2 \dots \times_d \bx_d
$$
which is a multilinear form of $(\bx_1,\bx_2,\dots,\bx_d)$. By multilinearity, it means a linear form of $\bx_j$ by fixing all $\bx_k$'s but $\bx_j$ for every $j=1,2,\dots,d$. Therefore, the tensor spectral $p$-norm is to maximize a multilinear form over the Cartesian product of $\ell_p$-spheres. It was shown that the tensor spectral $p$-norm is equal to the largest $\ell_p$-singular value of the tensor~\cite[Proposition~1]{lim2005singular}. Mode product of a tensor with a vector on $\ell_p$-sphere will decrease the spectral $p$-norm in the weak sense.
\begin{lemma}\label{thm:contraction}
If $\TT\in\R^{n_1\times n_2\times\dots\times n_d}$, $p\in[1,\infty]$ and $\|\bx\|_p=1$, then $\|\TT\times_k \bx\|_{p_\sigma}\le\|\TT\|_{p_\sigma}$ for any $k$.
\end{lemma}
The proof can be easily obtained by comparing feasibility with optimality to the optimization model~\eqref{def:spectral} and noticing that
$\langle \TT, \bx_1 \otimes\bx_2\otimes \dots \otimes \bx_d \rangle
= \langle \TT \times_k \bx_k, \bx_1 \otimes \dots\otimes \bx_{k-1}\otimes \bx_{k+1} \otimes\dots\otimes \bx_d \rangle.$


The nuclear $p$-norm~\cite{friedland2018nuclear} of $\TT\in\R^{n_1\times n_2\times\dots\times n_d}$ is defined as
\begin{equation}\label{def:nuclear}
    \|\TT\|_{p_*}:=\min \left\{\sum_{i=1}^r\left|\lambda_i\right|: \TT=\sum_{i=1}^r \lambda_i\, \bx_i^1 \otimes \bx_i^2 \otimes \dots \otimes \bx_i^d,\, \|\bx_i^k\|_p=1 \text{ for all } k \text{ and } i,\,  r\in\BN \right\}.
\end{equation}
This also includes the matrix nuclear $p$-norm for $d=2$ and the nuclear norm for $p=2$ as its special cases. Similar to the well-known duality between the spectral norm and nuclear norm, the nuclear $p$-norm does be the dual norm of the spectral $p$-norm, and vise versa.
\begin{lemma}[{\cite[Lemma 2.5]{chen2020tensor}}]\label{lma:norm-duality}
    For any $\TT\in\R^{n_1\times n_2\times\dots\times n_d}$ and $p\in[1,\infty]$, it holds that 
    $$
    \|\TT\|_{p_\sigma} =\max\{\langle\TT, \mathcal{Z}\rangle: \|\Z\|_{p_*} \le 1\}   
    \text{ and }
    \|\TT\|_{p_*} =\max\{\langle\TT, \Z\rangle:\|\Z\|_{p_\sigma} \le 1\}.
    $$
\end{lemma}
\li{From this perspective, the nuclear $p$-norm can be taken as a bilevel optimization problem (see, e.g.,~\cite{nie2017bilevel,nie2021lagrange,nie2023plmes} for more details)
\begin{equation*}
    \|\TT\|_{p_*} =\max\left\{\left\langle\TT, \Z\right\rangle: \left\langle \Z, \bx_1 \otimes\bx_2\otimes \dots \otimes \bx_d \right\rangle -1\le 0,\,(\bx_1,\bx_2,\dots,\bx_d)\in\mathbb{O}(\Z)\right\},
\end{equation*}
where $\mathbb{O}(\Z)$ 
is the set of optimal solutions of $\max \left\{\left\langle\Z, \bx_1 \otimes \bx_2 \otimes \dots \otimes \bx_d\right\rangle:\|\bx_k\|_p=1,\, k=1,2,\dots,d\right\}$.
Similar reformulations can be obtained to the spectral $p$-norm as well as the approximation of the nuclear $p$-norm to be introduced in Lemma~\ref{prop:matrix-equi-pq-vecp}.}

For computational complexity of the spectral $p$-norm and nuclear $p$-norm, the beautiful duality result by Friedland and Lim~\cite{friedland2016computational} states that if one of them can be computed in polynomial time then the other also can, and if one of them is NP-hard to compute then the other is also NP-hard. For $d=1$, the vector case is trivial since the spectral $p$-norm becomes the $\ell_q$-norm where $\frac{1}{p}+\frac{1}{q}=1$ and nuclear $p$-norm becomes the $\ell_p$-norm. For $d=2$, the matrix spectral $p$-norm and nuclear $p$-norm can be computed in polynomial time when $p=1,2,\infty$ and are NP-hard when $p\in(2,\infty)$, while it remains unknown for the rest of $p$~\cite{steinberg2005computation}. However, for $d\ge3$, everything becomes NP-hard or unknown; see~\cite[Proposition~1]{hou2014hardness} and~\cite[Section~7]{friedland2018nuclear} for details. 

\subsection{Hitting set for $\ell_p$-sphere covering}\label{sec:lp-sphere-covering}

Throughout this paper, $p\in[1,\infty]$ is a given constant and $q\in[1,\infty]$ satisfies $\frac{1}{p}+\frac{1}{q}=1$. A set of vectors on $\ell_p$-sphere, $\BH^n=\{\bv_j \in \BS_p^n: j=1,2, \dots, m\}$, is called a hitting set of $\BS_p^n$ with hitting ratio $\tau\ge0$ and cardinality $m\in\BN$ if
$$
\bigcup_{j=1}^m \left\{\bx\in\BS_q^n:\bx^{\T}\bv_j\ge\tau\right\}=\BS_q^n.
$$
For simplicity, we usually call this $\BH^n$ a $\tau$-hitting set. Denote all the hitting sets of $\BS_p^n$ with hitting ratio at least $\tau$ and cardinality at most $m$ to be $\BT^n_p(\tau, m)$, i.e.,
$$
\BT^n_p(\tau, m):=\left\{\BH^n \subseteq \BS_p^n: \BH^n \text{ is a }  \tau \text {-hitting set and }|\BH^n|\le m\right\}.
$$
From the geometric point of view, a hitting set $\BH^n$ of $\BS_p^n$ is a representation of $\BS_p^n$ and approximately covers $\BS_q^n$, in the sense that any $\bx\in\BS_q^n$ is approximately covered by some $\bv_j\in\BH^n$ with $\bx^{\T}\bv_j\ge\tau$. By H\"{o}lder's inequality, $\bx^{\T}\bv_j\le1$ for $\bx\in\BS^n_p$ and $\bv_j\in\BS^n_q$ and so the best possible $\tau$ is obviously one when the hitting set is $\BS^n_p$ itself. The larger the $\tau$, the better coverage to $\BS_q^n$.

\li{Closely related to hitting sets are $\epsilon$-nets and $\epsilon$-covering numbers widely used in geometric functional analysis and statistical learning theory; see, e.g.,~\cite[Chapter~4]{artstein2015asymptotic} and~\cite[Section~4.2]{vershynin2018high}. Given $\epsilon\ge 0$, an $\epsilon$-net of a set is a subset for which any other point in the set has a distance at most $\epsilon$ to the subset and an $\epsilon$-covering number is the smallest possible cardinality of an $\epsilon$-net of the set. 
For $p=2$, the Euclidean sphere, due to the fact that $\|\bx-\by\|_2=\sqrt{2-2\,\bx^{\T}\by}$ for any $\bx,\by\in\BS^n_2$, the concepts of $\epsilon$-nets and hitting sets are equivalent: 
An $\epsilon$-net is the same as a $(1-\frac{\epsilon^2}{2})$-hitting set. 
However, this equivalence does not hold for $p\neq2$. As a simple example, it is easy to see that $\BE^n\cup(-\BE^n)$ is a $1$-hitting set of $\BS_1^n$ and attains the best possible hitting ratio, but it is only a $(1-\frac{1}{n})$-net of $\BS_1^n$ and is actually quite bad in this case.
Even though, we are not aware of any result concerning $\epsilon$-nets and $\epsilon$-covering numbers for general $\ell_p$-spheres. We remark that even in the Euclidean sphere, the study of $\epsilon$-nets in the literature always treats $\epsilon$ as a universal constant and they can only be transferred to hitting sets with a constant hitting ratio. This would result the cardinality of hitting sets being exponential in $n$~\cite[Theorem~3.15]{brieden2001deterministic} while our focus is on designing hitting sets with cardinality being polynomial in $n$.}

We end this section with the norm equivalence property between $\ell_p$-norms that is frequently used in the paper.
\begin{lemma}\label{lma:lp-norm-equiv}
 If $1\le r\le p\le\infty$, then $\|\bx\|_{p}\le\|\bx\|_{r}\le n^{1/r-1/p}\|\bx\|_{p}$ for any $\bx\in\R^n$.
\end{lemma}

\section{Approximating matrix nuclear $p$-norm}\label{sec:matrix-p-norm}

This section is devoted to a polynomial-time method that approximates the matrix nuclear $p$-norm within a factor of $1/\delta_G >\frac{2}{\pi}\ln(1+\sqrt{2}) > 0.561$ when $p\in\BQ\cap(2,\infty)$, where $\delta_G$ 
is the well-known Grothendieck constant~\cite{Grothendieck}.

\subsection{From matrix spectral $p$-norm to matrix nuclear $p$-norm}\label{sec:kG-matrix-nuclear-p-norm}

As mentioned in Section~\ref{sec:norm}, both the matrix spectral $p$-norm and nuclear $p$-norm are NP-hard to compute when $p\in(2,\infty)$. While the approximation of the matrix nuclear $p$-norm is almost blank in the literature, the approximation of the matrix spectral $p$-norm has been well studied under an equivalent concept called the $p\rightarrow q$ norm. In particular, via a convex optimization relaxation, Nesterov~\cite{N00} showed that the matrix spectral $p$-norm can be approximated within a factor of $\frac{2\sqrt{3}}{\pi}-\frac{2}{3}>0.435$. Ben-Tal and Nemirovski~\cite{BN01} and Steinberg~\cite{steinberg2005computation} gave a better analysis of Nesterov's relaxation method but their approximation bound is better only for certain $p$'s and certain dimensions of the matrix. Hou and So~\cite{hou2014hardness} provided the best approximation bound $1/\delta_G>0.561$ that uniformly beats all the above. We quote their result that is used in our analysis for the approximation of the nuclear $p$-norm.
\begin{lemma}[{\cite[Proposition~3]{hou2014hardness}}]\label{lma:vecp}
For any matrix $A\in\R^{m\times n}$ and constant $p\in\BQ\cap(2,\infty)$,
\begin{equation}\label{opt:vecp-original}
\|A\|_{p_v}:= \max \left\{\left\langle \begin{pmatrix}
O & A/2 \\
A^\T/2 & O
\end{pmatrix}, X\right\rangle: \sum_{i=1}^m\left|x_{i i}\right|^{p/2} \le 1,\, \sum_{i=m+1}^{m+n}\left|x_{ii}\right|^{p/2} \le 1,\, X \succeq O\right\}
\end{equation}
satisfies $\|A\|_{p_v}/\delta_G  \le \|A\|_{p_\sigma} \le \|A\|_{p_v}$, where $\delta_G$ is the Grothendieck constant.
\end{lemma}

The quantity $\|A\|_{p_v}$ defines a matrix norm that can be solved to arbitrary accuracy in polynomial time using, e.g., the ellipsoid method~\cite{grotschelbook}. In order to approximate the matrix nuclear $p$-norm using the duality in Lemma~\ref{lma:norm-duality}, $\|A\|_{p_*}=\max\{\langle A, Z\rangle:\|Z\|_{p_\sigma} \le 1\}$,
one is naturally suggested using $\|\cdot\|_{p_v}$ to replace $\|\cdot\|_{p_\sigma}$ that is NP-hard to compute. 
In fact, we have the following observation. 
\begin{lemma}\label{prop:matrix-equi-pq-vecp}
For any matrix $A\in\R^{m\times n}$ and constant $p \in\BQ\cap(2, \infty)$,
\begin{align}\label{eq:matrix-nuclear-pnorm-pqnorm}
    \|A\|_{p_*}=\max\left\{\langle A, Z\rangle:\|Z\|_{p_\sigma} \le 1\right\},
\end{align}
and
\begin{equation}\label{eq:matrix-nuclear-pnorm-vecp}
\|A\|_{p_u}:=\max\left\{\langle A,Z\rangle:\|Z\|_{p_v}\le 1\right\}
\end{equation}
are equivalent in the sense that $\|A\|_{p_u}\le \|A\|_{p_*}\le \delta_G \|A\|_{p_u}$.
\end{lemma}
\begin{proof}
   We see from Lemma~\ref{lma:vecp} that $\|Z\|_{p_v}\le 1$ implies $\|Z\|_{p_\sigma}\le 1$. Therefore,~\eqref{eq:matrix-nuclear-pnorm-pqnorm} is a relaxation of~\eqref{eq:matrix-nuclear-pnorm-vecp}, implying that $\|A\|_{p_u}\le \|A\|_{p_*}$.
   
   On the other hand, let $Y$ be an optimal solution of~\eqref{eq:matrix-nuclear-pnorm-pqnorm}, i.e., $\|A\|_{p_*}=\langle A,Y\rangle$ and $\|Y\|_{p_\sigma}\le1$. By Lemma~\ref{lma:vecp} again, one has 
   $$
   \|Y/\delta_G\|_{p_v}=\|Y\|_{p_v}/\delta_G\le  \|Y\|_{p_\sigma}\le1,
   $$
   implying that $Y/\delta_G$ is a feasible solution to~\eqref{eq:matrix-nuclear-pnorm-vecp}. Therefore, $\|A\|_{p_u}\ge \langle A,Y/\delta_G\rangle=\|A\|_{p_*}/\delta_G$, proving the upper bound.
\end{proof}

For the purpose of obtaining a constant approximation bound of the matrix nuclear $p$-norm, it suffices to study~\eqref{eq:matrix-nuclear-pnorm-vecp} because of Lemma~\ref{prop:matrix-equi-pq-vecp}. Reformulate~\eqref{eq:matrix-nuclear-pnorm-vecp} in a more explicit way, we have
\begin{equation}\label{eq:matrix-nuclear-pnorm-vecp-essence}
\|A\|_{p_u}=\max \left\{\left\langle A,Z\right\rangle: \left\langle \begin{pmatrix}
O & Z/2 \\
Z^\T/2 & O
\end{pmatrix}, X\right\rangle\le1 \text{ for all } X\in\mathbb{M}^{(m+n)\times (m+n)}\right\},
\end{equation}
where
$$
\mathbb{M}^{(m+n)\times (m+n)}:=\left\{X\in\R^{(m+n)\times(m+n)}: \sum_{i=1}^m\left|x_{i i}\right|^{p/2} \le 1,\, \sum_{i=m+1}^{m+n}\left|x_{ii}\right|^{p/2} \le 1,\, X \succeq O\right\}.
$$
The maximization formulation of $\|A\|_{p_v}$ in~\eqref{opt:vecp-original} makes~\eqref{eq:matrix-nuclear-pnorm-vecp-essence} difficult to handle because of an infinite number of constraints in $X\in\mathbb{M}^{(m+n)\times (m+n)}$. Inspired by a technique in robust optimization~\cite{ben2009book, goldfarb2003robust}, the constraint can be completely explicit if one can equivalently transform the maximization problem to a minimization one. This motivates us to study the dual problem of~\eqref{opt:vecp-original}.

\subsection{Equivalent formulations for $\|\cdot\|_{p_v}$ and $\|\cdot\|_{p_u}$}\label{sec:dual-vecp} 

We first derive an equivalent formulation of $\|\cdot\|_{p_v}$ via the dual problem of~\eqref{opt:vecp-original}. 

\begin{proposition}\label{lma:dual-vecp}
    The dual problem of~\eqref{opt:vecp-original} is
    \begin{equation}\label{opt:dual-vecp-finite}
    \begin{array}{lll}
    \|A\|_{p_v}= &\min & u_1+ u_2+\theta_p\sum_{i=1}^{m+n} t_i    \\
    &\st   & {v_i}^{p/(p-2)}\le t_i {u_1}^{2/(p-2)}\quad i=1,2,\dots, m \\
    && {v_i}^{p/(p-2)}\le t_i {u_2}^{2/(p-2)}\quad i=m+1,m+2,\dots, m+n\\ 
    &&  u_1\ge0, \,u_2\ge 0, \,   \bt\ge {\bf 0},\, \Diag(\bv)\succeq \begin{pmatrix} O & A/2 \\ A^\T/2 & O \end{pmatrix},
    \end{array}
    \end{equation}
    where $\theta_p:=(2/p)^{2/(p-2)}-(2/p)^{p/(p-2)}>0$ when $p\in(2,\infty)$.
\end{proposition}
\begin{proof}
Denote $B=\begin{pmatrix}O & A/2 \\A^\T/2 & O\end{pmatrix}$. The Lagrangian function associated with the maximization problem~\eqref{opt:vecp-original} is
\begin{equation*}
\begin{aligned}
    &~~~f(X, u_1, u_2, V)\\
    &=\langle B, X\rangle+ u_1 \left(1-\sum_{i=1}^m |x_{i i}|^{p/2}\right) +  u_2\left(1-\sum_{j=m+1}^{m+n}|x_{j j}|^{p/2}\right) + \left\langle  V, X\right\rangle
    \\&=u_1+u_2-\sum_{i=1}^m u_1|x_{i i}|^{p/2}-\sum_{j=m+1}^{m+n}u_2|x_{ii}|^{p/2}+ \left\langle B+ V, X\right\rangle
    \\&=u_1+u_2-\sum_{i=1}^m u_1|x_{i i}|^{p/2}+\sum_{i=1}^{m} v_{i i}x_{i i}-\sum_{j=m+1}^{m+n}u_2|x_{ii}|^{p/2}+ \sum_{i=m+1}^{m+n} v_{ii}x_{ii} + \left\langle B+ V-\Diag(V), X-\Diag(X)\right\rangle,
\end{aligned}
\end{equation*}
where $V\succeq O$ is the multiplier associated with the constraint $X\succeq O$ and $u_1\ge0$ and $u_2\ge 0$ are the multipliers associated with the constraints $\sum_{i=1}^m |x_{i i}|^{p/2}\le 1$ and $\sum_{i=m+1}^{m+n}|x_{i}|^{p/2}\le 1$, respectively. The last equality holds because $\Diag(B)=O$. It is purposely written as the sum of two parts with the first part involving only the diagonal entries of $X$ while the second part (the inner product) involving only the off-diagonal entries of $X$.

Let us analyze $\max\left\{f(X, u_1, u_2, V):X\in\R^{(m+n)\times(m+n)}\right\}$ based on the two parts in the last equality. First we must have $B+V-\Diag(V)=O$ since $X$ is a free matrix variable and any off-diagonal entry of $X$ appears only in $\left\langle B+ V-\Diag(V), X-\Diag(X)\right\rangle$. It remains to deal with the part involving the diagonal entries of $X$. In fact, this part enjoys a separable structure for all $x_{ii}$'s, each of which appears in a same type of subproblem $\max\left\{-u |x|^{p/2}+vx:x\in\R\right\}$ with $u,v\ge0$. This is because of $u_1\ge0, \,u_2\ge 0$ and $V\succeq O$.

Let us consider the following cases for $\max\left\{-u |x|^{p/2}+vx:x\in\R\right\}$ with $u,v\ge0$.
\begin{itemize}
    \item If $v=0$ and $u\ge0$, then $\max\left\{-u|x|^{p/2}+vx:x\in\R\right\}=0$.
    \item If $v>0$ and $u=0$, then $\max\left\{-u |x|^{p/2}+vx:x\in\R\right\}=+\infty$.
    \item If $v>0$ and $u>0$, by noticing that an optimal $x$ must be nonnegative and $p\in(2,\infty)$, it can be calculated that 
    $$
    \max\left\{-u |x|^{p/2}+vx:x\in\R\right\}=\left(\left(\frac{2}{p}\right)^{2/(p-2)}-\left(\frac{2}{p}\right)^{p/(p-2)}\right)\frac{v^{p/(p-2)}}{u^{2/(p-2)}}=\frac{\theta_pv^{p/(p-2)}}{u^{2/(p-2)}}.
    $$
\end{itemize}
To summarize, we have
$$
\max\left\{-u |x|^{p/2}+vx:x\in\R\right\}=\theta_pg_p(u,v),
$$
where
\begin{align*}
     g_p(u,v)&=\begin{cases}
    \frac{ v^{p/(p-2)}}{ u^{2/(p-2)}} &u> 0 \\
    0 & u=v=0 \\
    +\infty & u=0, v>0.
    \end{cases}
\end{align*}

With the above calculation, we obtain the following dual formulation of~\eqref{opt:vecp-original} by minimizing $\max\left\{f(X, u_1, u_2, V):X\in\R^{(m+n)\times(m+n)}\right\}$ over $u_1,u_2\ge0$ and $V\succeq O$
    \begin{equation} \label{eq:dualproof}
    \begin{array}{ll}
  \min &u_1+ u_2+\theta_p\sum_{i=1}^m g_p(u_1, v_{ii})+\theta_p\sum_{i=m+1}^{m+n} g_p(u_2, v_{ii})\\
 \st
    &  B+ V-\Diag( V)=O \\
    &  u_1\ge0, \,u_2\ge 0,\,  V\succeq O.
    \end{array}
    \end{equation}

It remains a couple of treatments in order to make the above formulation simpler and explicit. First, the off-diagonal part of $V$ must be $-B$ since $\Diag(B)=O$ and $B+V-\Diag(V)=O$ in the constraint. $V\succeq O$ then becomes $\Diag(V)-B\succeq O$ and this makes the off-diagonal entries of $V$ disappeared completely in~\eqref{eq:dualproof}. We use a new notation $v_i$ to replace $v_{ii}$ for $i=1,2,\dots,m+n$ and denote $\bv\in\R^{m+n}$ be the vector consisting of these $v_i$'s. Therefore, $B+ V-\Diag( V)=O$ and $V\succeq O$ are equivalent to $\Diag(\bv)-B\succeq O$. 
Next, as $\theta_p>0$ and~\eqref{eq:dualproof} is a minimization problem, we may let $g_p(u_1,v_{i})=t_i$ for $i=1,2,\dots,m$ and $g_p(u_2,v_{i})=t_i$ for $i=m+1,m+2,\dots,m+n$ in the objective function and add additional constraints $g_p(u_j,v_{i})\le t_i$ for corresponding $j$ and $i$. An important observation is that 
$$
\min\left\{ t: g_p(u,v)\le t \right\} = \min\left\{ t: v^{p/(p-2)} \le t u^{2/(p-2)},\,t\ge0 \right\} \text{ when }u,v\ge0.
$$
With these two observations,~\eqref{eq:dualproof} can be equivalently formulated to
    \begin{equation*} 
    \begin{array}{ll}
  \min &u_1+ u_2+\theta_p\sum_{i=1}^{m+n} t_i\\
    \st   & {v_i}^{p/(p-2)}\le t_i {u_1}^{2/(p-2)}\quad i=1,2,\dots, m \\
    & {v_i}^{p/(p-2)}\le t_i {u_2}^{2/(p-2)}\quad i=m+1,m+2,\dots, m+n\\ 
    &  u_1\ge0, \,u_2\ge 0,\, \bt\ge{\bf 0},\, \Diag(\bv)-B\succeq O.
    \end{array}
    \end{equation*}

Finally, to show that the objective value of the above optimization problem is equal to $\|A\|_{p_v}$, we only need to show the strong duality. In fact, by letting $Y=\frac{1}{2}\begin{pmatrix}m^{-2/p}I_m & O \\
O & n^{-2/p}I_n
\end{pmatrix}$, one has $\sum_{i=1}^m |y_{i i}|^{p/2} = \sum_{i=m+1}^{m+n}|y_{ii}|^{p/2}= (\frac{1}{2})^{p/2} < 1$ and $Y\succ O$, i.e., $Y$ is strictly feasible to~\eqref{opt:vecp-original}. This means that the Slater's condition holds for the convex optimization problem~\eqref{opt:vecp-original} and hence the strong duality holds~\cite[Section~5.2.3]{boyd2004convex}.
\end{proof}



Armed with Proposition~\ref{lma:dual-vecp}, the infinite number of constraints in~\eqref{eq:matrix-nuclear-pnorm-vecp-essence} hold if and only if there exists one feasible solution in~\eqref{eq:dualproof} satisfying this constraint. As a result,
we are able to 
equivalently reformulate $\|A\|_{p_u}$, i.e., either~\eqref{eq:matrix-nuclear-pnorm-vecp} or~\eqref{eq:matrix-nuclear-pnorm-vecp-essence}, to a completely explicit optimization problem.
\begin{corollary}\label{cor:pnorm-simpler}
For any matrix $A\in\R^{m\times n}$ and constant $p \in\BQ\cap(2, \infty)$,
 \begin{equation}\label{eq:matrix-nuclear-pnorm-vecp-dual}
    \begin{array}{lll}
    \|A\|_{p_u}= &\max & \langle A,Z\rangle    \\
    &\st   &u_1+ u_2+\theta_p\sum_{i=1}^{m+n} t_i\le 1\\
    && {v_i}^{p/(p-2)}\le t_i {u_1}^{2/(p-2)}\quad i=1,2,\dots, m \\
    && {v_i}^{p/(p-2)}\le t_i {u_2}^{2/(p-2)}\quad i=m+1,m+2,\dots, m+n\\ 
    &&  u_1\ge0, \,u_2\ge 0, \,   \bt\ge {\bf 0},\, \Diag(\bv)\succeq \begin{pmatrix}
O & Z/2 \\
Z^\T/2 & O
\end{pmatrix}.
    \end{array}
    \end{equation}
\end{corollary}

Problem~\eqref{eq:matrix-nuclear-pnorm-vecp-dual} greatly relieves the unpleasant constraint $X\in\mathbb{M}^{(m+n)\times (m+n)}$ in~\eqref{eq:matrix-nuclear-pnorm-vecp-essence}. However, it is still not perfect because of a series of nonlinear constraints of the type ${v}^{p/(p-2)}\le t {u}^{2/(p-2)}$ albeit the remaining parts are either linear constraint or semidefinite constraint. In fact, $\{(u,v,t)\in\R_+^3:v^{p/(p-2)}\le t u^{2/(p-2)}\}$ admits an SOC representation, shown in the next subsection.

\subsection{Conic optimization model for $\|\cdot\|_{p_u}$}\label{sec:SDP-relax-nuclear-p-norm}

Let us start to deal with the constraint $v^{p/(p-2)}\le t u^{2/(p-2)}$. We first quote a result from a book by Ben-Tal and Nemirovski~\cite{ben2001lectures} together with a detailed construction. 
\begin{lemma}[{\cite[Example~3.11]{ben2001lectures}}]\label{lma:P-SOC} For any $k\in\BN$, the set
$$
\left\{\bx \in \R^{2^k}_+,\,y\in\R_+ : 
y \le\left(\prod_{i=1}^{2^k}x_i\right)^{2^{-k}}\right\}
$$
is SOC representable. In particular, when $k\ge2$, by defining additional variables $\bz^i\in\R^{2^i}_+$ for $i=1,2,\dots,k-1$, this set is the projection to the first $2^k+1$ entries of the following set defined by
\begin{equation} \label{eq:SOCA1P}
\begin{cases}
\left\|\left(
    y, \frac{1}{2}(z^1_1-z^1_2)
\right)\right\|_2\le\frac{1}{2}(z^1_1+z^1_2)\\
\left\|\left(
    z^{i}_{j}, \frac{1}{2}(z^{i+1}_{2j-1}-z^{i+1}_{2j})
\right)\right\|_2\le\frac{1}{2}(z^{i+1}_{2j-1}+z^{i+1}_{2j})& i=1,2,\dots,k-2,\, j=1,2,\dots 2^{i}\\
\left\|\left(
    z^{k-1}_{j}, \frac{1}{2}(x_{2j-1}-x_{2j})
\right)\right\|_2\le\frac{1}{2}(x_{2j-1}+x_{2j})& j=1,2,\dots, 2^{k-1}\\
\bx\ge{\bf 0},\,y\ge0,\,\bz^i\ge{\bf 0} & i=1,2,\dots, k-1.
\end{cases}
\end{equation}
\end{lemma}

We now provide an SOC representation of $\{(u,v,t)\in\R_+^3:v^{p/(p-2)}\le t u^{2/(p-2)}\}$.
\begin{proposition}\label{prop:SOC-A-1P}
For any constant $p\in\BQ\cap(2,\infty)$, denote $1/p=a/b$ with $a,b\in\BN$ being mutually prime and let $k=\lceil\log_2 (b+1)\rceil$. It holds that
\begin{align}
\BK^3_p:=&\{(u,v,t)\in\R_+^3:v^{p/(p-2)}\le t u^{2/(p-2)}\}  \nonumber \\
=&\{(u,v,t)\in\R^3: (u,v,t,\bz^1,\bz^2,\dots,\bz^{k-1}) \text{ satisfies~\eqref{eq:SOC-A-1P}},\,\bz^i\in\R^{2^i} \text{ for } i=1,2,\dots,k-1 \}, \label{eq:J3p}
\end{align}
where
\begin{equation} \label{eq:SOC-A-1P}
\begin{cases}
\left\|\left(
    v, \frac{1}{2}(z^1_1-z^1_2)
\right)\right\|_2\le\frac{1}{2}(z^1_1+z^1_2)\\
\left\|\left(
    z^{i}_{j}, \frac{1}{2}(z^{i+1}_{2j-1}-z^{i+1}_{2j})
\right)\right\|_2\le\frac{1}{2}(z^{i+1}_{2j-1}+z^{i+1}_{2j})& i=1,2,\dots,k-2,\, j=1,2,\dots 2^{i}\\
z^{k-1}_{j}\le u& j=1,2,\dots, a\\
z^{k-1}_{j}\le t& j=a+1,a+2,\dots, \lfloor\frac{b}{2}\rfloor\\
\left\|\left(
    z^{k-1}_{j}, \frac{1}{2}(t-v)
\right)\right\|_2\le\frac{1}{2}(t+v)& j=\lfloor\frac{b}{2}\rfloor+1,\lfloor\frac{b}{2}\rfloor+2,\dots,\lceil\frac{b}{2}\rceil\\
z^{k-1}_{j}\le v& j=\lceil\frac{b}{2}\rceil+1,\lceil\frac{b}{2}\rceil+2,\dots, 2^{k-1}\\
\bz^i\ge{\bf 0} & i=1,2,\dots, k-1\\
u\ge 0,\,v\ge0,\,t\ge0.
\end{cases}
\end{equation}
\end{proposition}
\begin{proof}
Consider a linear mapping $f_{p}:\R^3\rightarrow\R^{2^k+1}$ where
\begin{align*}
f_p(u,v,t) =(\underbrace{ u, u,\dots, u}_{2a}, \underbrace{t,t, \dots, t}_{b-2a}, \underbrace{v,  v, \dots,  v}_{2^k-b}, v).
\end{align*}
Noticing that $2a>0$, $b-2a>0$ as $p\in(2,\infty)$ and $2^k-b>0$, it can be calculated that
$$ v^{p/(p-2)}\le t u^{2/(p-2)}  \Longleftrightarrow v\le (u^{2a}t^{b-2a}v^{2^k-b})^{2^{-k}}  \text{ when } u,v,t\ge0.$$

If $(u,v,t)\in\BK^3_p$, then $f_p(u,v,t)\in \{\bx \in \R^{2^k}_+,\, y\in\R_+ : 
y \le(\prod_{i=1}^{2^k}x_i)^{2^{-k}}\}$. It is straightforward to verify that~\eqref{eq:SOCA1P} becomes exactly~\eqref{eq:SOC-A-1P} for $f_p(u,v,t)$. By applying Lemma~\ref{lma:P-SOC}, there must exist $\bz^i\in\R^{2^i}_+$ for $i=1,2,\dots,k-1$ such that $(u,v,t,\bz^1,\bz^2,\dots,\bz^{k-1}) \text{ satisfies~\eqref{eq:SOC-A-1P}}$.

On the other hand, if $(u,v,t)$ belongs to the \li{right hand side} of~\eqref{eq:J3p}, then $(u,v,t,\bz^1,\bz^2,\dots,\bz^{k-1})$ satisfies~\eqref{eq:SOC-A-1P}. Noticing that $\|(w, \frac{1}{2}(x-y))\|_2\le\frac{1}{2}(x+y)$ is equivalent to $w\le\sqrt{xy}$ when $w,x,y\ge0$, it is easy to verify recursively from~\eqref{eq:SOC-A-1P} that
$$
v\le \left(\prod_{j=1}^2z^1_j\right)^{2^{-1}} \le \left(\prod_{j=1}^4z^2_j\right)^{2^{-2}} \le \dots \le \left(\prod_{j=1}^{2^{k-1}}z^{k-1}_{j}\right)^{2^{-(k-1)}}\le (u^{2a}t^{b-2a}v^{2^k-b})^{2^{-k}}.
$$
Therefore, $v^{p/(p-2)}\le t u^{2/(p-2)}$, implying that $(u,v,t)\in\BK^3_p$.
\end{proof}

With the SOC representation of $\BK^3_p$ in Proposition~\ref{prop:SOC-A-1P} in an explicit way, we provide semidefinite optimization models for both $\|A\|_{p_v}$ and $\|A\|_{p_u}$.
\begin{corollary}\label{lma:dual-SDP}
For any matrix $A\in\R^{m\times n}$ and constant $p\in\BQ\cap(2,\infty)$,
    \begin{equation}\label{cor:final-spectral-p-norm}
    \begin{array}{lll}
    \|A\|_{p_v}= &\min & u_1+ u_2+\theta_p\sum_{i=1}^{m+n} t_i    \\
    &\st   & (v_i, u_1,t_i)\in\BK^3_p\quad i=1,2,\dots, m \\
    && (v_i, u_2,t_i)\in\BK^3_p\quad i=m+1,m+2,\dots, m+n\\
    && 
    \Diag(\bv)\succeq \begin{pmatrix} O & A/2 \\A^\T/2 & O\end{pmatrix}
    \end{array}
    \end{equation}
and
\begin{equation}\label{cor:final-nuclear-p-norm}
    \begin{array}{lll}
    \|A\|_{p_u}= &\max & \langle A,Z\rangle    \\
    &\st   &u_1+ u_2+\theta_p\sum_{i=1}^{m+n} t_i\le 1\\
    && (v_i, u_1,t_i)\in\BK^3_p\quad i=1,2,\dots, m \\
    && (v_i, u_2,t_i)\in\BK^3_p\quad i=m+1,m+2,\dots, m+n\\ 
    && 
    \Diag(\bv)\succeq \begin{pmatrix} O & Z/2 \\Z^\T/2 & O\end{pmatrix},
    \end{array}
    \end{equation}
where $\BK^3_p$ is defined in~\eqref{eq:J3p}.
\end{corollary}

As the given constant $\frac{1}{p}=\frac{a}{b}$ has $a,b\in\BN$ mutually prime and $2a<b$, the number of variables and the number of constraints describing $\BK^3_p$ in~\eqref{eq:J3p} are $O(b)$ and $O(b\ln b)$, respectively. The number of variables in~\eqref{cor:final-nuclear-p-norm} is $O(bm+bn+mn)$ while there are $O((m+n)b\ln b)$ linear or simple SOC constraints together with one positive semidefinite constraint. Anyway,~\eqref{cor:final-nuclear-p-norm} can be solved to arbitrary accuracy in polynomial time using, e.g., an interior-point method. In fact,~\eqref{cor:final-spectral-p-norm} provides a much better formulation of $\|\cdot\|_{p_v}$ than~\eqref{opt:vecp-original} \li{which needs to be solved by the impractical ellipsoid method}. By combining Lemma~\ref{lma:vecp}, Lemma~\ref{prop:matrix-equi-pq-vecp} and Corollary~\ref{lma:dual-SDP}, we conclude this section with the following important result in matrix theory.
\begin{theorem}\label{thm:KG-matrix-nuclear-pnorm}
For any matrix $A\in\R^{m\times n}$ and constant $p \in\BQ\cap(2, \infty)$, $\|A\|_{p_u}$ can be obtained by the optimal value of~\eqref{cor:final-nuclear-p-norm} such that
$$
\|A\|_{p_u}\le \|A\|_{p_*}\le {\delta_G}\|A\|_{p_u}.
$$
Moreover, any optimal solution $Z\in\R^{m\times n}$ of~\eqref{cor:final-nuclear-p-norm} is an approximate dual certificate of $\|A\|_{p_*}$ in the sense that
$$
\|Z\|_{p_\sigma}\le 1 \text{ and } \langle A,Z\rangle\ge \|A\|_{p_*}/\delta_G.
$$
\end{theorem}
\begin{proof}
The first statement is an immediate consequence of Lemma~\ref{prop:matrix-equi-pq-vecp} and the second part of Corollary~\ref{lma:dual-SDP}. For the second statement, an optimal solution of~\eqref{cor:final-nuclear-p-norm} exists since the primal problem~\eqref{opt:vecp-original} is bounded and strictly feasible, mentioned at the end of the proof of Proposition~\ref{lma:dual-vecp}. According to~\eqref{cor:final-spectral-p-norm}, any optimal $Z$ of~\eqref{cor:final-nuclear-p-norm} satisfies $\|Z\|_{p_v}\le1$. Therefore, $\|Z\|_{p_\sigma}\le\|Z\|_{p_v}\le1$ by Lemma~\ref{lma:vecp} and $ \langle A,Z\rangle=\|A\|_{p_u}\ge \|A\|_{p_*}/\delta_G$.
\end{proof}

\section{$\ell_p$-sphere covering}\label{sec:hitting-sets}

This section is devoted to explicit constructions of hitting sets to approximately cover $\BS^n_p$ in $\R^n$. In particular, we focus on hitting sets that are aimed to maximize the hitting ratio while keeping the cardinality bounded by a polynomial function of $n$, called polynomial cardinality. This is important to the approximation schemes of the tensor nuclear $p$-norm to be run in polynomial time in Section~\ref{sec:algorithms}. Some other hitting sets constructed as byproducts have independent interest. \li{All these hitting sets are beneficial to applications in decision making on $\ell_p$-spheres.}

When $p=2$, the Euclidean sphere covering is well studied in computational geometry since the pioneering work of Rogers~\cite{R58}; see~\cite{he2023approx} and \li{the} references therein. In particular, several hitting sets with the hitting ratio $\Omega(\sqrt{\ln{n}/{n}})$ were explicitly constructed in~\cite{he2023approx}. However, little was known for $\ell_p$-sphere. This section can be taken as the $\ell_p$-generalization of the work in~\cite{he2023approx}. Unlike the Euclidean sphere which is self-dual, $\ell_p$-sphere makes the study more difficult and the covering results are slightly worse than that of the Euclidean sphere. We are left with an open problem to explicitly construct a deterministic hitting set with hitting ratio $\Omega(\sqrt{\ln{n}/{n}})$ while its existence can be shown; see Section~\ref{sec:rand-hitting-set}. Most results developed in this section apply for $p\in[2,\infty)$ and some apply for $p\in(1,2)$ as well. For ease of reference, we summarize the specifications of our main constructions below in Table~\ref{tab:hitting-sets}.
\begin{table}[!ht]
\centering
\caption{Hitting sets of $\BS_p^n$ with polynomial cardinality 
}
\label{tab:hitting-sets}
\begin{tabular}{|lc|ccc|c|}
\hline
Hitting sets  & $p$ & Hitting ratio  & Cardinality   & Type          & Reference                             \\ \hline
$\BH_{1}^n(\alpha,\beta)$  & $(1,\infty)$                                      & $\mu_{\alpha,\beta}\sqrt[\leftroot{-2}\uproot{2}q]{\frac{\ln{n}}{n+\ln{n}}}$                                       & $n^{\ln{\nu_{\alpha,\beta}}}\left(\frac{\nu_{\alpha,\beta}n}{\ln n}+1\right)$                               & Deterministic & Corollary~\ref{cor:lnn1q}       \\
$\BH_{2}^n(\alpha,\beta)$ & $[2,\infty)$& $\frac{\mu_{\alpha,\beta}\sqrt[\leftroot{-2}\uproot{2}p]{\ln n}}{\sqrt{2n}}$ & $\frac{\nu_{\alpha,\beta}n^{2\ln{\nu_{\alpha,\beta}}+1}}{\ln{n}}$                                   & Deterministic & Corollary~\ref{thm:h2} \\
$\BH_{3}^n(\epsilon)$          & $[2,\infty)$                                          & $\sqrt{\frac{\delta_0 \ln{n}}{2n}}$ & $\limath{\delta_3n^{\delta_2}\left( \left(\frac{1}{2}+\frac{1}{q}\right)n \ln n+\ln \frac{1}{\epsilon}\right)}$ & Randomized & Theorem~\ref{thm:rand-hitting-set}    \\ \hline
\end{tabular}
\end{table}

\subsection{$\Omega(\sqrt[q]{\ln n/n})$-hitting sets of $\BS_p^n$}\label{sec:worst-hitting} 

We start with the best hitting set of $\BS_2^n$ recently proposed by He et al.~\cite{he2023approx}. In particular, our first construction generalizes the $\Omega(\sqrt{\ln n/n})$-hitting set~\cite[Theorem 2.14]{he2023approx} of $\BS_2^n$ to $\Omega(\sqrt[q]{\ln n/n})$-hitting sets of $\BS_p^n$ when $p\in(1,\infty)$. They are the only hitting sets covering the range of $p\in (1,2)$ in the paper.

An important preparation in the construction is an $\Omega(1)$-hitting set of $\BS_p^n$ although its cardinality has to be exponential in $n$. To this end, we adapt the construction in~\cite[Algorithm 2.1]{he2023approx} for $\BS^n_2$ to a one for $\BS_p^n$, i.e., the hitting set $\BH_{H}^n(\alpha, \beta)$ shown in Algorithm~\ref{alg:alg1}.

\begin{algorithm}[!h]\caption{$\BH_{H}^n(\alpha, \beta)$: $\Omega(1)$-hitting sets of $\BS_p^n$ with $p\in(1,\infty)$}
\begin{algorithmic}[1]
\REQUIRE A dimension $n\in\BN$, a constant $p\in(1,\infty)$ and two parameters $\alpha\ge 1$ and $\beta\ge\alpha+1$.
\ENSURE An $\Omega(1)$-hitting set of $\BS_p^n$.
\STATE Let $m=\left\lceil\log_\beta \alpha n \right\rceil$ and partition $\BI:=\{1,2,\dots,n\}$ into 
$\{\BI_1,\BI_2,\dots,\BI_m\}$ such that
\begin{equation}\label{eq:partition}
\BI=\bigcup_{j=1}^m\BI_j,\,\BI_i\bigcap\BI_j=\emptyset \text{ for } i\neq j,\,|\BI_1|=n-\sum_{j=2}^m|\BI_j| \text{ and } |\BI_j|=\left\lfloor \frac{\alpha n}{\beta^{j-1}} \right\rfloor
\text{ for } j=2,3\dots, m.
\end{equation}
\STATE Generate a set of vectors
$$\BX^n=\bigcup_{\{\BI_1,\BI_2,\dots,\BI_m\}\text{ satisfies~\eqref{eq:partition}}}\left\{\bz\in\R^n: z_i\in\left\{\pm1,\pm \beta^{\frac{j-1}{p}}\right\} \text{ if } i\in\BI_j \text{ for } j=1,2,\dots,m \right\}.$$
\RETURN $\BH_{H}^n(\alpha, \beta):=\{\bz/\|\bz\|_p\in\BS_p^n: \bz \in \BX^n\}$.
\end{algorithmic}\label{alg:alg1}
\end{algorithm}

Algorithm~\ref{alg:alg1} definitely generates a nonempty hitting set since there exists some $\{\BI_1,\BI_2,\dots,\BI_m\}$ satisfying~\eqref{eq:partition}. In particular, 
$$
\sum_{j=2}^m |\BI_j| \le \sum_{j=2}^m \frac{\alpha n}{\beta^{j-1}} =\frac{\alpha n}{\beta-1}-\frac{\alpha n}{\beta^{m-1}(\beta-1)}\le\frac{\alpha n}{\beta-1}\le n,
$$
implying that $|\BI_1|=n-\sum_{j=2}^m |\BI_j|\ge0$. We now show that $\BH_{H}^n(\alpha, \beta)$ is an $\Omega(1)$-hitting set of $\BS_p^n$ \li{by exploiting the distinct geometry of the $\ell_p$-sphere.}
\begin{proposition}\label{thm:HE}
\li{If $\alpha\ge 1$ and $\beta\ge\alpha+1$, then}
$$
\BH_{H}^n(\alpha, \beta)\in\allowbreak\BT_p^n\left(\mu_{\alpha,\beta},{\nu_{\alpha,\beta}}^n\right),
$$
where
$$\mu_{\alpha,\beta}:=\sqrt[\leftroot{-2}\uproot{2}p]{\frac{\alpha}{\beta(\alpha+1)}}\left(1-\frac{1}{\alpha}\right) \text{ and } \nu_{\alpha,\beta}:=2^{\frac{\beta+\alpha-1}{\beta-1}} \alpha^{-\frac{\alpha}{\beta-1}}\beta^{\frac{\alpha\beta}{(\beta-1)^2}}\left(\frac{\beta-1}{\beta-\alpha-1}\right)^{\frac{\beta-\alpha-1}{\beta-1}}+o(1).$$
\end{proposition}
\begin{proof}
    The upper bound ${\nu_{\alpha,\beta}}^n$ of the cardinality was proved exactly in~\cite[Proposition~2.9]{he2023approx}. We only estimate the lower bound, $\mu_{\alpha,\beta}$, for the hitting ratio.
    
    For any given $\bx\in\BS_q^n$, denote the index sets
    $$
    \BD_0(\bx)=\left\{i \in \BI:\left|x_i\right| \le \frac{1}{\sqrt[\leftroot{-2}\uproot{2}q]{\alpha n}}\right\}
    \text { and }
    \BD_j(\bx)=\left\{i \in \BI: \sqrt[\leftroot{-2}\uproot{2}q]{\frac{\beta^{j-1}}{\alpha n}}<\left|x_i\right| \le \sqrt[\leftroot{-2}\uproot{2}q]{\frac{\beta^j}{\alpha n}}\right\} \text{ for } j=1,2,\dots,m.
    $$
    Notice that $|x_i|\le 1\le \sqrt[q]{\beta^m/(\alpha n)}$ and so $\{\BD_0(\bx),\BD_1(\bx),\dots,\BD_m(\bx)\}$ is a partition of $\BI$. Besides, it is obvious that for $j\ge2$,
$$
\frac{\beta^{j-1}}{\alpha n}  |\BD_j(\bx)|= \sum_{i \in \BD_j(\bx)}\frac{\beta^{j-1}}{\alpha n} < \sum_{i \in \BD_j(\bx)}|x_i|^q \le 1.
$$
This implies that $|\BD_j(\bx)| < \alpha n/\beta^{j-1}$, i.e., $|\BD_j(\bx)| \le\lfloor\alpha n/\beta^{j-1}\rfloor$ for $j=2,3, \dots, m$. Hence, there exists a partition $\{\BI_1,\BI_2,\dots,\BI_m\}$ of $\BI$ satisfying~\eqref{eq:partition} such that $\BD_j(\bx)\subseteq \BI_j$ for $j=2,3, \dots, m$. Furthermore, we may find a vector $\bz\in\BX^n$ such that
    $$
    z_i= \begin{cases}\sign(x_i) & i \in \BD_0(\bx)\cup \BD_1(\bx) \\ \sign(x_i) \beta^{\frac{j-1}{p}} & i \in \BD_j(\bx) \text { for } j=2,3,\dots, m,\end{cases}
    $$
    where the $\sign$ function takes $1$ for nonnegative reals and $-1$ for negative reals.
    
    We can now estimate $\bz^{\T}\bx$ and $\|\bz\|_p$. First of all, we have
    $$
    \sum_{i \in \BD_0(x)} |x_i|^q \le \sum_{i \in \BD_0(x)} \frac{1}{\alpha n} \le \frac{1}{\alpha} \text{ and }\sum_{j=1}^m \sum_{i \in \BD_j(x)} |x_i|^q = \sum_{i\in\BI} |x_i|^q -  \sum_{i \in \BD_0(\bx)} |x_i|^q \ge 1-\frac{1}{\alpha}.
    $$
    Next, we have
    $$
    \sum_{i \in \BD_0(\bx)} |z_i|^p=\left|\BD_0(\bx)\right| \le n \text{ and }\sum_{j=1}^m \sum_{i \in \BD_j(\bx)} |z_i|^p=\sum_{k=1}^m \sum_{i \in \BD_j(\bx)} \beta^{j-1}<\sum_{k=1}^m \sum_{i \in \BD_j(\bx)} \alpha n |x_i|^q \le \alpha n,
    $$
    implying that $\|\bz\|_p^p \le(\alpha+1) n$ by summing up the two inequalities. Lastly, noticing that $\sign(z_i)=\sign(x_i)$ for every $i$, we have
    $$
    \bz^{\T} \bx \ge \sum_{j=1}^m \sum_{i \in \BD_j(\bx)} \beta^{\frac{j-1}{p}}\left|x_i\right| \ge \sum_{j=1}^m \sum_{i \in \BD_j(\bx)} \sqrt[\leftroot{-2}\uproot{2}p]{\frac{\alpha n}{\beta}}\left|x_i\right|^q \ge \sqrt[\leftroot{-2}\uproot{2}p]{\frac{\alpha n}{\beta}}\left(1-\frac{1}{\alpha}\right),
    $$
    where the second inequality follows from $\left|x_i\right| \le \sqrt[q]{\beta^j/(\alpha n)}$ in defining $\BD_j(\bx)$. Therefore, we conclude that the vector $\bz/\|\bz\|_p\in\BH_{H}^n(\alpha, \beta)$ satisfies
    $$
    \bx^{\T} \frac{\bz}{\|\bz\|_p} \ge \sqrt[\leftroot{-2}\uproot{2}p]{\frac{\alpha n}{\beta}}\left(1-\frac{1}{\alpha}\right) \frac{1}{\sqrt[\leftroot{-2}\uproot{2}p]{(\alpha+1)n}} = 
    \sqrt[\leftroot{-2}\uproot{2}p]{\frac{\alpha}{\beta(\alpha+1)}}\left(1-\frac{1}{\alpha}\right),
    $$
    proving the desired lower bound $\mu_{\alpha,\beta}$.
\end{proof}


We remark that Proposition~\ref{thm:HE} actually improves an existing $\Omega(1)$-hitting set due to Brieden et al.~\cite{brieden2001deterministic}, at least from numerical examples to be shown later. In particular, by letting $
\BY^n(\gamma)=\{\bz\in\BZ^n: 0<\|\bz\|_p\le \gamma n^{1/p}\}$ with a parameter $\gamma>1$, the hitting set 
$$\BH_B^n(\gamma):=\left\{\frac{\bz}{\|\bz\|_p}\in\BS_p^n: \bz \in \BY^n(\gamma)\right\}$$
that was implicitly given in the proof of~\cite[Lemma~3.7]{brieden2001deterministic} satisfies that
\begin{equation}\label{eq:bupper}
\BH_B^n(\gamma)\in\BT_p^n\left(1-\frac{1}{\gamma}, \left(e^{\frac{p}{12}}(2\gamma+1)\max\left\{1,\sqrt{\frac{2\pi}{p}}\right\}\right)^n\right).
\end{equation}


$\BH_B^n(\gamma)$ is essentially based on grid sampling in an $\ell_p$-ball while $\BH_{H}^n(\alpha,\beta)$ provides a more refined selection that has an obvious advantage in terms of the cardinality if both attain the same hitting ratio. As an example, let us choose $p=6$ which is the global minimizer of the above \li{bound}~\eqref{eq:bupper} for $|\BH_B^n(\gamma)|$ and let $\beta=\alpha+1$ for $\BH_{H}^n(\alpha,\beta)$ which maximizes $\mu_{\alpha,\beta}$ for fixed $\alpha$. By setting the hitting ratio $1-\frac{1}{\gamma}=(1-\frac{1}{\alpha})\sqrt[6]{\frac{\alpha}{(\alpha+1)^2}}$, we have the following upper bounds of cardinalities
$$
|\BH_B^n(\gamma)|^{\frac{1}{n}} \le 
\sqrt{\frac{e \pi}{3}}\left(\frac{2}{1-\left(1-\frac{1}{\alpha}\right)\sqrt[\leftroot{-2}\uproot{2}6]{\frac{\alpha}{(\alpha+1)^2}}}+1\right)
\text{ and }
|\BH_{H}^n(\alpha,\alpha+1)|^{\frac{1}{n}} \le \frac{4 (\alpha+1)^{\frac{\alpha+1}{\alpha}}}{\alpha}.
$$
The comparison of the two upper bounds with respect to the same hitting ratio is shown in Figure~\ref{fig:bvh}.
It verifies a clear lower cardinality for $\BH_{H}^n$ compared to that for $\BH_B^n$. Moreover, $|\BH_{H}^n(\alpha,\beta)|$ is invariant to $p$ from Proposition~\ref{thm:HE} while $|\BH_B^n(\gamma)|$ attains its minimum when $p=6$. If $p$ deviates more from $6$, the gap between the two cardinalities may be even larger. On the other hand, $\BH_B^n(\gamma)$ is able to obtain a close-to-one hitting ratio that $\BH_{H}^n(\alpha,\beta)$ cannot achieve, but with a price of large cardinality.

\setlength{\abovecaptionskip}{0pt plus 3pt minus 2pt}
\begin{wrapfigure}[17]{r}{0.42\linewidth}
    \centering
    \includegraphics[width=\linewidth]{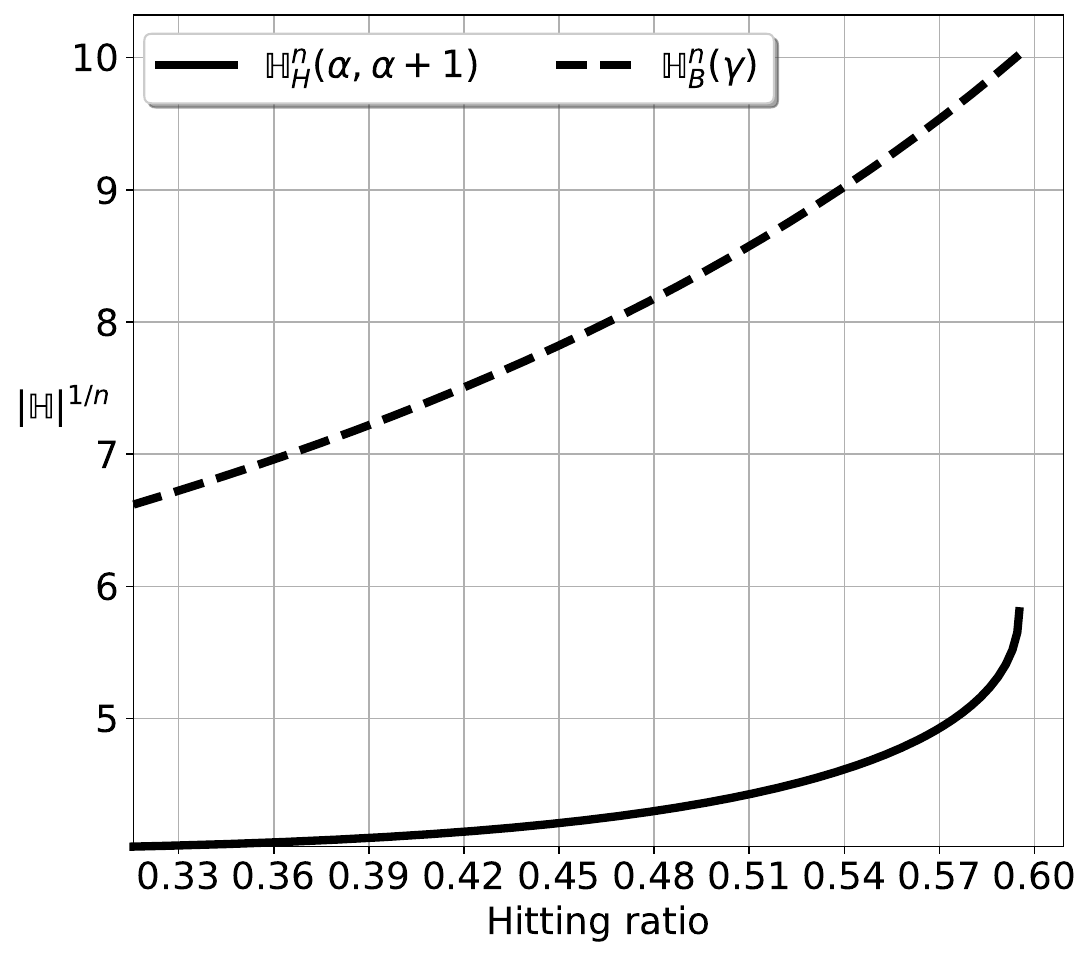}
    \caption{Comparison of upper bounds of $|\BH_B^n(\gamma)|$ and $|\BH_{H}^n(\alpha,\alpha+1)|$ when $p=6$.}
    \label{fig:bvh}
\end{wrapfigure}
Before proceeding to construct $\Omega(\sqrt[q]{\ln n/n})$-hitting sets of $\BS_p^n$, we need to establish two basic properties of hitting sets of $\BS_p^n$. They are straightforwardly generalized from~\cite[Lemma~2.12]{he2023approx} and~\cite[Lemma~2.13]{he2023approx} for hitting sets of $\BS_2^n$, the Euclidean sphere. We omit the proofs as they are simple and similar to that for $\BS_2^n$.

\begin{lemma}\label{lma:kronecker-hitting}
For any $p\in(1,\infty)$, if a hitting set $\BH^{n_1} \in \BT^{n_1}_p(\tau, m)$ with $\tau\ge0$, then 
$$\BE^{n_2} \boxtimes \BH^{n_1} \in \BT^{n_1 n_2}_p\left(\frac{\tau}{\sqrt[\leftroot{-2}\uproot{2}q]{n_2}}, m n_2\right).$$
\end{lemma}

\begin{lemma}\label{lma:appending-hitting}
For any $p\in(1,\infty)$, if two hitting sets $\BH^{n_1} \in \BT^{n_1}_p\left(\tau_1, m_1\right)$ and $\BH^{n_2} \in \BT^{n_2}_p\left(\tau_2, m_2\right)$ with $\tau_1,\tau_2>0$, then
$$
\left(\BH^{n_1} \vee \mathbf{0}_{n_2}\right) \bigcup\left(\mathbf{0}_{n_1} \vee \BH^{n_2}\right) \in \BT^{n_1+n_2}_p\left(\frac{\tau_1 \tau_2}{\sqrt[\leftroot{-2}\uproot{2}q]{{\tau_1}^q+{\tau_2}^q}}, m_1+m_2\right).
$$
\end{lemma}

With all the preparation ready, in particular the $\Omega(1)$-hitting sets of $\BS^n_p$, we are able to construct $\Omega(\sqrt[q]{\ln n/n})$-hitting sets of $\BS_p^n$ with polynomial cardinality. The construction is based on the Kronecker product with the standard basis of a Euclidean space.
\begin{corollary}\label{cor:lnn1q}
Given an integer $n\ge 2$, let $n_1=\lceil\ln n\rceil$, $n_2=\lfloor \frac{n}{n_1}\rfloor$ and $n_3 = n-n_1 n_2$.
For any $p\in(1,\infty)$, $\alpha\ge1$ and $\beta\ge\alpha+1$, one has 
        \begin{align*}
        \BH_{1}^n(\alpha,\beta)&:=\left(\left(\BE^{n_2} \boxtimes \BH_H^{n_1}(\alpha, \beta)\right) \vee \mathbf{0}_{n_3}\right) \bigcup\left(\mathbf{0}_{n_1 n_2} \vee \BH_H^{n_3}(\alpha, \beta)\right) \\
        &\in \BT_p^n\left(\mu_{\alpha,\beta}\sqrt[\leftroot{-2}\uproot{2}q]{\frac{\ln{n}}{n+\ln{n}}}, n^{\ln{\nu_{\alpha,\beta}}}\left(\frac{\nu_{\alpha,\beta}n}{\ln n}+1\right)\right).
        \end{align*}
\end{corollary}

The above result is an immediate consequence of Proposition~\ref{thm:HE}, Lemma~\ref{lma:kronecker-hitting} and Lemma~\ref{lma:appending-hitting}. 
\li{We remark that the hitting ratio $\Omega(\sqrt[q]{\ln n/n})$ for $\ell_p$-sphere in Corollary~\ref{cor:lnn1q} is the best possible by a deterministic hitting set with polynomial cardinality for $p\in(1,2]$; see~\cite[Theorem~3.2]{brieden1998approximation}.}

\subsection{$\Omega(\sqrt[p]{\ln{n}}/\sqrt{n})$-hitting sets of $\BS_p^n$}\label{sec:brieden-hitting-set}

\li{Although $\BH_{1}^n(\alpha,\beta)$ achieves the best hitting ratio 
when $p\in(1,2]$, it is not the best when $p\in(2,\infty)$. In this subsection we propose an improved one only when $p\in[2,\infty)$. It is accomplished by a novel application of Walsh-Hadamard transform that plays a key role of improvement in the construction.} 


To begin with, we introduce \li{Walsh-Hadamard transform~\cite{YH97}, a class} of matrices defined recursively via Kronecker products:
$$I_{1,0}:=I_1 \text{ and } I_{1,k}:=\begin{pmatrix}1&1\\ 1&-1\end{pmatrix} \boxtimes I_{1,k-1}=\begin{pmatrix}I_{1,k-1}&I_{1,k-1}\\ I_{1,k-1}&-I_{1,k-1}\end{pmatrix}\in\R^{2^k\times 2^k} \text{ for } k=1,2,\dots$$
Sometimes there is a normalization factor $2^{-k/2}$ with $I_{1,k}$ and this makes its eigenvalues to be $\pm1$. Without the normalization, the entries of these matrices are either $1$ or $-1$. For our purpose, we extend Walsh-Hadamard transform by replacing $1$ with $I_m$ and $-1$ with $-I_m$, i.e.,
\begin{equation}\label{eq:hmatrix}
  I_{m,0}:=I_m \text{ and } I_{m,k}:=\begin{pmatrix}1&1\\ 1&-1\end{pmatrix}\boxtimes I_{m,k-1}=\begin{pmatrix}I_{m,k-1}&I_{m,k-1}\\ I_{m,k-1}&-I_{m,k-1}\end{pmatrix}\in\R^{2^km\times 2^km} \text{ for } k=1,2,\dots  
\end{equation}
$I_{m,k}$ is symmetric and nonsingular with its inverse given by $2^{-k} I_{m,k}$. Its eigenvalues are $\pm2^{-k/2}$. In fact, upon scaling by the normalization factor $2^{-k/2}$, it actually becomes orthogonal. In particular, the following properties of $I_{m,k}$ are essential to the construction of an improved hitting set.
\begin{lemma}\label{lma:parallelogram}
    For any integer $k\ge0$ and $\bx\in\R^{2^km}$, one has $\|I_{m,k}\bx\|_2=2^{k/2}\|\bx\|_2$. If further $\bx=\be_i\boxtimes \by$ where $\be_i\in\R^{2^k}$ and $\by\in\R^{m}$, then $\|I_{m,k}\bx\|_p=2^{k/p}\|\bx\|_p$.
\end{lemma}
\begin{proof}
    We prove the first statement by induction. The base case $k=0$ is obvious. Assuming that the statement holds for $k$, we have for any $\bx=\bx_1\vee\bx_2\in\R^{2^{k+1}m}$ with $\bx_1,\bx_2\in\R^{2^km}$ that
    \begin{align*}
    {\|I_{m,k+1}\bx\|_2}^2&={\left\|\begin{pmatrix}I_{m,k} & I_{m,k} \\ I_{m,k} & -I_{m,k}\end{pmatrix}\binom{\bx_1}{\bx_2}\right\|_2}^2\\
    &={\|I_{m,k}\bx_1+I_{m,k}\bx_2\|_2}^2+{\|I_{m,k}\bx_1-I_{m,k}\bx_2\|_2}^2    \\
    &=2\left({\|I_{m,k}\bx_1\|_2}^2+{\|I_{m,k}\bx_2\|_2}^2\right)\\
    &=2\left(2^{k} {\|\bx_1\|_2}^2+2^{k} {\|\bx_2\|_2}^2\right)\\
    &=2^{k+1}{\|\bx\|_2}^2,
    \end{align*}
    where the second-to-last equality is due to the induction assumption.

    For the second statement, we observe from~\eqref{eq:hmatrix} that the matrix $I_{m,k}$ can be partitioned into $2^k\times 2^k$ block matrices with each block being either $I_m$ or $-I_m$ since the entries of Walsh-Hadamard transform are $\pm1$. We also observe that the vector $\be_i\boxtimes \by$ can be partitioned into $2^k$ block vectors in $\R^m$ with the $i$th block being $\by$ and others being zero vectors. Therefore, the vector $I_{m,k}(\be_i\boxtimes \by)\in\R^{2^km}$ can be partitioned into $2^k$ block vectors in $\R^m$ with each block being either $\by$ or $-\by$, implying that
    $$
    \|I_{m,k}\bx\|_p=\|I_{m,k}(\be_i\boxtimes \by)\|_p=\left({2^k{\|\by\|_p}^p}\right)^{1/p}=2^{k/p}\|\by\|_p=2^{k/p}\|\bx\|_p.
    $$   
\end{proof}

Denote $J=2^{-k/p}I_{m,k}$. If $\bx=\be_i\boxtimes \by$, Lemma~\ref{lma:parallelogram} tells us that $\|J\bx\|_p=\|\bx\|_p$ but $\|J\bx\|_2=2^{k/2-k/p}\|\bx\|_2$ is actually enlarged when $p\in(2,\infty)$. Moreover, the property is independent on $m$. \li{This informative observation is key to the improvement and explains how it} can be made by applying $J$ from the construction $\BH_{1}^n(\alpha,\beta)$ when $p\in(2,\infty)$.

\begin{lemma}\label{lma:2km}
For any $p\in[2,\infty)$, if a hitting set $\BH^{m} \in \BT_p^m(\tau, m_0)$, then 
$$2^{-\frac{k}{p}}I_{m,k}\big(\BE^{2^k} \boxtimes \BH^{m}\big) \in \BT_p^{2^km}\big(2^{-\frac{k}{2}} m^{\frac{1}{2}-\frac{1}{q}}\tau,2^k m_0\big).$$
\end{lemma}
\begin{proof}
  Given a vector $2^{-k/p}I_{m,k}(\be_i\boxtimes\by)$ in $2^{-k/p}I_{m,k}(\BE^{2^k} \boxtimes \BH^{m})$, $\|2^{-k/p}I_{m,k}(\be_i\boxtimes \by)\|_p=1$ has been verified by the second statement of Lemma~\ref{lma:parallelogram}. The upper bound of the cardinality, $2^km_0$, is also obvious. It suffices to show the lower bound of the hitting ratio.
    

    For any $\bx\in\BS_q^{2^km}$ with $q\in(1,2]$, denote $2^{-k/p}I_{m,k}\bx=\bz_1\vee\bz_2\vee\dots\vee\bz_{2^k}$ where $\bz_i\in\R^m$ for $i=1,2,\dots,2^k$. By the first statement of Lemma~\ref{lma:parallelogram} and the bounds between $\ell_p$-norms (Lemma~\ref{lma:lp-norm-equiv}), we have
    $$
    2^{\frac{k}{2}}\max_{1\le i\le 2^k}\|\bz_i\|_2 \ge \left(\sum_{i=1}^{2^k}{\|\bz_i\|_2}^2\right)^{1/2}
    = \|2^{-\frac{k}{p}}I_{m,k}\bx\|_2 = 2^{\frac{k}{2}-\frac{k}{p}}\|\bx\|_2\ge 2^{\frac{k}{2}-\frac{k}{p}} (2^km)^{\frac{1}{2}-\frac{1}{q}} \|\bx\|_q = m^{\frac{1}{2}-\frac{1}{q}}.
    $$   
    There exists some $j$ such that $\|\bz_j\|_2\ge 2^{-k/2}m^{1/2-1/q}$.
    Moreover, there also exists a vector $\bw\in\BH^m$ such that $\bw^{\T}(\bz_j/\|\bz_j\|_q)\ge\tau$. As a result, we find a vector $2^{-k/p}I_{m,k}(\be_j\boxtimes \bw)\in 2^{-k/p}I_{m,k}(\BE^{2^k} \boxtimes \BH^{m})$ such that
    $$
    \bx^{\T}2^{-\frac{k}{p}}I_{m,k}(\be_j\boxtimes \bw)=\big(2^{-\frac{k}{p}}I_{m,k}\bx\big)^{\T}(\be_j\boxtimes \bw)=\bz_j^{\T}\bw\ge \tau \|\bz_j\|_q\ge  \tau \|\bz_j\|_2\ge 2^{-\frac{k}{2}}m^{\frac{1}{2}-\frac{1}{q}} \tau,
    $$
    where the first equality holds because $I_{m,k}$ is symmetric.
    %
\end{proof}

\li{The trickiest part in this proof is to evenly partition $2^{-k/p}I_{m,k}\bx$.
} If we merely adopt $\BE^{2^k} \boxtimes \BH^{m}$ without applying $2^{-k/p}I_{m,k}$, then the hitting ratio is $2^{-k/q}\tau$ by Lemma~\ref{lma:kronecker-hitting}. Therefore, applying $2^{-k/p}I_{m,k}$ improves the ratio with a factor of $2^{k/q-k/2}m^{1/2-1/q}=(2^k/m)^{1/q-1/2}$, which is indeed an improvement if $2^k>m$ for $q\in(1,2)$.

Suppose now we have $2^k\approx n/ \ln n$ and $m\approx\ln n$. Applying Lemma~\ref{lma:2km} with an $\Omega(1)$-hitting set $\BH^m$, the hitting ratio of $2^{-k/p}I_{m,k}(\BE^{2^k} \boxtimes \BH^{m})$ is then
$2^{-k/2} m^{1/2-1/q}\Omega(1)
=\Omega(\sqrt[p]{\ln{n}}/\sqrt{n})$, a clear improvement from the previous $\Omega(\sqrt[q]{\ln n/n})$-hitting sets in Section~\ref{sec:worst-hitting} when $p\in(2,\infty)$. Therefore, we only need to handle a bit of details on the integerity as $2^k$ can never be $n/\ln n$. Appending a short vector such as Lemma~\ref{lma:appending-hitting} is not doable in this case as we may need to append a vector of dimension $2^k\approx n/\ln n$. Instead, we can apply cutting from a longer vector, \li{a new tool to the design of vectors. It is} essentially the following monotonicity.
\begin{lemma}\label{lma:cut}
For any $p\in(1,\infty)$ and $n_1,n_2\in\BN$ with $n_1\le n_2$, if a hitting set $\BH^{n_2} \in \BT^{n_2}_p\left(\tau, m\right)$ with $\tau>0$, then
$$
\BH^{n_1}(\BH^{n_2}):=\left\{\frac{\bz}{\|\bz\|_p}:\bz\in\R^{n_1},\,\bz\neq {\bf 0},\,\bz\vee\by\in\BH^{n_2}\text{ for some }\by\in\R^{n_2-n_1}\right\}\in \BT^{n_1}_p\left(\tau, m\right).
$$
\end{lemma}
\begin{proof}
    For any $\bx\in\BS_q^{n_1}$, it is obvious that $\bx\vee{\bf 0}_{n_2-n_1}\in\BS_q^{n_2}$. There is a vector $\bz\vee\by\in\BH^{n_2}$ with $\bz\in\R^{n_1}$ and $\by\in\R^{n_2-n_1}$ such that $(\bx\vee{\bf 0})^{\T} (\bz\vee\by) \ge \tau$, i.e., $\bx^{\T}\bz \ge \tau$. By that $\bz\neq {\bf 0}$ as $\tau>0$, we find a vector $\bz/\|\bz\|_p\in\BH^{n_1}(\BH^{n_2})$ such that 
    $$
    \bx^{\T}\frac{\bz}{\|\bz\|_p}\ge\frac{\tau}{\|\bz\|_p}\ge\frac{\tau}{\|\bz\vee\by\|_p} = \tau.
    $$
    On the other hand, it is easy to see that $|\BH^{n_1}(\BH^{n_2})|\le|\BH^{n_2}|\le m$.
\end{proof}

We conclude this subsection by proposing the following hitting set with an improved hitting ratio when $p\in(2,\infty)$ as an immediate consequence of Lemma~\ref{lma:2km}, Proposition~\ref{thm:HE} and Lemma~\ref{lma:cut}.
\begin{corollary}\label{thm:h2}
Given an integer $n\ge 2$, let $k=\lfloor\log_2\frac{n}{\ln n}\rfloor$ and $m=\lceil2^{-k}n\rceil$.
For any $p\in[2,\infty)$, $\alpha\ge1$ and $\beta\ge\alpha+1$, one has 
        \begin{align*}
        \BH_{2}^n(\alpha,\beta)&:=\BH^{n}\left(2^{-\frac{k}{p}}I_{m,k}\big(\BE^{2^k}\boxtimes \BH_{H}^m(\alpha, \beta) \big)\right)
\in \BT_p^n\left(\frac{\mu_{\alpha,\beta}\sqrt[\leftroot{-2}\uproot{2}p]{\ln n}}{\sqrt{2n}}, \frac{\nu_{\alpha,\beta}n^{2\ln{\nu_{\alpha,\beta}}+1}}{\ln{n}}\right).
        \end{align*}
\end{corollary}
\begin{proof}
   We first observe the following bounds
   $$\frac{n}{2\ln n}< 2^k\le \frac{n}{\ln n},\,\ln n\le 2^{-k}n\le m<2^{-k}n+1 \text{ and } n\le 2^km< n+2^k\le 2n.$$
      By Lemma~\ref{lma:2km} and Proposition~\ref{thm:HE}, $2^{-k/p}I_{m,k}(\BE^{2^k}\boxtimes \BH_{H}^m(\alpha, \beta))$ is a hitting set of $\BS^{2^km}_p$ with the hitting ratio 
    $$
    2^{-\frac{k}{2}} m^{\frac{1}{2}-\frac{1}{q}} \mu_{\alpha,\beta}=(2^km)^{-\frac{1}{2}}m^{\frac{1}{p}}\mu_{\alpha,\beta}
    \ge\frac{\mu_{\alpha,\beta}\sqrt[\leftroot{-2}\uproot{2}p]{\ln n}}{\sqrt{2n}}
    $$
    and the cardinality no more than 
    $$2^k{\nu_{\alpha,\beta}}^m \le \frac{n}{\ln{n}}{\nu_{\alpha,\beta}}^{2^{-k}n+1}\le\frac{n}{\ln{n}}{\nu_{\alpha,\beta}}^{2\ln{n}+1}=\frac{\nu_{\alpha,\beta}n^{2\ln{\nu_{\alpha,\beta}}+1}}{\ln{n}},$$
    where the last inequality follows from $\frac{n}{2\ln n}<2^k$. The conclusion then follows by Lemma~\ref{lma:cut} since $n\le 2^km$.
\end{proof}

\subsection{Randomized $\Omega(\sqrt{\ln n/n})$-hitting sets of $\BS_p^n$}\label{sec:rand-hitting-set}

The $\Omega(\sqrt[p]{\ln{n}}/\sqrt{n})$-hitting sets in Section~\ref{sec:brieden-hitting-set} already significantly increased the $\Omega(\sqrt[q]{\ln{n}/n})$ one in Section~\ref{sec:worst-hitting} when $p\in(2,\infty)$. However, there is still room to be improved with the help of randomization. In particular, we provide randomized $\Omega(\sqrt{\ln{n}/{n}})$-hitting sets of $\BS_p^n$. 

\li{Similar to the randomized $\Omega(\sqrt{\ln{n}/{n}})$-hitting sets of $\BS_2^n$~\cite[Section~2.1]{he2023approx}, the following probability bound of the inner product between a random vector on $\BS_p^n$ and an arbitrary vector on $\BS_q^n$ is a necessity.}
\begin{lemma}[{\cite[Lemma 3.3]{khot2008linear}}]\label{lma:lp-sampling}
    For any $p\in(2,\infty)$ and integer $n\ge2$, let $\by\in\R^n$ whose entries are i.i.d.\ random variables with the probability density function \li{$p e^{-|x|^p}/\int_0^{\infty} 2t^{\frac{1-p}{p}} e^{-t} dt$}. There exist universal constants $\delta_0,\delta_1,\delta_2>0$ such that for any $\bz\in\BS_q^n$
    $$
    \Prob \left\{\frac{\bz^\T \by}{\|\by\|_p} \ge \sqrt{\frac{\delta_0 \ln{n}}{n}}\right\} \ge \frac{\delta_1}{n^{\delta_2}}.
    $$
\end{lemma}
\li{However, this result is certainly not enough and even does not guarantee the existence of a hitting set with the required hitting ratio. Moreover, unlike the $\ell_2$-sphere where the closeness can be transferred meaning that if $\bx$ is close to $\by$ and $\by$ is close to $\bz$ then $\bx$ is close to $\bz$ as they all sit on the same sphere, the non-self-duality of the $\ell_p$-sphere creates problems again. To overcome the difficulty, we first need to create a grid on the $\ell_q$-sphere that is close enough to the same $\ell_q$-sphere and then construct a randomized hitting set on the $\ell_p$-sphere that is close to the grid on the $\ell_q$-sphere. To begin with, we first settle a grid on the $\ell_q$-sphere with a desired property.}
\begin{lemma}\label{lma:lq-enet}
    For any $p\in[1,\infty]$ and $m,n\in\BN$, one has $|\BH^n_G(m)|\le(2m+1)^n$ where
    $$
    \BH^n_G(m):=\left\{\frac{\bw}{\|\bw\|_p}\in\BS_p^n:\bw\neq\bd{0},\, w_i\in\left\{0,\pm\frac{1}{m},\pm\frac{2}{m},\dots,\pm 1\right\}\text{ for } i=1,2,\dots, n\right\}.
    $$
    For any $\bx\in\BS_p^n$, there exists $\by\in\BH^n_G(m)$ such that for any $\bz\in\R^n$,
    \begin{align*}
        \by^\T\bz-\frac{\|\bz\|_1}{m} &\le \bx^\T\bz \le \left(1+\frac{n^{1/p}}{m}\right)\by^\T\bz+\frac{\|\bz\|_1}{m} &\text{if }\by^\T\bz\ge 0,\\
        \left(1+\frac{n^{1/p}}{m}\right)\by^\T\bz-\frac{\|\bz\|_1}{m}&\le\bx^\T\bz\le \by^\T\bz+\frac{\|\bz\|_1}{m} &\text{if }\by^\T\bz< 0.
    \end{align*}
\end{lemma}
\begin{proof}
    $|\BH^n_G(m)|\le(2m+1)^n$ obviously holds from its definition. For any $\bx\in\BS_p^n$, we choose $w_i\in\{0,\pm\frac{1}{m},\pm\frac{2}{m},\dots,\pm 1\}$ to be the closest to $x_i$ satisfying $|x_i|\le|w_i|$, i.e., for $i=1,2,\dots,n$
    $$
    w_i=\begin{cases}
        \frac{\lceil m x_i\rceil}{m} & \text{if } x_i\ge 0 \\
        \frac{\lfloor m x_i\rfloor}{m} &  \text{if } x_i< 0.
    \end{cases}
    $$
    By noticing that $0\le |w_i|-|x_i|\le \frac{1}{m}$, one has 
    \begin{equation}\label{eq:bound-y}
    1=\|\bx\|_p\le\|\bw\|_p=\left\||\bw|\right\|_p\le\left\||\bx|+\frac{1}{m}\bd{1}\right\|_p\le\left\||\bx|\right\|_p+\frac{1}{m}\left\|\bd{1}\right\|_p=1+\frac{n^{1/p}}{m}.
    \end{equation}
    Since $|w_i-x_i|\le \frac{1}{m}$, for any $\bz\in\R^n$,
    $$
    \left|\bw^\T\bz-\bx^\T\bz\right| \le \sum_{i=1}^n |w_i-x_i|\cdot |z_i| \le \sum_{i=1}^n \frac{1}{m} |z_i| = \frac{\|\bz\|_1}{m}.
    $$
    Let $\by=\bw/\|\bw\|_p\in\BS_p^n$ and the above leads to
    $$
    \|\bw\|_p\by^\T\bz-\frac{\|\bz\|_1}{m}\le\bx^\T\bz\le \|\bw\|_p\by^\T\bz+\frac{\|\bz\|_1}{m}.
    $$
    The desired bounds of $\bx^\T\bz$ can be obtained immediately based on the sign of $\by^\T\bz$ by noticing the lower and upper bounds of $\|\bw\|_p$ in~\eqref{eq:bound-y}.
\end{proof}

\li{We are ready to present randomized $\Omega(\sqrt{\ln{n}/{n}})$-hitting sets of $\BS_p^n$ by putting all the pieces together and carefully applying many probabilistic techniques.} To simplify the language, we call the vector $\by/\|\by\|_p\in\BS^n_p$ to be the even distribution on $\ell_p$-sphere, where the entries of $\by\in\R^n$ are i.i.d.\ random variables with the probability density function \li{$p e^{-|x|^p}/\int_0^{\infty} 2t^{\frac{1-p}{p}} e^{-t} dt$}. 
\begin{theorem}\label{thm:rand-hitting-set}
For any $p\in(2,\infty)$ and $\epsilon\in(0,1)$, there exist universal constants $\delta_0,\delta_2,\delta_3>0$, such that
$$
\BH_{3}^n(\epsilon):=\left\{\bz_i \text{ is i.i.d.\ even on $\BS_p^n$ for }i=1,2,\dots, \left\lceil\delta_3n^{\delta_2}\left( \left(\frac{1}{2}+\frac{1}{q}\right)n \ln n+\ln \frac{1}{\epsilon}\right)\right\rceil \right\}
$$
satisfies
$$\Prob\left\{\BH_{3}^n(\epsilon)\in\BT_p^n\left(\sqrt{\frac{\delta_0 \ln{n}}{2n}}, \left\lceil\delta_3n^{\delta_2}\left( \left(\frac{1}{2}+\frac{1}{q}\right)n \ln n+\ln \frac{1}{\epsilon}\right)\right\rceil\right)\right\}\ge1-\epsilon.
$$
\end{theorem}


\begin{proof}
The proof consists two parts. We first provide sufficient conditions for $\BH_3^n(\epsilon)$ to cover $\BS^n_q$ with high probability by using $\BH^n_G(m)$ as an intermediary. We then analyze the parameters to make sure the probability lower bound and the upper bound of $|\BH_{3}^n(\epsilon)|$ are achieved uniformly.

First, we show that $\BH_3^n(\epsilon)$ is close enough to $\BH^n_G(m)$ with high probability, i.e.,
    \begin{equation}\label{eq:halfway}
    \Prob\left\{\min_{\by\in\BH^n_G(m)}\max_{\bz\in\BH_{3}^n(\epsilon)}\bz^\T\by \ge \sqrt{\frac{\delta_0 \ln{n}}{n}}\right\} \ge 1-\epsilon.
    \end{equation}
    To see why, we have that the probability of its complement
    \begin{align*}
    \Prob\left\{\min_{\by\in\BH^n_G(m)}\max_{\bz\in\BH_{3}^n(\epsilon)}\bz^\T\by<\sqrt{\frac{\delta_0 \ln{n}}{n}}\right\}
    &=\Prob\left\{\bigcup_{\by\in\BH^n_G(m)}\bigcap_{\bz\in\BH_{3}^n(\epsilon)}\left\{\bz^\T\by<\sqrt{\frac{\delta_0 \ln{n}}{n}}\right\}\right\}\\
    &\le \sum_{\by\in\BH^n_G(m)}\Prob\left\{\bigcap_{\bz\in\BH_{3}^n(\epsilon)}\left\{\bz^\T\by<\sqrt{\frac{\delta_0 \ln{n}}{n}}\right\}\right\}\\
    &= \sum_{\by\in\BH^n_G(m)}\prod_{\bz\in\BH_{3}^n(\epsilon)}\Prob\left\{\bz^\T\by<\sqrt{\frac{\delta_0 \ln{n}}{n}}\right\}
    \\&\le (2m+1)^n \left(1-\frac{\delta_1}{n^{\delta_2}}\right)^{|\BH_{3}^n(\epsilon)|},
    \end{align*}
    where the first inequality follows from the subadditivity of probability measure, the second equality is due to the independence of $\bz$'s, and the last inequality follows from Lemma~\ref{lma:lq-enet} and Lemma~\ref{lma:lp-sampling}. To ensure the upper bound $(2m+1)^n (1-\frac{\delta_1}{n^{\delta_2}})^{|\BH_{3}^n(\epsilon)|}\le\epsilon$, it is not difficult to verify that it suffices to have 
    \begin{equation}\label{eq:bh3card}
    |\BH_{3}^n(\epsilon)|\ge\frac{n^{\delta_2}}{\delta_1}\left(n\ln{(2m+1)}+\ln{\frac{1}{\epsilon}}\right).
    \end{equation}
    Therefore, for any $\bx\in\BS_q^n$, there exists $\by\in\BH^n_G(m)$ by Lemma~\ref{lma:lq-enet}, and then there exists $\bz\in\BH_{3}^n(\epsilon)$ with probability at least $1-\epsilon$ by~\eqref{eq:halfway}, such that
    $$
    \bz^\T\bx\ge \bz^\T\by - \frac{\|\bz\|_1}{m} \ge \sqrt{\frac{\delta_0 \ln{n}}{n}} - \frac{n^{1/q}\|\bz\|_p}{m} =\sqrt{\frac{\delta_0 \ln{n}}{n}} - \frac{n^{1/q}}{m},
    $$
    where the second inequality is due to the bounds between $\ell_p$-norms (Lemma~\ref{lma:lp-norm-equiv}). To ensure the lower bound $\sqrt{\frac{\delta_0 \ln{n}}{n}} - \frac{n^{1/q}}{m}\ge \sqrt{\frac{\delta_0 \ln{n}}{2n}}$, the hitting ratio, we need to have $\frac{n^{1/q}}{m} \le (\sqrt{2}-1)\sqrt{\frac{\delta_0 \ln{n}}{2n}}$. In a word, it suffices to have
    \begin{equation}\label{eq:bh3lower}
     2m+1\ge\sqrt{\frac{47}{\delta_0 \ln{n}}}n^{\frac{1}{q} + \frac{1}{2}} + 1.
    \end{equation}

    Let us carefully define $m$ for every $n$ in order to satisfy both~\eqref{eq:bh3card} and~\eqref{eq:bh3lower}. Let $n_0\ge2$ be the smallest one such that $n\ge \sqrt{\frac{47}{\delta_0 \ln{n}}}n + 3$ for all $n\ge n_0$ and this $n_0$ depends only on $\delta_0$. We then let $\delta\ge1$ be the smallest one such that $n^{\delta}\ge \sqrt{\frac{47}{\delta_0 \ln{n}}}n + 3$ for all $2\le n\le n_0$ and this $\delta$ also depends only on $\delta_0$. Therefore, we uniformly have $n^{\delta}\ge \sqrt{\frac{47}{\delta_0 \ln{n}}}n + 3$ for any $n\ge2$. Since $\delta\ge1$ and $\frac{1}{q}-\frac{1}{2}\ge0$, one has
    $$n^{(\frac{1}{q} + \frac{1}{2})\delta} = n^{(\frac{1}{q} - \frac{1}{2})\delta}  n^\delta \ge n^{\frac{1}{q} - \frac{1}{2}}\left( \sqrt{\frac{47}{\delta_0 \ln{n}}}n + 3\right) \ge  \sqrt{\frac{47}{\delta_0 \ln{n}}}n^{\frac{1}{q} + \frac{1}{2}} + 3.$$
    To summarize, if we define $m=\lfloor\frac{1}{2}(n^{\delta/q+\delta/2}-1)\rfloor$, then 
    $$ n^{(\frac{1}{q} + \frac{1}{2})\delta} \ge 2m+1 \ge n^{(\frac{1}{q} + \frac{1}{2})\delta} -2
      \ge\sqrt{\frac{47}{\delta_0 \ln{n}}}n^{\frac{1}{q} + \frac{1}{2}} + 1.
    $$
    This shows that~\eqref{eq:bh3lower} holds for any $n\ge2$. In the meantime, if we let $\delta_3=\delta/\delta_1$ that depends only on $\delta_0$ and $\delta_1$ and is thus universal, we have 
    $$
    |\BH_{3}^n(\epsilon)| 
    \ge \frac{\delta}{\delta_1} n^{\delta_2}\left( \left(\frac{1}{2}+\frac{1}{q}\right)n \ln n+\ln \frac{1}{\epsilon}\right)
    = \frac{n^{\delta_2}}{\delta_1}\left(\delta n\ln n^{\frac{1}{2}+\frac{1}{q}}+\delta\ln{\frac{1}{\epsilon}}\right)
    \ge \frac{n^{\delta_2}}{\delta_1}\left(n\ln{(2m+1)}+\ln{\frac{1}{\epsilon}}\right),
    $$
    ensuring the validity of~\eqref{eq:bh3card}.
    %
\end{proof}


Theorem~\ref{thm:rand-hitting-set} not only provides a simple construction via randomization but also trivially implies the existence of $\Omega(\sqrt{\ln n/n})$-hitting sets of $\BS_p^n$. However, we are currently unable to explicitly construct a deterministic one when $p\in(2,\infty)$. As we shall see in Section~\ref{sec:algorithms}, a deterministic $\Omega(\sqrt{\ln n/n})$-hitting set can be used to improve the best-known approximation bound by a deterministic polynomial-time algorithm for both the tensor spectral $p$-norm and nuclear $p$-norm.

We remark that the hitting ratio $\Omega(\sqrt{\ln n/n})$ for $\BS_p^n$ is the largest possible by a deterministic hitting set with polynomial cardinality when $p\in[2,\infty)$; see~\cite[Theorem~3.2]{brieden1998approximation}. Therefore, with the $\Omega(\sqrt[q]{\ln n/n})$-hitting set $\BH_{1}^n(\alpha,\beta)$ that attains the largest hitting ratio with polynomial cardinality when $p\in(1,2]$, this section almost provides a complete story under the deterministic framework, except explicit construction of deterministic $\Omega(\sqrt{\ln n/n})$-hitting sets when $p\in(2,\infty)$.

\section{Approximating tensor nuclear $p$-norm}\label{sec:algorithms}

As an application of the results in Section~\ref{sec:matrix-p-norm} and Section~\ref{sec:hitting-sets}, this section is devoted to the design and analysis of polynomial-time algorithms to approximate the tensor nuclear $p$-norm. As mentioned in the introduction, such results are almost blank in the literature mainly because of the lack of good approximation of the matrix nuclear $p$-norm. Armed with the $\Omega(1)$-approximation bound of the matrix nuclear $p$-norm developed in Section~\ref{sec:matrix-p-norm}, we provide an overview of approximating the tensor nuclear $p$-norms using existing tools as well as what we have developed in Section~\ref{sec:hitting-sets}, from a basic approach to the best approximation. Here in this section, we consider the tensor space $\R^{n_1\times n_2\times \dots \times n_d}$ of order $d \ge 3$ and assume without loss of generality that $2 \leq n_1 \leq n_2 \leq \dots \leq n_d$. 

Before discussing the tensor nuclear $p$-norm, let us first propose very simple algorithms to approximate the tensor spectral $p$-norm, as an immediate application of $\ell_p$-sphere covering in Section~\ref{sec:hitting-sets}.

\begin{algorithm}[!h]
\begin{algorithmic}[1]
\REQUIRE A tensor $\TT\in\R^{n_1 \times n_2 \times \dots \times n_d}$, a constant $p\in\BQ\cap(2,\infty)$ and $d-2$ hitting sets $\BH^{n_k}\in\BT^{n_k}_p(\tau_k, O({n_k}^{\alpha_k}))$ for $k=1,2,\dots,d-2$.
\ENSURE An approximation of $\|\TT\|_{p_\sigma}$.
\STATE Applying~\eqref{cor:final-spectral-p-norm} to compute
\begin{equation*}
u =\max\left\{\left\|\TT\times_1\bx_1\dots\times_{d-2}\bx_{d-2}\right\|_{p_v} :
\bx_k\in \BH^{n_k},\,k=1,2,\dots, d-2\right\};
\end{equation*}
\RETURN $u/\delta_G$.
\end{algorithmic}
\caption{Approximating the tensor spectral $p$-norm based on $\ell_p$-sphere covering}
\label{alg:alg0}
\end{algorithm}

\begin{theorem}\label{thm:snorm}
For any $\TT\in\R^{n_1 \times n_2 \dots \times n_d}$, $p\in \BQ\cap(2,\infty)$ and hitting sets $\BH^{n_k}\in\BT^{n_k}_p(\tau_k, O({n_k}^{\alpha_k}))$ for $k=1,2,\dots,d-2$,
Algorithm~\ref{alg:alg0} is a deterministic polynomial-time algorithm whose output $\operatorname{Alg}_{\,\ref{alg:alg0}}(\|\TT\|_{p_\sigma})$ satisfies
$$
\left(\frac{1}{\delta_G}\prod_{k=1}^{d-2}\tau_k\right)\|\TT\|_{p_\sigma}\le \operatorname{Alg}_{\,\ref{alg:alg0}}(\|\TT\|_{p_\sigma}) \le \|\TT\|_{p_\sigma}.
$$
\end{theorem}
\begin{proof}
Denote $(\by_1,\by_2,\dots,\by_d)$ to be an optimal solution of~\eqref{def:spectral}, i.e., $\langle\TT,\by_1\otimes\by_2\otimes \dots\otimes\by_d\rangle=\|\TT\|_{p_\sigma}$ with $\|\by_k\|_p=1$ for $k=1,2,\dots,d$.
For the vector $\bv_1=\TT\times_2\by_2\dots\times_{d}\by_{d}\in\R^{n_1}$, either $\|\bv_1\|_q=0$ or there exists $\bz_1\in\BH^{n_1}$ such that $\bz_1^{\T}\bv_1/\|\bv_1\|_q\ge\tau_1$. In any case, one has $\bz_1^{\T} \bv_1 \ge \tau_1 \|\bv_1\|_q$ and
\[
\langle\TT,\bz_1\otimes\by_2\otimes \dots\otimes\by_d\rangle= \bz_1^{\T} \bv_1 \ge \tau_1 \|\bv_1\|_q
\ge \tau_1 \by_1^{\T}\bv_1=\tau_1 \langle\TT,\by_1\otimes\by_2\otimes \dots\otimes\by_d\rangle,
\]
where the last inequality follows from H\"{o}lder's inequality and $\|\by_1\|_p=1$. Similarly, for every $k=2,3\dots,d-2$ that are chosen one by one increasingly, there exists $\bz_k\in\BH^{n_k}$ such that
\[
\langle \TT, \bz_1\otimes\dots\otimes\bz_k\otimes\by_{k+1}\otimes\dots\otimes\by_{d}\rangle = \bz_k^{\T} \bv_k \ge \tau_k \|\bv_k\|_q 
\ge \tau_k \by_k^{\T}\bv_k= \tau_k \langle \TT, \bz_1\otimes\dots\otimes\bz_{k-1}\otimes\by_{k}\otimes\dots\otimes\by_{d}\rangle,
\]
where $\bv_k=\TT\times_1\bz_1\dots\times_{k-1}\bz_{k-1}\times_{k+1}\by_{k+1}\dots\times_d\by_{d}\in\R^{n_k}$. By applying the above inequalities recursively, one has
\[
\langle \TT, \bz_1\otimes\dots\otimes\bz_{d-2}\otimes\by_{d-1}\otimes\by_{d}\rangle\ge \left(\prod_{k=1}^{d-2}\tau_k\right) \langle\TT,\by_1\otimes\by_2\otimes \dots\otimes\by_d\rangle  = \left(\prod_{k=1}^{d-2}\tau_k\right) \|\TT\|_{p_\sigma}.
\]
Therefore, by Lemma~\ref{lma:vecp}, the $u$ in Algorithm~\ref{alg:alg0} satisfies
\begin{align*}
u& \ge \left\|\TT\times_1\bz_1\dots\times_{d-2}\bz_{d-2}\right\|_{p_v}\\
&\ge \left\|\TT\times_1\bz_1\dots\times_{d-2}\bz_{d-2}\right\|_{p_\sigma} \\
&= \max\left\{\left\langle\TT\times_1\bz_1\dots\times_{d-2}\bz_{d-2},\bx_{d-1}\otimes\bx_{\limath{d}} \right\rangle: \|\bx_{d-1}\|_p=\|\bx_{\limath{d}}\|_p=1\right\} \\
&\ge\langle \TT, \bz_1\otimes\dots\otimes\bz_{d-2}\otimes\by_{d-1}\otimes\by_{d}\rangle \\
&\ge \left(\prod_{k=1}^{d-2}\tau_k\right) \|\TT\|_{p_\sigma},
\end{align*}
implying that $u/\delta_G\ge (\prod_{k=1}^{d-2}\tau_k) \|\TT\|_{p_\sigma}/\delta_G$. On the other hand, one also has
\begin{align*}
\frac{u}{\delta_G}& = \max\left\{\frac{1}{\delta_G}\left\|\TT\times_1\bx_1\dots\times_{d-2}\bx_{d-2}\right\|_{p_v}:
\bx_k\in \BH^{n_k},\,k=1,2,\dots, d-2\right\}\\
&\le \max\left\{\left\|\TT\times_1\bx_1\dots\times_{d-2}\bx_{d-2}\right\|_{p_\sigma}:
\bx_k\in \BH^{n_k},\,k=1,2,\dots, d-2\right\} \\
&\le  \|\TT\|_{p_\sigma},
\end{align*}
where the last inequality follows by applying Lemma~\ref{thm:contraction} $d-2$ times.
\end{proof}

By applying deterministic hitting sets $\BH_{2}^n(\alpha,\beta)$ with hitting ratio $\Omega(\sqrt[p]{\ln{n}}/\sqrt{n})$ to Algorithm~\ref{alg:alg0}, the tensor spectral $p$-norm can be approximated within a bound of $\Omega(\prod_{k=1}^{d-2} \sqrt[p]{\ln{n_k}}/\sqrt{n_k})$ and by applying randomized hitting sets $\BH_{3}^n(\epsilon)$ with hitting ratio $\Omega(\sqrt{\ln n/n})$, the approximation bound can be improved to 
$\Omega(\prod_{k=1}^{d-2} \sqrt{\ln{n_k}/n_k})$ when $p\in\BQ\cap(2,\infty)$. Both bounds are exactly the same to the best-known ones by Hou and So~\cite[Theorem~7]{hou2014hardness}. However, Algorithm~\ref{alg:alg0} is much simpler than that in~\cite{hou2014hardness}.

We remark that approximation of the tensor spectral $p$-norm serves a foundation to many $\ell_p$-sphere or $\ell_p$-ball constrained polynomial optimization problems~\cite{he2010approximation,so2011deterministic,hou2014hardness}. Hence, Algorithm~\ref{alg:alg0} provides a new tool to study approximation algorithms of these \li{problems. As an example, consider $\max_{\bx\in\BS_p^n}p(\bx)$ where $p(\bx)$ is a homogeneous polynomial function of degree $d$. By adopting the tensor relaxation approach~\cite{libook12}, i.e., relaxing the problem to $\max_{\bx_k\in\BS_p^n,\,k=1,2,\dots,d}\langle\TT,\bx_1\otimes\bx_2\otimes \dots\otimes\bx_d\rangle$ where $\TT$ is the symmetric tensor associated with the homogeneous polynomial $p(\bx)$, an approximate solution obtained in Theorem~\ref{thm:snorm} can be used to construct an approximate solution of the original problem $\max_{\bx\in\BS_p^n}p(\bx)$ via the so-called polarization formula introduced in~\cite[Lemma~1]{he2010approximation}. In particular, the following result elaborates the equivalence (in the approximation sense) between the $\ell_p$-sphere constrained homogeneous polynomial optimization and the tensor spectral $p$-norm.
\begin{lemma}[{\cite[Theorem~2]{hou2014hardness}}]\label{lma:equiv-hou}
    Let $p\in [2,\infty]$ and integer $d\ge3$ be given. Suppose there is a polynomial-time algorithm that approximates the spectral $p$-norm of any symmetric tensor $\TT$ of order $d$ within a factor of $\alpha\in(0,1]$.
    Then, there is a polynomial-time algorithm that approximates $\max_{\bx\in\BS_p^n}p(\bx)$ where $p(\bx)=\langle\TT,\bx\otimes \bx\otimes \dots\otimes\bx\rangle$ within a factor of $d!  d^{-d}\alpha$ when $d$ is odd. The approximation ratio is in a relative sense if $d$ is even.    
\end{lemma}
Therefore, by applying Algorithm~\ref{alg:alg0} with either the deterministic hitting set $\BH_{2}^n(\alpha,\beta)$ or the randomized one $\BH_3^n(\epsilon)$, Theorem~\ref{thm:snorm} and Lemma~\ref{lma:equiv-hou} lead to the following result on $\max_{\bx\in\BS_p^n}p(\bx)$. We leave the details to interested readers. 
\begin{corollary}
For any given $p\in(2,\infty)$ and integer $d\ge3$, there is a deterministic (respectively, randomized) polynomial-time approximation algorithm for $\ell_p$-sphere constrained $n$-variate degree-$d$ homogeneous polynomial optimization with approximation ratio $\Omega\big((\ln n)^{\frac{d-2}{p}} n^{-\frac{d-2}{2}}\big)$ (respectively, $\Omega\big((\ln n)^{\frac{d-2}{2}} n^{-\frac{d-2}{2}}\big)$) when $d$ is odd. The approximation ratio is in a relative sense when $d$ is even. 
\end{corollary}
The above approximation ratios match currently the best ones by Hou and So~\cite{hou2014hardness}. Our algorithms are easy to implement thanks to the SDP representations in Corollary~\ref{lma:dual-SDP} and the constructions of hitting sets while the methods in~\cite{hou2014hardness} are not implementable. We remark that it is also possible to develop approximation algorithms for general inhomogeneous polynomial optimization over the $\ell_p$-ball with the technique developed in~\cite{he2015approximation}. We leave these to interested readers.} 

\li{Let us} now summarize approximation methods of the tensor nuclear $p$-norm in this section in Table~\ref{tab:alg-ratios} before the detailed discussion.
\begin{table}[!ht]
\centering
\caption{Approximation methods of the nuclear $p$-norm of a tensor $\TT\in\R^{n_1\times n_2\dots\times n_d}$}
\label{tab:alg-ratios}
\begin{tabular}{|lc|ccc|}
\hline
Algorithm &$p$ & Approximation bound    & Approach    & Type        \\ \hline
\cite[Proposition~4.3]{chen2020tensor} & $(1,\infty)$ & $\prod_{k=1}^{d-1} \frac{1}{\sqrt[\leftroot{-2}\uproot{2}q]{n_k}}$                                 &Partition to vectors  & Deterministic  \\
Algorithm~\ref{alg:matricization}  & $\BQ\cap(2,\infty)$& $\frac{1}{\delta_G}\prod_{k=1}^{d-2} \frac{1}{\sqrt[\leftroot{-2}\uproot{2}q]{n_k}}$    &Matrix unfolding   & Deterministic \\
Algorithm~\ref{alg:partition} &  $\BQ\cap(2,\infty)$  & $\frac{1}{\delta_G}\prod_{k=1}^{d-2} \frac{1}{\sqrt[\leftroot{-2}\uproot{2}q]{n_k}}$                                 &Partition to matrices & Deterministic \\
Algorithm~\ref{alg:alg3} with $\BH_{1}^{n_k}$  & $\BQ\cap(2,\infty)$ & $\Omega\left(\prod_{k=1}^{d-2} \sqrt[\leftroot{-2}\uproot{2}q]{\frac{\ln n_k}{n_k}}\right)$                        & $\ell_p$-sphere covering     & Deterministic \\
Algorithm~\ref{alg:alg3} with $\BH_{2}^{n_k}$  & $\BQ\cap(2,\infty)$ & $\Omega\left(\prod_{k=1}^{d-2} \frac{\sqrt[\leftroot{-2}\uproot{2}p]{\ln{n_k}}}{\sqrt{n_k}}\right)$ & $\ell_p$-sphere covering & Deterministic \\
Algorithm~\ref{alg:rand-alg}   & $\BQ\cap(2,\infty)$ & $\Omega\left(\prod_{k=1}^{d-2} \sqrt{\frac{\ln{n_k}}{n_k}}\right)$    & $\ell_p$-sphere covering                                                    & Randomized  \\ \hline
\end{tabular}%
\end{table}

\subsection{Deterministic algorithms via tensor unfolding and partition}\label{sec:alg1}

We first introduce the only known approximation bound of the tensor nuclear $p$-norm via vector fibers. It is essentially the \li{following} result by Chen and Li~\cite{chen2020tensor}.
\begin{lemma}[{\cite[Proposition 4.3]{chen2020tensor}}]\label{thm:vectorbound}
For any $\TT\in\R^{n_1\times n_2\times \dots\times n_d}$ and $p\in[1,\infty]$, denote $\bt_{i_1i_2\dots i_{d-1}}\in\R^{n_d}$ to be the mode-$d$ fiber of $\TT$ by fixing the mode-$k$ index to be $i_k$ for $k=1,2,\dots,d-1$. One has
\begin{equation}
  \label{eq:vectorbound}
  \|\TT\|_p:= \left(\sum_{i_1=1}^{n_1}\sum_{i_2=1}^{n_2}\dots\sum_{i_{d}=1}^{n_{d}}\left|t_{i_1i_2\dots i_{d}}\right|^p\right)^{1/p}
  \le \|\TT\|_{p_*} \le \sum_{i_1=1}^{n_1}\sum_{i_2=1}^{n_2}\dots\sum_{i_{d-1}=1}^{n_{d-1}}\|\bt_{i_1i_2\dots i_{d-1}}\|_p.
\end{equation}
\end{lemma}
It is obvious that the lower and upper bounds of~\eqref{eq:vectorbound} can be computed in polynomial time. Noticing that $$\|\TT\|_p= \left(\sum_{i_1=1}^{n_1}\sum_{i_2=1}^{n_2}\dots\sum_{i_{d-1}=1}^{n_{d-1}}{\|\bt_{i_1i_2\dots i_{d-1}}\|_p}^p\right)^{1/p},
$$
the lower and upper bounds are actually $\ell_p$-norm and $\ell_1$-norm of an $\prod_{k=1}^{d-1}n_k$-dimensional vector, respectively. By the bounds between $\ell_p$-norms (Lemma~\ref{lma:lp-norm-equiv}), $\|\TT\|_p$ provides an approximation bound of $1/\prod_{k=1}^{d-1}\sqrt[q]{n_k}$ for $\|\TT\|_{p_*}$ when $p\in[1,\infty]$. The result is simple but is currently the only possible method when $p\in(1,2)$.

To proceed with better approximation bounds, we have to deal with the matrix nuclear $p$-norms. In fact, Chen and Li~\cite{chen2020tensor} did provide better approximation frameworks of the tensor nuclear $p$-norm via tensor unfoldings and partitions.
\begin{lemma}[{\cite[Theorem 4.7]{chen2020tensor}}]\label{lma:chen2020}
  Let $\TT\in\R^{n_1\times n_2\times \dots\times n_d}$ and $p\in[1,\infty]$. Let $\left\{\BI_1,\BI_2\right\}$ be a partition of $\{1,2,\dots,d\}$ and choose any $i\in\BI_1$ and $j\in\BI_2$. Denote $\Mat(\TT)$ to be the matrix unfolding of $\TT$ by combining modes of $\BI_1$ into the row index and modes of $\BI_2$ into the column index, i.e., a $(\prod_{k\in\BI_1} n_k) \times (\prod_{k\in\BI_2} n_k)$ matrix. Consider the set of matrix slices of $\TT$ obtained by fixing every mode index except modes $i$ and $j$, i.e., a set of $\prod_{1\le k\le d,\,k\neq i,j}n_k$ number of $n_i\times n_j$ matrices and denote $\bt_{p_*}\in\R^{\prod_{1\le k\le d,\,k\neq i,j}n_k}$ to be the vector whose entries are the nuclear $p$-norms of this set of matrix slices. One has
  \begin{equation} \label{eq:matbound}
    \left\|\bt_{p_*}\right\|_p
    \le \|\Mat(\TT)\|_{p_*}
    \le \|\TT\|_{p_*}
    \le \left\| \bt_{p_*} \right\|_p \prod_{1\le k\le d,\,k\neq i,j}{n_k}^{\frac{1}{q}}
    \le \|\Mat(\TT)\|_{p_*} \prod_{1\le k\le d,\,k\neq i,j}{n_k}^{\frac{1}{q}}.
  \end{equation}
\end{lemma}

Although the above bounds are tighter than~\eqref{eq:vectorbound}, they cannot be computed in polynomial time because computing the matrix nuclear $p$-norm is NP-hard. Well, the $\Omega(1)$-approximation bound of the matrix nuclear $p$-norm developed in Section~\ref{sec:matrix-p-norm} enlightens the bounds in~\eqref{eq:matbound}. We now propose two polynomial-time algorithms to approximate the tensor nuclear $p$-norm. The first algorithm is fairly straightforward, i.e., computing $\|\cdot\|_{p_u}$-norm of a proper matrix unfolding of $\TT$.
\begin{algorithm}[!h]
\begin{algorithmic}[1]
\REQUIRE A tensor $\TT\in\R^{n_1 \times n_2 \times \dots \times n_d}$, a constant $p\in\BQ\cap(2,\infty)$ and a partition $\{\BI_1,\BI_2\}$ of $\{1,2,\dots,d\}$ with $d-1\in\BI_1$ and $d\in\BI_2$.
\ENSURE An approximation of $\|\TT\|_{p_*}$.
\STATE Unfold the tensor $\TT$ to $\Mat(\TT)\in\R^{\left(\prod_{k\in\BI_1} n_k\right) \times \left(\prod_{k\in\BI_2} n_k\right)}$ by combining modes of $\BI_1$ into the row index and modes of $\BI_2$ into the column index;
\STATE Compute $\|\Mat(\TT)\|_{p_u}$, i.e., the optimal value of the following SDP
\begin{equation*}\label{opt:d-tensor-matric}
    \begin{array}{ll}
    \max & \langle \Mat(\TT),Z\rangle    \\
    \st   &u_1+ u_2+\theta_p\sum_{i=1}^{\prod_{k\in\BI_1} n_k + \prod_{k\in\BI_2}n_k} t_i\le 1\\
    & (v_i, u_1,t_i)\in\BK^3_p\quad i=1,2,\dots, \prod_{k\in\BI_1} n_k \\
    & (v_i, u_2,t_i)\in\BK^3_p\quad i=\prod_{k\in\BI_1} n_k+1,\prod_{k\in\BI_1} n_k+2,\dots, \prod_{k\in\BI_1} n_k + \prod_{k\in\BI_2}n_k\\ 
    & 
    \Diag(\bv)\succeq \begin{pmatrix} O & Z/2 \\Z^\T/2 & O\end{pmatrix};
    \end{array}
    \end{equation*} 
\RETURN $\|\Mat(\TT)\|_{p_u}$.
\end{algorithmic}
\caption{Approximating the tensor nuclear $p$-norm based on matrix unfolding}
\label{alg:matricization}
\end{algorithm}

We see from~\eqref{eq:matbound} that $\|\Mat(\TT)\|_{p_*}$ is already a better approximation of $\|\TT\|_{p_*}$ with an approximation bound $1/\prod_{1\le k\le d,\,k\neq i,j}\sqrt[q]{n_k}$. By choosing the largest two $n_k$'s, i.e., $n_{d-1}$ and $n_d$ as in Algorithm~\ref{alg:matricization}, the bound attains the best one $1/\prod_{k=1}^{d-2}\sqrt[q]{n_k}$. Combining it with the polynomial-time method to approximate $\|\Mat(\TT)\|_{p_*}$ using $\|\Mat(\TT)\|_{p_u}$ with approximation bound $1/\delta_G$ (Theorem~\ref{thm:KG-matrix-nuclear-pnorm}), Algorithm~\ref{alg:matricization} immediately leads to the following approximation bound of $\|\TT\|_{p_*}$.
\begin{proposition}
  For any $\TT\in\R^{n_1 \times n_2 \times \dots \times n_d}$, $p\in \BQ\cap(2,\infty)$ and partition $\{\BI_1,\BI_2\}$ of $\{1,2,\dots,d\}$ with $d-1\in\BI_1$ and $d\in\BI_2$, Algorithm~\ref{alg:matricization} is a deterministic polynomial-time algorithm whose output $\operatorname{Alg}_{\,\ref{alg:matricization}}(\|\TT\|_{p_*})$ satisfies
$$
\left(\frac{1}{\delta_G}\prod_{k=1}^{d-2} \frac{1}{\sqrt[\leftroot{-2}\uproot{2}q]{n_k}}\right)\|\TT\|_{p_*}\le \operatorname{Alg}_{\,\ref{alg:matricization}}(\|\TT\|_{p_*})\le \|\TT\|_{p_*}.
$$
\end{proposition}

The second approach to approximate $\|\TT\|_{p_*}$ is using the quantity $\|\bt_{p_*}\|_p$ in~\eqref{eq:matbound}. The main effort is to compute entries of $\bt_{p_*}$. Every entry is the nuclear $p$-norm of a matrix and we use $\|\cdot\|_{p_u}$ to approximate it.
\begin{algorithm}[!h]
\begin{algorithmic}[1]
\REQUIRE A tensor $\TT\in\R^{n_1 \times n_2 \times \dots \times n_d}$ and a constant $p\in\BQ\cap(2,\infty)$.
\ENSURE An approximation of $\|\TT\|_{p_*}$.
\STATE Partition $\TT$ by fixing every mode index except modes $d-1$ and $d$ to obtain a $\prod_{k=1}^{d-2}n_k$ number of $n_{d-1}\times n_d$ matrices, namely $\{T_1,T_2,\dots,T_{\prod_{k=1}^{d-2}n_k}\}\subseteq\R^{n_{d-1}\times n_d}$;
\STATE For every $i=1,2,\dots,\prod_{k=1}^{d-2}n_k$, compute $u_i=\|T_i\|_{p_u}$, i.e., the optimal value of the following SDP
    \begin{equation*} \label{opt:d-tensor-partition}
    \begin{array}{ll}
    \max & \langle T_i,Z\rangle    \\
    \st   &u_1+ u_2+\theta_p\sum_{i=1}^{n_{d-1}+n_d} t_i\le 1\\
    & (v_i, u_1,t_i)\in\BK^3_p\quad i=1,2,\dots, n_{d-1} \\
    & (v_i, u_2,t_i)\in\BK^3_p\quad i=n_{d-1}+1,n_{d-1}+2,\dots, n_{d-1} + n_d\\
    & 
    \Diag(\bv)\succeq \begin{pmatrix} O & Z/2 \\Z^\T/2 & O\end{pmatrix};
    \end{array}
    \end{equation*} 
\RETURN $\|(u_1,u_2,\dots,u_{\prod_{k=1}^{d-2}n_k})\|_p$.
\end{algorithmic}
\caption{Approximating the tensor nuclear $p$-norm based on matrix partition}
\label{alg:partition}
\end{algorithm}

We have exactly the same theoretical approximation bound to that of Algorithm~\ref{alg:matricization}.
\begin{proposition}
  For any $\TT\in\R^{n_1 \times n_2 \times \dots \times n_d}$ and $p\in \BQ\cap(2,\infty)$, Algorithm~\ref{alg:partition} is a deterministic polynomial-time algorithm whose output $\operatorname{Alg}_{\,\ref{alg:partition}}(\|\TT\|_{p_*})$ satisfies
$$
\left(\frac{1}{\delta_G}\prod_{k=1}^{d-2} \frac{1}{\sqrt[\leftroot{-2}\uproot{2}q]{n_k}}\right)\|\TT\|_{p_*}\le \operatorname{Alg}_{\,\ref{alg:partition}}(\|\TT\|_{p_*})\le \|\TT\|_{p_*}.
$$
\end{proposition}
\begin{proof}
    Denote $\bu=(u_1,u_2,\dots,u_{\prod_{k=1}^{d-2}n_k})^{\T}$. By Theorem~\ref{thm:KG-matrix-nuclear-pnorm}, one has $\|T_i\|_{p_*}/\delta_G\le u_i=\|T\|_{p_u} \le \|T_i\|_{p_*}$ for $i=1,2,\dots,\prod_{k=1}^{d-2}n_k$. This implies that $\|\bt_{p_*}\|_p/\delta_G\le\|\bu\|_p\le\|\bt_{p_*}\|_p$. The result follows immediately from the bounds $\left\|\bt_{p_*}\right\|_p   \le \|\TT\|_{p_*} \le \left\| \bt_{p_*} \right\|_p \prod_{k=1}^{d-2} {n_k}^{1/q}$ in~\eqref{eq:matbound}.
\end{proof}

Both Algorithm~\ref{alg:matricization} and Algorithm~\ref{alg:partition} have the same theoretical performance guarantee but their complexities are different. In particular, Algorithm~\ref{alg:matricization} needs to solve only one SDP with dimension $O(\prod_{k=1}^d n_k)$ while Algorithm~\ref{alg:partition} needs to solve $\prod_{k=1}^{d-2}n_k$ number of SDPs each with dimension $O(n_{d-1} n_d)$. Therefore, Algorithm~\ref{alg:partition} runs faster than Algorithm~\ref{alg:matricization} in general, especially when $d$ is large. 

\subsection{Deterministic approximation via $\ell_p$-sphere covering}\label{sec:alg2}

The approximation bounds of the tensor nuclear $p$-norm obtained in Section~\ref{sec:alg1} cannot be improved via manipulating matrices as they are restricted by the bounds in Lemma~\ref{lma:chen2020}. In order to make improvement so as to match the best approximation bound for the tensor spectral $p$-norm, we now apply the $\ell_p$-sphere covering developed in Section~\ref{sec:hitting-sets}. It also needs to adopt certain reformulation and convex optimization proposed in~\cite{hu2022complexity} with some extra treatments.

To explain the main idea, let us first focus on tensors of order three, i.e., $\TT\in\R^{n_1\times n_2\times n_3}$. By the duality in Lemma~\ref{lma:norm-duality}, one has
\begin{align}\label{eq:nuclear-pnorm-pqnorm}
    \|\TT\|_{p_*}&=\max \left\{\langle\TT, \mathcal{Z}\rangle:\|\mathcal{Z}\|_{p_\sigma} \le 1\right\}\notag
    \\&=\max\left\{\langle\TT,\Z\rangle:\langle\Z,\bx\otimes\by\otimes\bz\rangle\le 1 \text{ for all } \bx\in\BS_p^{n_1},\,\by\in\BS_p^{n_2},\,\bz\in\BS_p^{n_3}\right\}\notag
    \\&=\max\left\{\langle\TT,\Z\rangle:\langle\Z\times_1\bx,\by\otimes\bz\rangle\le 1 \text{ for all } \bx\in\BS_p^{n_1},\,\by\in\BS_p^{n_2},\,\bz\in\BS_p^{n_3}\right\}\notag
    \\&=\max\left\{\langle\TT,\Z\rangle:\|\Z\times_1\bx\|_{p_\sigma}\le 1\text{ for all }\bx\in\BS_p^{n_1}\right\}.
\end{align}
By noticing that $\|A\|_{p_v}/\delta_G  \le \|A\|_{p_\sigma} \le \|A\|_{p_v}$ in Lemma~\ref{prop:matrix-equi-pq-vecp}, we see that 
$$\max\left\{\langle\TT,\Z\rangle:\|\Z\times_1\bx\|_{p_v}\le 1\text{ for all }\bx\in\BS_p^{n_1}\right\}$$
becomes a restriction of~\eqref{eq:nuclear-pnorm-pqnorm} and
$$\max\left\{\langle\TT,\Z\rangle:\|\Z\times_1\bx\|_{p_v}/\delta_G\le 1\text{ for all }\bx\in\BS_p^{n_1}\right\}=\delta_G\max\left\{\langle\TT,\Z\rangle:\|\Z\times_1\bx\|_{p_v}\le 1\text{ for all }\bx\in\BS_p^{n_1}\right\}$$
becomes a relaxation of~\eqref{eq:nuclear-pnorm-pqnorm}. Therefore, we obtain
\begin{equation}\label{eq:equi-pq-vecp}
\|\TT\|_{p_*}/\delta_G
\le \max\left\{\langle\TT,\Z\rangle:\|\Z\times_1\bx\|_{p_v}\le 1 \text{ for all }\bx\in\BS_p^{n_1}\right\} \le \|\TT\|_{p_*}.
\end{equation}
It suffices to focus on $\max\{\langle\TT,\Z\rangle:\|\Z\times_1\bx\|_{p_v}\le 1 \text{ for all }\bx\in\BS_p^{n_1}\}$. In fact, by the duality theory in Corollary~\ref{lma:dual-SDP}, it is easy to verify that
\begin{equation}\label{eq:nuclear-pnorm-SDP}
    \begin{array}{lllll}
\max & \langle\TT,\Z\rangle &= &\max & \langle\TT,\Z\rangle \\
\st  &\|\Z\times_1\bx\|_{p_v}\le 1 \quad \bx\in\BS_p^{n_1} & &\st  &  u_1^{\bx}+ u_2^{\bx}+\theta_p\sum_{i=1}^{n_2+n_3} t_i^{\bx}\le1 \quad  \bx\in\BS_p^{n_1}  \\
&& && (v_i^{\bx}, u_1^{\bx},t_i^{\bx})\in\BK^3_p\quad \bx\in\BS_p^{n_1},\,i=1,2,\dots,n_2 \\
&& && (v_i^{\bx}, u_2^{\bx},t_i^{\bx})\in\BK^3_p\quad \bx\in\BS_p^{n_1},\, i=n_2+1,n_2+2,\dots, n_2+n_3\\
&& && 
    \Diag(\bv^{\bx})\succeq \begin{pmatrix} O & \Z\times_1\bx/2 \\(\Z\times_1\bx)^\T/2 & O\end{pmatrix} \quad \bx\in\BS_p^{n_1}.
    \end{array}
    \end{equation}
Everything looks doable except the infinite number of constraints brought by $\bx\in\BS_p^{n_1}$. 
Now, the hitting sets developed in Section~\ref{sec:hitting-sets} are helpful. This is essentially the idea in~\cite[Section~3.2]{he2023approx}, relaxing~\eqref{eq:nuclear-pnorm-SDP} by replacing $\BS_p^{n_1}$ with a hitting set that consists of a polynomial number of vectors. We summarize the polynomial-time algorithm to approximate $\|\TT\|_{p_*}$ as well as its theoretical guarantee based on $\ell_p$-sphere covering below.

\begin{algorithm}[!h]
\begin{algorithmic}[1]
\REQUIRE A tensor $\TT\in\R^{n_1\times n_2\times n_3}$, a constant $p\in\BQ\cap(2,\infty)$ and a hitting set $\BH^{n_1}\in\BT^{n_1}_p\left(\tau, O({n_1}^{\alpha_1})\right)$.
\ENSURE An approximation of $\|\TT\|_{p_*}$.
\STATE Solve the following SDP
\begin{equation}\label{opt:three-tensor-sdp}
    \begin{array}{ll}
\max & \langle\TT,\Z\rangle \\
\st  &  u_1^{\bx}+ u_2^{\bx}+\theta_p\sum_{i=1}^{n_2+n_3} t_i^{\bx}\le1 \quad  \bx\in\BH^{n_1}  \\
& (v_i^{\bx}, u_1^{\bx},t_i^{\bx})\in\BK^3_p\quad \bx\in\BH^{n_1},\,i=1,2,\dots,n_2 \\
& (v_i^{\bx}, u_2^{\bx},t_i^{\bx})\in\BK^3_p\quad \bx\in\BH^{n_1},\, i=n_2+1,n_2+2,\dots, n_2+n_3\\
& 
    \Diag(\bv^{\bx})\succeq \begin{pmatrix} O & \Z\times_1\bx/2 \\(\Z\times_1\bx)^\T/2 & O\end{pmatrix} \quad \bx\in\BH^{n_1}
    \end{array}
\end{equation}
to obtain its optimal value $u$;
\RETURN $\tau u$.
\end{algorithmic}
\caption{Approximating the tensor nuclear $p$-norm based on $\ell_p$-sphere covering}
\label{alg:alg2}
\end{algorithm}

\begin{theorem}\label{thm:approx-ratio-three} For any $\TT\in\R^{n_1 \times n_2 \times n_3}$, $p\in \BQ\cap(2,\infty)$ and hitting set $\BH^{n_1}\in\BT^{n_1}_p\left(\tau, O({n_1}^{\alpha_1})\right)$,
Algorithm~\ref{alg:alg2} is a deterministic polynomial-time algorithm whose output $\operatorname{Alg}_{\,\ref{alg:alg2}}(\|\TT\|_{p_*})$ satisfies
$$
\frac{\tau}{\delta_G}\|\TT\|_{p_*}\le \operatorname{Alg}_{\,\ref{alg:alg2}}(\|\TT\|_{p_*})\le \|\TT\|_{p_*}.
$$
\end{theorem}
\begin{proof}
  Since the cardinality of $\BH^{n_1}$ is $O({n_1}^{\alpha_1})$, the SDP in~\eqref{opt:three-tensor-sdp} does have a polynomial number of constraints and so the algorithm runs in polynomial time. To show the approximation bounds, let $\Y$ be an optimal solution of~\eqref{opt:three-tensor-sdp}.

    First, the lower bound can be shown by
    $$
    \|\TT\|_{p_*}/ \delta_G\le 
    \max\left\{\langle\TT,\Z\rangle:\|\Z\times_1\bx\|_{p_v}\le 1 \text{ for all } \bx\in\BS_p^{n_1}\right\}
    \le\langle\TT,\Y\rangle=\operatorname{Alg}_{\,\ref{alg:alg2}}(\|\TT\|_{p_*})/\tau,
    $$
    where the first inequality is due to the lower bound of~\eqref{eq:equi-pq-vecp} while the second inequality is due to~\eqref{eq:nuclear-pnorm-SDP} and that~\eqref{opt:three-tensor-sdp} is a relaxation of~\eqref{eq:nuclear-pnorm-SDP}.
    
    Next, as $\BH^{n_1}$ is a $\tau$-hitting set, we have that
    \begin{align*}
    \max_{\bx \in \BH^{n_1},\,\by\in\BS_p^{n_2},\,\bz\in\BS_p^{n_3}} \langle\Y\times_1\bx,\by\otimes\bz\rangle
    &=\max_{\by\in\BS_p^{n_2},\,\bz\in\BS_p^{n_3}} \max_{\bx \in \BH^{n_1}} \left\langle\Y\times_2\by\times_3\bz,\bx\right\rangle\\
    &\ge\max_{\by\in\BS_p^{n_2},\,\bz\in\BS_p^{n_3}} \tau\|\Y\times_2\by\times_3\bz\|_q\\
    &=\tau\max_{\bx\in\BS_p^{n_1},\,\by\in\BS_p^{n_2},\,\bz\in\BS_p^{n_3}} \left\langle\Y\times_2\by\times_3\bz,\bx\right\rangle\\
    &=\tau\|\Y\|_{p_\sigma}.
    \end{align*}
    By the feasibility of $\Y$ to~\eqref{opt:three-tensor-sdp} and Corollary~\ref{lma:dual-SDP}, we have that $\|\Y\times_1\bx\|_{p_v}\le 1$ for all $\bx\in\BH^{n_1}$. Therefore,
    $$
    \|\tau \Y\|_{p_\sigma} \le \max_{\bx \in \BH^{n_1},\by\in\BS_p^{n_2},\bz\in\BS_p^{n_3}} \langle\Y\times_1\bx,\by\otimes\bz\rangle=\max _{\bx \in \BH^{n_1}}\|\Y\times_1\bx\|_{p_\sigma} \le \max _{\bx \in \BH^{n_1}}\|\Y\times_1\bx\|_{p_v} \le 1,
    $$
    implying that $\tau \Y$ is feasible to $\max \left\{\langle\TT, \mathcal{Z}\rangle:\|\mathcal{Z}\|_{p_\sigma} \le 1\right\}=\|\TT\|_{p_*}$. Finally, we have that
    $$
    \|\TT\|_{p_*}=\max \left\{\langle\TT, \mathcal{Z}\rangle:\|\mathcal{Z}\|_{p_\sigma} \le 1\right\} \ge\langle\TT, \tau \Y\rangle=\operatorname{Alg}_{\,\ref{alg:alg2}}(\|\TT\|_{p_*}),
    $$
    implying the upper bound.
\end{proof}

It is straightforward to generalize Algorithm~\ref{alg:alg2} to approximate the nuclear $p$-norm of higher-order tensors; see Algorithm~\ref{alg:alg3}. The theoretical approximation bound is summarized in Theorem~\ref{thm:ratio-higher-order}, whose proof is similar to that of Theorem~\ref{thm:approx-ratio-three} and is thus omitted.
\begin{algorithm}[!h]
\begin{algorithmic}[1]
\REQUIRE A tensor $\TT\in\R^{n_1 \times n_2 \times \dots \times n_d}$, a constant $p\in\BQ\cap(2,\infty)$ and $d-2$ hitting sets $\BH^{n_k}\in\BT^{n_k}_p(\tau_k, O({n_k}^{\alpha_k}))$ for $k=1,2,\dots,d-2$.
\ENSURE An approximation of $\|\TT\|_{p_*}$.
\STATE Solve the following SDP
\begin{equation*}
    \begin{array}{ll}
\max & \langle\TT,\Z\rangle \\
\st  &  u_1^{\bx}+ u_2^{\bx}+\theta_p\sum_{i=1}^{n_{d-1}+n_d} t_i^{\bx}\le 1   \quad  \bx\in\BH^{n_1}\vee\BH^{n_2}\vee\dots\vee\BH^{n_{d-2}} \\
& (v_i^{\bx}, u_1^{\bx},t_i^{\bx})\in\BK^3_p\quad \bx\in\BH^{n_1}\vee\BH^{n_2}\vee\dots\vee\BH^{n_{d-2}},\, i=1,2,\dots,n_{d-1} \\
& (v_i^{\bx}, u_2^{\bx},t_i^{\bx})\in\BK^3_p\quad \bx\in\BH^{n_1}\vee\BH^{n_2}\vee\dots\vee\BH^{n_{d-2}},\, i=n_{d-1}+1,n_{d-1}+2,\dots, n_{d-1}+n_d\\
& 
    \Diag(\bv)\succeq \frac{1}{2}\begin{pmatrix} O & \Z\times_1\bx_1\dots\times_{d-2}\bx_{d-2} \\(\Z\times_1\bx_1\dots\times_{d-2}\bx_{d-2})^\T & O\end{pmatrix} \quad  \bx\in\BH^{n_1}\vee\BH^{n_2}\vee\dots\vee\BH^{n_{d-2}}\\
&\bx=\bx_1\vee\bx_2\vee\dots\vee\bx_{d-2}
    \end{array}
\end{equation*}
to obtain its optimal value $u$;
\RETURN $u\prod_{k=1}^{d-2} \tau_k$.
\end{algorithmic}
\caption{Approximating the tensor nuclear $p$-norm based on $\ell_p$-sphere covering}
\label{alg:alg3}
\end{algorithm}

\begin{theorem}\label{thm:ratio-higher-order}
For any $\TT\in\R^{n_1 \times n_2 \dots \times n_d}$, $p\in \BQ\cap(2,\infty)$ and hitting sets $\BH^{n_k}\in\BT^{n_k}_p(\tau_k, O({n_k}^{\alpha_k}))$ for $k=1,2,\dots,d-2$,
Algorithm~\ref{alg:alg3} is a deterministic polynomial-time algorithm whose output $\operatorname{Alg}_{\,\ref{alg:alg3}}(\|\TT\|_{p_*})$ satisfies
$$
\left(\frac{1}{\delta_G}\prod_{k=1}^{d-2}\tau_k\right)\|\TT\|_{p_*}\le \operatorname{Alg}_{\,\ref{alg:alg3}}(\|\TT\|_{p_*}) \le \|\TT\|_{p_*}.
$$
\end{theorem}

As a remark, the cardinality of any hitting set in Algorithm~\ref{alg:alg3} must be a polynomial function of the problem dimension in order to ensure the algorithm running in polynomial time. By directly applying Theorem~\ref{thm:ratio-higher-order} with the hitting sets $\BH_{1}^n(\alpha,\beta)$ in Section~\ref{sec:worst-hitting} and  $\BH_{2}^n(\alpha,\beta)$ in Section~\ref{sec:brieden-hitting-set}, we have the following improved approximation bounds for the tensor nuclear $p$-norm.

\begin{corollary}\label{cor:nonbridge-gap-determ}
By choosing $\BH^{n_k}=\BH_{1}^{n_k}(\alpha,\beta)$ for $k=1,2,\dots,d-2$ in Algorithm~\ref{alg:alg3}, we have
$$
\Omega\left(\prod_{k=1}^{d-2} \sqrt[\leftroot{-2}\uproot{2}q]{\frac{\ln n_k}{n_k}}\right)\|\TT\|_{p_*}\le \operatorname{Alg}_{\,\ref{alg:alg3}}(\|\TT\|_{p_*}) \le \|\TT\|_{p_*}
$$
for any $\TT\in\R^{n_1 \times n_2 \times\dots \times n_d}$ and $p\in \BQ\cap(2,\infty)$.
\end{corollary}

\begin{corollary}\label{cor:bridge-gap-determ}
By choosing $\BH^{n_k}=\BH_{2}^{n_k}(\alpha,\beta)$ for $k=1,2,\dots,d-2$ in Algorithm~\ref{alg:alg3}, we have
$$
\Omega\left(\prod_{k=1}^{d-2} \frac{\sqrt[\leftroot{-2}\uproot{2}p]{\ln{n_k}}}{\sqrt{n_k}}\right)\|\TT\|_{p_*}\le \operatorname{Alg}_{\,\ref{alg:alg3}}(\|\TT\|_{p_*}) \le \|\TT\|_{p_*}
$$
for any $\TT\in\R^{n_1 \times n_2 \times\dots \times n_d}$ and $p\in \BQ\cap(2,\infty)$.
\end{corollary}

In terms of the approximation bounds, Corollary~\ref{cor:bridge-gap-determ} improves Corollary~\ref{cor:nonbridge-gap-determ} while the latter improves the bounds by manipulating matrices in Section~\ref{sec:alg1}. 
Well, the most significant result is that the best bound, $\Omega(\prod_{k=1}^{d-2} \sqrt[p]{\ln{n_k}}/\sqrt{n_k})$ in Corollary~\ref{cor:bridge-gap-determ}, matches the best-known approximation bound of its dual norm (the tensor spectral $p$-norm)~\cite[Theorem~7]{hou2014hardness} by a deterministic polynomial-time algorithm. 

\subsection{Randomized approximations via $\ell_p$-sphere covering}\label{sec:alg3}

As shown in Section~\ref{sec:rand-hitting-set}, with the help of randomization, the randomized hitting set $\BH_3^n(\epsilon)$ enjoys the best possible hitting ratio. This can also be applied to Algorithm~\ref{alg:alg3} to improve the approximation bound of the tensor nuclear $p$-norm. For every $\ell_p$-sphere, one can sample a sufficiently large number of vectors on $\BS_p^n$ to construct a hitting set with high probability by Theorem~\ref{thm:rand-hitting-set}. Once all $d-2$ hitting sets are constructed successfully, they can be applied to Algorithm~\ref{alg:alg3} directly. However, we need to take a little bit care of the overall probability. As $d-2$ hitting sets need to be input for Algorithm~\ref{alg:alg3} independently, if the success rate for every randomized hitting set is set as $1-\epsilon$, then the overall success rate would be $(1-\epsilon)^{d-2}\ge1-\epsilon(d-2)$. Therefore, in order to make the overall success rate to be $1-\epsilon$, it suffices to set the success rate of each hitting set to be $1-\frac{\epsilon}{d-2}$. Here are the randomized algorithm and its theoretical guarantee. 

\begin{algorithm}[!h]
\begin{algorithmic}[1]
\REQUIRE A tensor $\TT\in\R^{n_1\times n_2\times \dots\times n_d}$, a constant $p\in\BQ\cap(2,\infty)$ and a tolerance $\epsilon\in(0,1)$.
\ENSURE An approximation of $\|\TT\|_{p_*}$ with probability at least $1-\epsilon$.
\STATE For every $k=1,2,\dots,d-2$, generate $\lceil\delta_3{n_k}^{\delta_2}( (\frac{1}{2}+\frac{1}{q}){n_k} \ln {n_k}+\ln \frac{d-2}{\epsilon})\rceil$ samples of random vectors independently and evenly on $\BS^{n_k}_p$ to form a hitting set $\BH_3^{n_k}(\frac{d-2}{\epsilon})$.
\RETURN $\operatorname{Alg}_{\,\ref{alg:alg3}}(\|\TT\|_{p_*})$ with $\BH^{n_k}=\BH_3^{n_k}(\frac{d-2}{\epsilon})$ for $k=1,2,\dots,d-2$.
\end{algorithmic}
\caption{Approximating the tensor nuclear $p$-norm based on randomized $\ell_p$-sphere covering}
\label{alg:rand-alg}
\end{algorithm}

\begin{theorem}
    For any $\TT\in\R^{n_1 \times n_2 \dots \times n_d}$, $p\in\BQ\cap(2,\infty)$ and $\epsilon\in(0,1)$, Algorithm~\ref{alg:rand-alg} is a randomized polynomial-time algorithm whose output $\operatorname{Alg}_{\,\ref{alg:rand-alg}}(\|\TT\|_{p_*})$ satisfies
    $$
    \Omega\left(\prod_{k=1}^{d-2} \sqrt{\frac{\ln{n_k}}{n_k}}\right)\|\TT\|_{p_*} \le \operatorname{Alg}_{\,\ref{alg:rand-alg}}(\|\TT\|_{p_*})
    \text{ and }
    \Prob\left\{\operatorname{Alg}_{\,\ref{alg:rand-alg}}(\|\TT\|_{p_*})\le\|\TT\|_{p_*}\right\}\ge 1-\epsilon.
    $$
\end{theorem}

The approximation bound, $\Omega(\prod_{k=1}^{d-2} \sqrt{\ln{n_k}/n_k})$, matches the best-known approximation bound for the tensor spectral $p$-norm by a randomized algorithm~\cite[Theorem~8]{hou2014hardness}. Together with the approximation bound in Corollary~\ref{cor:bridge-gap-determ}, the approximation bounds of the tensor nuclear $p$-norm derived in this paper, exactly match \li{those} of the tensor spectral $p$-norm, no matter by deterministic polynomial-time algorithms or randomized ones.

\subsection{\li{Numerical approximation results}}\label{sec:numerical}
{\color{black} Let us evaluate the numerical performance of the approximation methods for the nuclear $p$-norm studied in this section (cf.~Table~\ref{tab:alg-ratios}). All the experiments\footnote{The source code of the experiments is publicly available at \url{https://github.com/seemjwguan/lp-sphere-covering}.} are implemented in MATLAB 2020b on a Ubuntu computer with an Intel Core i9-10900K CPU and 64GB of memory. We use MATLAB Tensor Toolbox 2.6~\cite{TTB_Software} whenever tensor operations are called and use MOSEK~\cite{mosek} to solve the SDPs embedded in some of these algorithms. 

All the test instances in the experiments are set as $p=d=3$. To implement the randomized algorithm, i.e., Algorithm~\ref{alg:rand-alg}, we set the failure probability $\epsilon=0.05$ and use the existing estimates in~\cite[Remark~1]{hou2014hardness} for the various universal constants involved in. As the estimated constant $\delta_3\approx 1,926$ is too large to make the resulting SDP practically solvable, we decrease it to $20$. For the uniform sampling on the $\ell_p$-spheres, i.e., the distribution introduced in Lemma~\ref{lma:lp-sampling}, we use the inverse transform sampling method; see e.g.,~\cite[Section~2.2]{devroye2018non}. For the hitting set $\BH_{H}^n(\alpha, \beta)$ defined in Algorithm~\ref{alg:alg1}, we set $\alpha=\frac{5+\sqrt{33}}{2}$ and $\beta=\alpha+1$, leading to the largest hitting ratio when $p=3$.

In order to understand the true approximation performance of the methods, we need to identify some tensor instances whose nuclear $p$-norms can be computed exactly, instead of using their upper bounds. This turns out to be a nontrivial task. To begin with, we first observe the following spectral $p$-norm characterization of the so-called identity tensor. 
\begin{lemma}\label{lma:identity}
    If $n,d\in\BN$ and $p \in [2,\infty]$, then
    $$
    \left\|\sum_{i=1}^n{\be_i}^{\otimes d}\right\|_{p_\sigma}=\begin{cases}
        1 & p\le d \\
        n^{1-d/p} & p> d.
    \end{cases}
    $$
\end{lemma}
\begin{proof}
    We notice that $\sum_{i=1}^n{\be_i}^{\otimes d}$ is nonnegative and symmetric. Since $p \in [2,\infty]$, by~\cite[Corollary~5]{nikiforov2019p} (the $\ell_p$-version of Banach's theorem), we have
    $$
    \left\|\sum_{i=1}^n{\be_i}^{\otimes d}\right\|_{p_\sigma}
    =\max\left\{\sum_{i=1}^n {x_i}^d:\sum_{i=1}^n |x_i|^p = 1\right\}
    =\max\left\{\sum_{i=1}^n |x_i|^d:\sum_{i=1}^n |x_i|^p = 1\right\}.
    $$
    The desired result follows immediately from the bounds between $\ell_p$-norms (Lemma~\ref{lma:lp-norm-equiv}).
\end{proof}
With the help of Lemma~\ref{lma:identity}, we can single out a large class of nonnegative symmetric tensors whose nuclear $p$-norm can be determined by inspection.




\begin{corollary}\label{cor:random}
    If $n,d,r\in\BN$, $p=d\ge2$, $\lambda_i\ge0$ and $\bx_i\in\BS_p^n\cap\R_+^n$ for $i=1,2,\dots,r$, then
    $$
    \left\|\sum_{i=1}^r \lambda_i\, {\bx_i}^{\otimes d}\right\|_{p_*}=\sum_{i=1}^r\lambda_i.
    $$
\end{corollary}
\begin{proof}
    Denote $\TT=\sum_{i=1}^r \lambda_i\, {\bx_i}^{\otimes d}$. By the definition of nuclear $p$-norm in~\eqref{def:nuclear}, this decomposition obviously implies that $\|\TT\|_{p_*}\le\sum_{i=1}^r \lambda_i$ since $\lambda_i\ge0$ and $\|\bx_i\|_p=1$. On the other hand, we know from Lemma~\ref{lma:identity} that $\left\|\sum_{j=1}^n{\be_j}^{\otimes d}\right\|_{p_\sigma}=1$ as $p=d$. Therefore, by the duality between the spectral $p$-norm and nuclear $p$-norm (Lemma~\ref{lma:norm-duality}), we have
    $$
    \|\TT\|_{p_*}\ge\left\langle\TT,\sum_{j=1}^n{\be_j}^{\otimes d}\right\rangle=\sum_{i=1}^r \lambda_i\left\langle{\bx_i}^{\otimes d},\sum_{j=1}^n{\be_j}^{\otimes d}\right\rangle=\sum_{i=1}^r \lambda_i {\|\bx_i\|_d}^d =\sum_{i=1}^r \lambda_i {\|\bx_i\|_p}^p= \sum_{i=1}^r \lambda_i,
    $$
    where the second equality is due to the nonnegativity of $\bx_i$'s. This completes the proof.   
\end{proof}



All the tensor instances are generated using the construction in Corollary~\ref{cor:random} so that their nuclear $p$-norms can be computed directly for comparison. In particular, $\lambda_i$'s follow i.i.d.\ uniform distribution on $[0,1]$ and $\bx_i$'s are independent of each other. To construct each $\bx_i$, we first generate a vector whose entries follow i.i.d.\ uniform distribution on $[0,1]$ and normalize it to make $\|\bx_i\|_p=1$. Twenty instances are generated for each $n=3,5,7,10$ and each $r=1,2,3,4,5,10$ with $p=d=3$. 
We then run these algorithms and record their minimum, average and maximum approximation ratios as well as average running time. 
The experimental results, together with the theoretical approximation bounds in Table~\ref{tab:alg-ratios} for comparison, are shown in Tables~\ref{tab:n=3}--\ref{tab:n=10}. 
We summarize the key observations.
\begin{itemize}
    \item The numerical results verify, and in fact clearly outperform the theoretical bounds for all the proposed algorithms in this section.
    \item The approximation bounds obtained by Algorithm~\ref{alg:alg3} with $\BH_{2}^{n}$ and Algorithm~\ref{alg:rand-alg} perform the best. Although Algorithm~\ref{alg:alg3} with $\BH_{1}^{n}$ is also based on sphere covering, its performance is not competitive. This pretty much matches the worst hitting ratio compared with the other two methods based on sphere covering. 
    \item Comparing between Algorithm~\ref{alg:alg3} with $\BH_{2}^{n}$ and Algorithm~\ref{alg:rand-alg}, the approximation bounds of the former are worse when $n=3,5$ but better when $n$ goes large. Overall, the former consistently obtains excellent approximation bounds across almost all settings, and enjoys much lower running times than the latter.
    \item For rank-one tensor instances, the first three algorithms obtain the optimality since the vectorization, matricization and matrix partitioning of a rank-one tensor preserve the tensor nuclear $p$-norm~\cite[Proposition~2.4]{chen2020tensor}.
    However, the performance of these methods deteriorates quickly when $r$ goes large. 
    \item Although the algorithms based on sphere covering achieve better approximation bounds than the first three algorithms based on vectorization or matrix operation, they obviously run much slower. It is important for the practitioner to balance the trade-off between effectiveness and efficiency. 
\end{itemize}

\begin{table}[!ht]
\centering
\caption{Approximation bounds of the tensor nuclear $p$-norm by algorithms in~Table~\ref{tab:alg-ratios} when $p=d=3$ and $n=3$.}
\label{tab:n=3}
\resizebox{\textwidth}{!}{%
\begin{tabular}{|ll|c|c|c|c|c|c|}
\hline
\multicolumn{2}{|l|}{Algorithm}                                      & \cite[Prop.~4.3]{chen2020tensor} & Alg.~\ref{alg:matricization} & Alg.~\ref{alg:partition} & Alg.~\ref{alg:alg3} ($\BH_{1}^{n}$) & Alg.~\ref{alg:alg3} ($\BH_{2}^{n}$) & Alg.~\ref{alg:rand-alg} \\ \hline
\multicolumn{2}{|l|}{Theoretical bound}                                      & $\frac{1}{\sqrt[\leftroot{-2}\uproot{2}q]{n^2}}\approx 0.2311$ & \multicolumn{2}{c|}{$\frac{1}{\sqrt[\leftroot{-2}\uproot{2}q]{n}}\approx 0.4807$} &  $\sqrt[\leftroot{-2}\uproot{2}q]{\frac{\ln n}{n}}\approx 0.5119$ &  $\frac{\sqrt[\leftroot{-2}\uproot{2}p]{\ln{n}}}{\sqrt{n}}\approx 0.5957$ &  $\sqrt{\frac{\ln{n}}{n}}\approx 0.6051$ \\ \hline\hline
\multicolumn{1}{|l|}{\multirow{4}{*}{$r=1$}}  & Ave ratio & 1.0000                           & 1.0000                       & 1.0000                   & 0.6900                                & 0.9398                                & 0.9834                  \\
\multicolumn{1}{|l|}{}                        & Min ratio & 1.0000                           & 1.0000                       & 1.0000                   & 0.5522                                & 0.8325                                & 0.9534                  \\
\multicolumn{1}{|l|}{}                        & Max ratio & 1.0000                           & 1.0000                       & 1.0000                   & 0.9466                                & 0.9909                                & 0.9996                  \\ \cline{2-8} 
\multicolumn{1}{|l|}{}                        & Ave time  & 0.0002                           & 0.4846                       & 1.0309                   & 3.2367                                & 3.5194                                & 25.6326                 \\ \hline
\hline\multicolumn{1}{|l|}{\multirow{4}{*}{$r=2$}}  & Ave ratio & 0.8690                           & 0.9302                       & 0.8967                   & 0.6471                                & 0.9630                                & 0.9895                  \\
\multicolumn{1}{|l|}{}                        & Min ratio & 0.6612                           & 0.8190                       & 0.7215                   & 0.5715                                & 0.8575                                & 0.9629                  \\
\multicolumn{1}{|l|}{}                        & Max ratio & 0.9993                           & 0.9996                       & 0.9995                   & 0.7897                                & 0.9989                                & 0.9999                  \\ \cline{2-8} 
\multicolumn{1}{|l|}{}                        & Ave time  & 0.0002                           & 0.4832                       & 1.0225                   & 2.3310                                & 3.5282                                & 25.9662                 \\ \hline
\hline\multicolumn{1}{|l|}{\multirow{4}{*}{$r=3$}}  & Ave ratio & 0.8612                           & 0.9256                       & 0.8977                   & 0.6594                                & 0.9589                                & 0.9965                  \\
\multicolumn{1}{|l|}{}                        & Min ratio & 0.6120                           & 0.8010                       & 0.7073                   & 0.5660                                & 0.8792                                & 0.9770                  \\
\multicolumn{1}{|l|}{}                        & Max ratio & 0.9725                           & 0.9837                       & 0.9801                   & 0.8463                                & 0.9980                                & 1.0000                  \\ \cline{2-8} 
\multicolumn{1}{|l|}{}                        & Ave time  & 0.0002                           & 0.4868                       & 1.0267                   & 2.6966                                & 3.5286                                & 25.9420                 \\ \hline
\hline\multicolumn{1}{|l|}{\multirow{4}{*}{$r=4$}}  & Ave ratio & 0.7919                           & 0.8894                       & 0.8445                   & 0.6728                                & 0.9638                                & 0.9995                  \\
\multicolumn{1}{|l|}{}                        & Min ratio & 0.6507                           & 0.8101                       & 0.7119                   & 0.5621                                & 0.9282                                & 0.9962                  \\
\multicolumn{1}{|l|}{}                        & Max ratio & 0.9548                           & 0.9680                       & 0.9630                   & 0.8324                                & 0.9925                                & 1.0000                  \\ \cline{2-8} 
\multicolumn{1}{|l|}{}                        & Ave time  & 0.0002                           & 0.4915                       & 1.0119                   & 2.0906                                & 3.5098                                & 25.7583                 \\ \hline
\hline\multicolumn{1}{|l|}{\multirow{4}{*}{$r=5$}}  & Ave ratio & 0.7553                           & 0.8706                       & 0.8180                   & 0.6777                                & 0.9521                                & 0.9987                  \\
\multicolumn{1}{|l|}{}                        & Min ratio & 0.6486                           & 0.8123                       & 0.7149                   & 0.5792                                & 0.9136                                & 0.9816                  \\
\multicolumn{1}{|l|}{}                        & Max ratio & 0.8962                           & 0.9475                       & 0.9306                   & 0.7630                                & 0.9911                                & 1.0000                  \\ \cline{2-8} 
\multicolumn{1}{|l|}{}                        & Ave time  & 0.0003                           & 0.4876                       & 1.0448                   & 2.0373                                & 3.5228                                & 25.7078                 \\ \hline
\hline\multicolumn{1}{|l|}{\multirow{4}{*}{$r=10$}} & Ave ratio & 0.7148                           & 0.8486                       & 0.7896                   & 0.6912                                & 0.9419                                & 0.9998                  \\
\multicolumn{1}{|l|}{}                        & Min ratio & 0.5849                           & 0.7719                       & 0.6662                   & 0.5978                                & 0.9150                                & 0.9960                  \\
\multicolumn{1}{|l|}{}                        & Max ratio & 0.8827                           & 0.9368                       & 0.9191                   & 0.7740                                & 0.9902                                & 1.0000                  \\ \cline{2-8} 
\multicolumn{1}{|l|}{}                        & Ave time  & 0.0002                           & 0.4782                       & 1.0449                   & 2.0715                                & 3.5343                                & 25.6905                 \\ \hline
\end{tabular}%
}
\end{table}

\begin{table}[!ht]
\centering
\caption{Approximation bounds of the tensor nuclear $p$-norm by algorithms in~Table~\ref{tab:alg-ratios} when $p=d=3$ and $n=5$.}
\label{tab:n=5}
\resizebox{\textwidth}{!}{%
\begin{tabular}{|ll|c|c|c|c|c|c|}
\hline
\multicolumn{2}{|l|}{Algorithm}                                      & \cite[Prop.~4.3]{chen2020tensor} & Alg.~\ref{alg:matricization} & Alg.~\ref{alg:partition} & Alg.~\ref{alg:alg3} ($\BH_{1}^{n}$) & Alg.~\ref{alg:alg3} ($\BH_{2}^{n}$) & Alg.~\ref{alg:rand-alg} \\ \hline
\multicolumn{2}{|l|}{Theoretical bound}                                      & $\frac{1}{\sqrt[\leftroot{-2}\uproot{2}q]{n^2}}\approx 0.1170$ & \multicolumn{2}{c|}{$\frac{1}{\sqrt[\leftroot{-2}\uproot{2}q]{n}}\approx 0.3420$} &  $\sqrt[\leftroot{-2}\uproot{2}q]{\frac{\ln n}{n}}\approx 0.4697$ &  $\frac{\sqrt[\leftroot{-2}\uproot{2}p]{\ln{n}}}{\sqrt{n}}\approx 0.5241$ &  $\sqrt{\frac{\ln{n}}{n}}\approx 0.5674$ \\ \hline\hline
\multicolumn{1}{|l|}{\multirow{4}{*}{$r=1$}}  & Ave ratio & 1.0000                           & 1.0000                       & 1.0000                   & 0.5165                                & 0.9015                                & 0.8883                  \\
\multicolumn{1}{|l|}{}                        & Min ratio & 1.0000                           & 1.0000                       & 1.0000                   & 0.4364                                & 0.8454                                & 0.8334                  \\
\multicolumn{1}{|l|}{}                        & Max ratio & 1.0000                           & 1.0000                       & 1.0000                   & 0.6048                                & 0.9758                                & 0.9421                  \\ \cline{2-8} 
\multicolumn{1}{|l|}{}                        & Ave time  & 0.0002                           & 0.9192                       & 2.1826                   & 6.1271                                & 31.8846                               & 110.8029                \\ \hline
\hline\multicolumn{1}{|l|}{\multirow{4}{*}{$r=2$}}  & Ave ratio & 0.8575                           & 0.9238                       & 0.8949                   & 0.5202                                & 0.9444                                & 0.9330                  \\
\multicolumn{1}{|l|}{}                        & Min ratio & 0.6891                           & 0.8410                       & 0.7620                   & 0.4669                                & 0.8679                                & 0.8854                  \\
\multicolumn{1}{|l|}{}                        & Max ratio & 0.9981                           & 0.9989                       & 0.9987                   & 0.6202                                & 0.9958                                & 0.9747                  \\ \cline{2-8} 
\multicolumn{1}{|l|}{}                        & Ave time  & 0.0002                           & 0.9186                       & 2.3024                   & 6.1299                                & 32.1755                               & 113.0247                \\ \hline
\hline\multicolumn{1}{|l|}{\multirow{4}{*}{$r=3$}}  & Ave ratio & 0.7912                           & 0.8918                       & 0.8498                   & 0.5425                                & 0.9678                                & 0.9350                  \\
\multicolumn{1}{|l|}{}                        & Min ratio & 0.6898                           & 0.8219                       & 0.7588                   & 0.4727                                & 0.9293                                & 0.9070                  \\
\multicolumn{1}{|l|}{}                        & Max ratio & 0.9356                           & 0.9665                       & 0.9569                   & 0.7998                                & 0.9994                                & 0.9756                  \\ \cline{2-8} 
\multicolumn{1}{|l|}{}                        & Ave time  & 0.0002                           & 0.9310                       & 2.1886                   & 6.1244                                & 32.1675                               & 112.5726                \\ \hline
\hline\multicolumn{1}{|l|}{\multirow{4}{*}{$r=4$}}  & Ave ratio & 0.7488                           & 0.8728                       & 0.8210                   & 0.5254                                & 0.9662                                & 0.9681                  \\
\multicolumn{1}{|l|}{}                        & Min ratio & 0.6293                           & 0.8060                       & 0.7279                   & 0.4655                                & 0.9399                                & 0.9505                  \\
\multicolumn{1}{|l|}{}                        & Max ratio & 0.8831                           & 0.9444                       & 0.9231                   & 0.6027                                & 0.9960                                & 0.9826                  \\ \cline{2-8} 
\multicolumn{1}{|l|}{}                        & Ave time  & 0.0002                           & 0.9060                       & 2.1791                   & 6.1586                                & 32.0104                               & 112.1802                \\ \hline
\hline\multicolumn{1}{|l|}{\multirow{4}{*}{$r=5$}}  & Ave ratio & 0.7503                           & 0.8702                       & 0.8221                   & 0.5269                                & 0.9647                                & 0.9537                  \\
\multicolumn{1}{|l|}{}                        & Min ratio & 0.6115                           & 0.7853                       & 0.7012                   & 0.4831                                & 0.9129                                & 0.9196                  \\
\multicolumn{1}{|l|}{}                        & Max ratio & 0.8690                           & 0.9234                       & 0.9020                   & 0.5943                                & 0.9984                                & 0.9851                  \\ \cline{2-8} 
\multicolumn{1}{|l|}{}                        & Ave time  & 0.0002                           & 0.9099                       & 2.1871                   & 6.1173                                & 31.9779                               & 111.8542                \\ \hline
\hline\multicolumn{1}{|l|}{\multirow{4}{*}{$r=10$}} & Ave ratio & 0.6464                           & 0.8127                       & 0.7457                   & 0.5725                                & 0.9361                                & 0.9865                  \\
\multicolumn{1}{|l|}{}                        & Min ratio & 0.5443                           & 0.7465                       & 0.6478                   & 0.5153                                & 0.9005                                & 0.9670                  \\
\multicolumn{1}{|l|}{}                        & Max ratio & 0.7685                           & 0.8849                       & 0.8454                   & 0.6581                                & 0.9766                                & 0.9973                  \\ \cline{2-8} 
\multicolumn{1}{|l|}{}                        & Ave time  & 0.0002                           & 0.9188                       & 2.1997                   & 6.1855                                & 31.9544                               & 111.2414                \\ \hline
\end{tabular}%
}
\end{table}

\begin{table}[!ht]
\centering
\caption{Approximation bounds of the tensor nuclear $p$-norm by algorithms in~Table~\ref{tab:alg-ratios} when $p=d=3$ and $n=7$.}
\label{tab:n=7}
\resizebox{\textwidth}{!}{%
\begin{tabular}{|ll|c|c|c|c|c|c|}
\hline
\multicolumn{2}{|l|}{Algorithm}                                      & \cite[Prop.~4.3]{chen2020tensor} & Alg.~\ref{alg:matricization} & Alg.~\ref{alg:partition} & Alg.~\ref{alg:alg3} ($\BH_{1}^{n}$) & Alg.~\ref{alg:alg3} ($\BH_{2}^{n}$) & Alg.~\ref{alg:rand-alg} \\ \hline
\multicolumn{2}{|l|}{Theoretical bound}                                      & $\frac{1}{\sqrt[\leftroot{-2}\uproot{2}q]{n^2}}\approx 0.0747$ & \multicolumn{2}{c|}{$\frac{1}{\sqrt[\leftroot{-2}\uproot{2}q]{n}}\approx 0.2733$} &  $\sqrt[\leftroot{-2}\uproot{2}q]{\frac{\ln n}{n}}\approx 0.4259$ &  $\frac{\sqrt[\leftroot{-2}\uproot{2}p]{\ln{n}}}{\sqrt{n}}\approx 0.4719$ &  $\sqrt{\frac{\ln{n}}{n}}\approx 0.5272$ \\ \hline\hline
\multicolumn{1}{|l|}{\multirow{4}{*}{$r=1$}}  & Ave ratio & 1.0000                           & 1.0000                       & 1.0000                   & 0.4142                                & 0.9110                                & 0.8078                  \\
\multicolumn{1}{|l|}{}                        & Min ratio & 1.0000                           & 1.0000                       & 1.0000                   & 0.3625                                & 0.8258                                & 0.7366                  \\
\multicolumn{1}{|l|}{}                        & Max ratio & 1.0000                           & 1.0000                       & 1.0000                   & 0.6394                                & 0.9858                                & 0.8811                  \\ \cline{2-8} 
\multicolumn{1}{|l|}{}                        & Ave time  & 0.0003                           & 1.6829                       & 3.7467                   & 13.0423                               & 481.2554                              & 391.0691                \\ \hline
\hline\multicolumn{1}{|l|}{\multirow{4}{*}{$r=2$}}  & Ave ratio & 0.8256                           & 0.9081                       & 0.8675                   & 0.4469                                & 0.9399                                & 0.8767                  \\
\multicolumn{1}{|l|}{}                        & Min ratio & 0.6663                           & 0.8218                       & 0.7243                   & 0.3906                                & 0.8638                                & 0.8157                  \\
\multicolumn{1}{|l|}{}                        & Max ratio & 0.9917                           & 0.9932                       & 0.9927                   & 0.6131                                & 0.9904                                & 0.9296                  \\ \cline{2-8} 
\multicolumn{1}{|l|}{}                        & Ave time  & 0.0003                           & 1.7272                       & 3.7273                   & 13.1360                               & 488.9772                              & 404.3527                \\ \hline
\hline\multicolumn{1}{|l|}{\multirow{4}{*}{$r=3$}}  & Ave ratio & 0.8152                           & 0.9047                       & 0.8693                   & 0.4253                                & 0.9557                                & 0.8558                  \\
\multicolumn{1}{|l|}{}                        & Min ratio & 0.6836                           & 0.8351                       & 0.7615                   & 0.3883                                & 0.8890                                & 0.7887                  \\
\multicolumn{1}{|l|}{}                        & Max ratio & 0.9402                           & 0.9648                       & 0.9568                   & 0.4606                                & 0.9943                                & 0.9089                  \\ \cline{2-8} 
\multicolumn{1}{|l|}{}                        & Ave time  & 0.0002                           & 1.6995                       & 3.7218                   & 13.0822                               & 486.5939                              & 403.7624                \\ \hline
\hline\multicolumn{1}{|l|}{\multirow{4}{*}{$r=4$}}  & Ave ratio & 0.7351                           & 0.8604                       & 0.8064                   & 0.4581                                & 0.9687                                & 0.8893                  \\
\multicolumn{1}{|l|}{}                        & Min ratio & 0.6369                           & 0.8059                       & 0.7202                   & 0.4131                                & 0.9410                                & 0.7995                  \\
\multicolumn{1}{|l|}{}                        & Max ratio & 0.9449                           & 0.9685                       & 0.9599                   & 0.4967                                & 0.9881                                & 0.9275                  \\ \cline{2-8} 
\multicolumn{1}{|l|}{}                        & Ave time  & 0.0003                           & 1.6525                       & 3.7735                   & 13.2003                               & 488.9220                              & 399.6296                \\ \hline
\hline\multicolumn{1}{|l|}{\multirow{4}{*}{$r=5$}}  & Ave ratio & 0.6917                           & 0.8407                       & 0.7799                   & 0.4671                                & 0.9657                                & 0.9155                  \\
\multicolumn{1}{|l|}{}                        & Min ratio & 0.5961                           & 0.7890                       & 0.6953                   & 0.4130                                & 0.9382                                & 0.8774                  \\
\multicolumn{1}{|l|}{}                        & Max ratio & 0.8302                           & 0.9150                       & 0.8863                   & 0.5001                                & 0.9910                                & 0.9483                  \\ \cline{2-8} 
\multicolumn{1}{|l|}{}                        & Ave time  & 0.0002                           & 1.7021                       & 3.7189                   & 13.0873                               & 487.7752                              & 396.6553                \\ \hline
\hline\multicolumn{1}{|l|}{\multirow{4}{*}{$r=10$}} & Ave ratio & 0.6345                           & 0.8090                       & 0.7397                   & 0.4721                                & 0.9375                                & 0.9364                  \\
\multicolumn{1}{|l|}{}                        & Min ratio & 0.5215                           & 0.7345                       & 0.6412                   & 0.4339                                & 0.9025                                & 0.9078                  \\
\multicolumn{1}{|l|}{}                        & Max ratio & 0.7418                           & 0.8642                       & 0.8208                   & 0.5251                                & 0.9749                                & 0.9647                  \\ \cline{2-8} 
\multicolumn{1}{|l|}{}                        & Ave time  & 0.0002                           & 1.6739                       & 3.9940                   & 13.0368                               & 487.8102                              & 395.5568                \\ \hline
\end{tabular}%
}
\end{table}

\begin{table}[!ht]
\centering
\caption{Approximation bounds of the tensor nuclear $p$-norm by algorithms in~Table~\ref{tab:alg-ratios} when $p=d=3$ and $n=10$.}
\label{tab:n=10}
\resizebox{\textwidth}{!}{%
\begin{tabular}{|ll|c|c|c|c|c|c|}
\hline
\multicolumn{2}{|l|}{Algorithm}                               & \cite[Prop.~4.3]{chen2020tensor} & Alg.~\ref{alg:matricization} & Alg.~\ref{alg:partition} & Alg.~\ref{alg:alg3} ($\BH_{1}^{n}$) & Alg.~\ref{alg:alg3} ($\BH_{2}^{n}$) & Alg.~\ref{alg:rand-alg} \\ \hline
\multicolumn{2}{|l|}{Theoretical bound}                                      & $\frac{1}{\sqrt[\leftroot{-2}\uproot{2}q]{n^2}}\approx 0.0464$ & \multicolumn{2}{c|}{$\frac{1}{\sqrt[\leftroot{-2}\uproot{2}q]{n}}\approx 0.2154$} &  $\sqrt[\leftroot{-2}\uproot{2}q]{\frac{\ln n}{n}}\approx 0.3757$ &  $\frac{\sqrt[\leftroot{-2}\uproot{2}p]{\ln{n}}}{\sqrt{n}}\approx 0.4176$ &  $\sqrt{\frac{\ln{n}}{n}}\approx 0.4799$ \\ \hline\hline
\multicolumn{1}{|l|}{\multirow{4}{*}{$r=1$}}  & Ave ratio & 1.0000                           & 1.0000                       & 1.0000                   & 0.4083                                & 0.7890                                & 0.7219                  \\
\multicolumn{1}{|l|}{}                        & Min ratio & 1.0000                           & 1.0000                       & 1.0000                   & 0.3671                                & 0.7093                                & 0.6768                  \\
\multicolumn{1}{|l|}{}                        & Max ratio & 1.0000                           & 1.0000                       & 1.0000                   & 0.5084                                & 0.8988                                & 0.7928                  \\ \cline{2-8} 
\multicolumn{1}{|l|}{}                        & Ave time  & 0.0157                           & 4.6467                       & 6.8902                   & 147.8669                              & 303.4447                              & 1980.3324               \\ \hline
\hline\multicolumn{1}{|l|}{\multirow{4}{*}{$r=2$}}  & Ave ratio & 0.8634                           & 0.9270                       & 0.8990                   & 0.4356                                & 0.8243                                & 0.7790                  \\
\multicolumn{1}{|l|}{}                        & Min ratio & 0.6755                           & 0.8222                       & 0.7210                   & 0.3801                                & 0.7536                                & 0.7243                  \\
\multicolumn{1}{|l|}{}                        & Max ratio & 0.9932                           & 0.9960                       & 0.9949                   & 0.5150                                & 0.8844                                & 0.8215                  \\ \cline{2-8} 
\multicolumn{1}{|l|}{}                        & Ave time  & 0.0010                           & 4.3290                       & 6.7973                   & 151.8938                              & 389.5651                              & 2519.7346               \\ \hline
\hline\multicolumn{1}{|l|}{\multirow{4}{*}{$r=3$}}  & Ave ratio & 0.7834                           & 0.8887                       & 0.8471                   & 0.4283                                & 0.8716                                & 0.7747                  \\
\multicolumn{1}{|l|}{}                        & Min ratio & 0.6141                           & 0.8050                       & 0.7120                   & 0.3930                                & 0.7815                                & 0.7373                  \\
\multicolumn{1}{|l|}{}                        & Max ratio & 0.9150                           & 0.9584                       & 0.9445                   & 0.4634                                & 0.9599                                & 0.8174                  \\ \cline{2-8} 
\multicolumn{1}{|l|}{}                        & Ave time  & 0.0030                           & 4.8250                       & 6.8015                   & 152.7054                              & 358.2596                              & 2465.2732               \\ \hline
\hline\multicolumn{1}{|l|}{\multirow{4}{*}{$r=4$}}  & Ave ratio & 0.7099                           & 0.8510                       & 0.7948                   & 0.4374                                & 0.9039                                & 0.8069                  \\
\multicolumn{1}{|l|}{}                        & Min ratio & 0.5958                           & 0.7867                       & 0.7033                   & 0.3976                                & 0.8494                                & 0.7488                  \\
\multicolumn{1}{|l|}{}                        & Max ratio & 0.8169                           & 0.9043                       & 0.8719                   & 0.5030                                & 0.9694                                & 0.8529                  \\ \cline{2-8} 
\multicolumn{1}{|l|}{}                        & Ave time  & 0.0004                           & 4.5871                       & 6.7897                   & 152.3298                              & 328.8995                              & 2303.7053               \\ \hline
\hline\multicolumn{1}{|l|}{\multirow{4}{*}{$r=5$}}  & Ave ratio & 0.6848                           & 0.8389                       & 0.7778                   & 0.4497                                & 0.9134                                & 0.7809                  \\
\multicolumn{1}{|l|}{}                        & Min ratio & 0.6128                           & 0.8046                       & 0.7197                   & 0.4182                                & 0.8347                                & 0.7220                  \\
\multicolumn{1}{|l|}{}                        & Max ratio & 0.8168                           & 0.9041                       & 0.8735                   & 0.4885                                & 0.9924                                & 0.8105                  \\ \cline{2-8} 
\multicolumn{1}{|l|}{}                        & Ave time  & 0.0004                           & 4.6683                       & 6.7397                   & 150.6115                              & 327.2175                              & 2232.9938               \\ \hline
\hline\multicolumn{1}{|l|}{\multirow{4}{*}{$r=10$}} & Ave ratio & 0.6159                           & 0.8039                       & 0.7306                   & 0.4590                                & 0.9748                                & 0.8092                  \\
\multicolumn{1}{|l|}{}                        & Min ratio & 0.5014                           & 0.7339                       & 0.6253                   & 0.4266                                & 0.9308                                & 0.7809                  \\
\multicolumn{1}{|l|}{}                        & Max ratio & 0.7292                           & 0.8618                       & 0.8156                   & 0.4959                                & 0.9993                                & 0.8435                  \\ \cline{2-8} 
\multicolumn{1}{|l|}{}                        & Ave time  & 0.0004                           & 4.4623                       & 6.7456                   & 148.7543                              & 316.9498                              & 2166.7352               \\ \hline
\end{tabular}%
}
\end{table}
}

\section{Concluding remarks}\label{sec:conclusion}

We study approximation algorithms of the tensor nuclear $p$-norm with an aim to establish the approximation bound matching the best one of its dual norm, the tensor spectral $p$-norm. Driven by the application of sphere covering to approximate tensor spectral and nuclear norms, we propose several types of hitting sets that approximately represent $\ell_p$-sphere with varying cardinalities and covering ratios, providing an independent toolbox for decision making on $\ell_p$-spheres. Using the idea in robust optimization and second-order cone programming, we obtain the first polynomial-time algorithm with an $\Omega(1)$-approximation bound for the computation of the matrix nuclear $p$-norm when $p\in\BQ\cap(2,\infty)$, paving a way for applications in modeling with the matrix nuclear $p$-norm. These two new results enable us to propose various polynomial-time algorithms for the computation of the tensor nuclear $p$-norm with the best approximation bound being the same to the best one of the tensor spectral $p$-norm, no matter for deterministic algorithms or randomized ones. In the meantime, our study opens up a \li{challenging} problem on how to explicitly construct a deterministic hitting set of $\BS^n_p$ with covering ratio $\Omega(\sqrt{\ln{n}/{n}})$ when $p\in(2,\infty)$.

The idea of covering and its applications may lead to some interesting future work. One important case that was not discussed in the paper is $p=\infty$ which corresponds to the so-called $\infty\mapsto1$ norm~\cite{he2013approximation} of a tensor and directly relates to binary constrained polynomial optimization. It is still not known how to construct deterministic hitting set to cover $\{-1,1\}^n$ with polynomial cardinality and best possible covering ratio. A positive answer could result a wide range of applications such as graph theory, neural networks, error-correcting codes; see~\cite[Section~1]{he2013approximation} and the references therein. 

A generalization of the spectral and nuclear $p$-norms is the spectral and nuclear $\bp$-norms where $\bp\in\R^d$ and $\ell_p$-spheres are replaced by $\ell_{p_k}$-spheres for different $p_k$'s; see~\cite{lim2005singular,li2020norm}. Although the constructed hitting sets of $\BS_p^n$ can still be used, the approximation of the matrix $(p_1,p_2)$-spectral norms remains mostly unknown except for some special $(p_1,p_2)$, the same to the matrix $(p_1,p_2)$-nuclear norms. If they can be approximated within a bound of $\Omega(1)$, then a deterministic $\Omega(\prod_{k=1}^{d-2} \sqrt[p_k]{\ln{n_k}}/\sqrt{n_k})$-approximation bound and a randomized $\Omega(\prod_{k=1}^{d-2} \sqrt{\ln{n_k}/n_k})$-approximation bound can be achieved.

Covering a subset of a sphere instead of the whole one may be of particular interest. One immediate case that has wider applications is the nonnegative sphere, $\BS^n_2\cap\R^n_+$, which only occupies $1/2^n$ of the whole sphere. The best possible hitting ratio with polynomial cardinality remains unknown and should be better than $\Omega(\sqrt{\ln{n}/{n}})$. Some probabilistic evidence hints that the best possible hitting ratio might be even $\Omega(1)$, somewhat surprising. A theoretical justification would have a profound influence to optimization problems such as copositive programming and nonnegative matrix factorization.

\bibliographystyle{abbrv}
\bibliography{references}{}

\end{document}